\documentclass[a4paper,12pt,reqno]{amsart}

\usepackage[T1]{fontenc}
\usepackage[utf8]{inputenc}
\usepackage{lmodern}
\usepackage[english]{babel}

\usepackage[%
	left=2.5cm,       
	right=2.5cm,      
	top=3.5cm,        
	bottom=3.5cm,     
	heightrounded,    
	bindingoffset=0mm 
]{geometry}

\usepackage{amsmath}
\usepackage{amssymb}
\usepackage{mathtools}
\usepackage{mathrsfs}

\numberwithin{equation}{section}

\usepackage{hyperref} 
\usepackage[noabbrev,capitalize]{cleveref}

\usepackage{bookmark}

\theoremstyle{plain}
	\newtheorem{theorem}{Theorem}[section]
	\newtheorem{lemma}[theorem]{Lemma}
	\newtheorem{proposition}[theorem]{Proposition}
	\newtheorem{corollary}[theorem]{Corollary}

\theoremstyle{definition}
	\newtheorem{definition}[theorem]{Definition}
	\newtheorem{example}[theorem]{Example}
	\newtheorem{remark}[theorem]{Remark}
	\newtheorem{open.problem}[theorem]{Open Problem}

\newcommand{\N}{\mathbb{N}}
\newcommand{\R}{\mathbb{R}}

\newcommand*{\transp}[2][-3mu]{\ensuremath{\mskip1mu\prescript{\smash{\mathrm t\mkern#1}}{}{\mathstrut#2}}}

\usepackage{calrsfs}	

\newcommand{\eps}{\varepsilon}

\usepackage[initials]{amsrefs}

\usepackage{tikz}
\usepackage{graphicx}
\usepackage{xcolor}
\usepackage[font=small]{caption}

\newcommand{\closure}[2][3]{%
  {}\mkern#1mu\overline{\mkern-#1mu#2}}
  
\usepackage{enumerate}

\usepackage{url}

\newcommand{\de}{\partial}

\newcommand{\weakto}{\rightharpoonup}

\renewcommand{\phi}{\varphi}
\renewcommand{\rho}{\varrho}
\renewcommand{\theta}{\vartheta}

\DeclareMathOperator{\dist}{dist}
\DeclareMathOperator{\diam}{diam}

\DeclareMathOperator{\supp}{supp}
\DeclareMathOperator{\sgn}{sgn}

\DeclareMathOperator{\diverg}{div}
\DeclareMathOperator{\tang}{Tan}
\DeclareMathOperator{\Lip}{Lip}
\DeclareMathOperator{\loc}{loc}

\DeclarePairedDelimiter{\set}{\{}{\}}
\DeclarePairedDelimiter{\abs}{|}{|}

\mathchardef\ordinarycolon\mathcode`\:
\mathcode`\:=\string"8000
\begingroup \catcode`\:=\active
  \gdef:{\mathrel{\mathop\ordinarycolon}}
\endgroup

\newcommand{\Haus}[1]{\mathscr{H}^{#1}} 
\newcommand{\Leb}[1]{\mathscr{L}^{#1}} 

\renewcommand{\div}{\mathrm{div}} 
\newcommand{\redb}{\mathscr{F}} 
\renewcommand{\dim}{{\rm dim}} 
\newcommand{\M}{\mathcal{M}} 

\newcommand{\res}{\mathop{\hbox{\vrule height 7pt width .5pt depth 0pt
\vrule height .5pt width 6pt depth 0pt}}\nolimits}

\allowdisplaybreaks

\begin{document}

\title[A distributional approach to fractional Sobolev spaces and variation]{A distributional approach to fractional Sobolev spaces and fractional variation: existence of blow-up}

\author[G. E. Comi]{Giovanni E. Comi}
\address{Scuola Normale Superiore, Piazza Cavalieri 7, 56126 Pisa, Italy}
\email{giovanni.comi@sns.it}

\author[G. Stefani]{Giorgio Stefani}
\address{Scuola Normale Superiore, Piazza Cavalieri 7, 56126 Pisa, Italy}
\email{giorgio.stefani@sns.it}

\date{\today}



\thanks{\textit{Acknowledgements}. The authors thank Luigi Ambrosio for valuable comments and suggestions during the realisation of this work. The authors also thank Quoc-Hung Nguyen for having shared with them many useful insights on Bessel potential spaces. This research was partially supported by the PRIN2015 MIUR Project ``Calcolo delle Variazioni''.}

\begin{abstract}
We introduce the new space $BV^{\alpha}(\R^n)$ of functions with bounded fractional variation in~$\R^n$ of order $\alpha \in (0, 1)$ via a new distributional approach exploiting suitable notions of fractional gradient and fractional divergence already existing in the literature. In analogy with the classical $BV$ theory, we give a new notion of set~$E$ of (locally) finite fractional Caccioppoli $\alpha$-perimeter and we define its fractional reduced boundary~$\redb^{\alpha} E$. We are able to show that $W^{\alpha,1}(\R^n)\subset BV^\alpha(\R^n)$ continuously and, similarly, that sets with (locally) finite standard fractional $\alpha$-perimeter have (locally) finite fractional Caccioppoli $\alpha$-perimeter, so that our theory provides a natural extension of the known fractional framework. Our main result partially extends De Giorgi's Blow-up Theorem to sets of locally finite fractional Caccioppoli $\alpha$-perimeter, proving existence of blow-ups and giving a first characterisation of these (possibly non-unique) limit sets.
\end{abstract}

\maketitle

\tableofcontents

\section{Introduction}

\subsection{Fractional Sobolev spaces and the quest for a fractional gradient}

In the last decades, fractional Sobolev spaces have been given an increasing attention, see~\cite{DiNPV12}*{Section~1} for a detailed list of references in many directions. If $p\in[1,+\infty)$ and $\alpha\in(0,1)$, the fractional Sobolev space $W^{\alpha,p}(\R^n)$ is the space
\begin{equation*}
W^{\alpha,p}(\R^n):=\set*{u\in L^p(\R^n) : [u]_{W^{\alpha,p}(\R^n)}:=\left(\int_{\R^{n}} \int_{\R^{n}} \frac{|u(x)-u(y)|^p}{|x-y|^{n+p\alpha}}\,dx\,dy\right)^{\frac{1}{p}}<+\infty}
\end{equation*}
endowed with the natural norm
\begin{equation*}
\|u\|_{W^{\alpha,p}(\R^n)}:=\|u\|_{L^p(\R^n)}+[u]_{W^{\alpha,p}(\R^n)}.
\end{equation*}

Differently from the standard Sobolev space $W^{1,p}(\R^n)$, the space $W^{\alpha,p}(\R^n)$ does not have an evident distributional nature, in the sense that the seminorm $[\,\cdot\,]_{W^{\alpha,p}(\R^n)}$ does not seem to be the $L^p$-norm of some kind of weakly-defined gradient of fractional order. 

Recently, the search for a good notion of differential operator in this fractional setting has led several authors to consider the following \emph{fractional gradient} 
\begin{equation}\label{intro_eq:def_D_alpha}
\nabla^\alpha u(x):=\mu_{n,\alpha}\int_{\R^n}\frac{(y-x)(u(y)-u(x))}{|y-x|^{n+\alpha+1}}\,dy, 
\end{equation}
where $\mu_{n,\alpha}$ is a multiplicative normalising constant controlling the behaviour of~$\nabla^{\alpha}$ as $\alpha \to 1^{-}$. For a detailed account on the existing literature on this operator, see~\cite{SS18}*{Section~1}. Here we only refer to~\cites{SSS15,SSVanS17,SSS18,SS15,SS18,S18} for the articles tightly connected to the present work. According to~\cite{SS18}*{Section~1}, it is interesting to notice that~\cite{H59} seems to be the earliest reference for the operator defined in~\eqref{intro_eq:def_D_alpha}. 

From its very definition, it is not difficult to see that the fractional gradient~$\nabla^\alpha$ is well defined in $L^1(\R^n; \R^{n})$ for functions in $W^{\alpha,1}(\R^n)$, since we have the simple estimate
\begin{equation}\label{intro_eq:frac_Sobolev_enough}
\int_{\R^n}|\nabla^\alpha u|\, dx\le \mu_{n,\alpha}\,[u]_{W^{\alpha,1}(\R^n)}.
\end{equation}
In analogy with the standard Sobolev space $W^{1,p}(\R^n)$, this observation leads to consider the space
\begin{equation*}\label{intro_eq:def_distrib_frac_Sobolev_p}
S^{\alpha,p}_0(\R^n):=\closure[-1]{C^\infty_c(\R^n)}^{\|\cdot\|_{S^{\alpha,p}(\R^n)}},
\end{equation*}
where  
\begin{equation*}\label{intro_eq:def_distrib_frac_Sobolev_norm_p}
\|u\|_{S^{\alpha,p}(\R^n)}:= \|u\|_{L^p(\R^n)}+\|\nabla^\alpha u\|_{L^p(\R^n;\R^n)}
\end{equation*}
for all $u\in C^\infty_c(\R^n)$. This is essentially the line followed in~\cite{SS15}, where the space $S^{\alpha,p}_0(\R^n)$ has been introduced (with a different, but equivalent, norm). By~\cite{SS15}*{Theorem~2.2}, it is known that 
\begin{equation}\label{intro_eq:S_space_rel_1}
S^{\alpha+\eps,p}_0(\R^n)\subset W^{\alpha, p}(\R^{n})\subset S^{\alpha-\eps,p}_0(\R^n)
\end{equation}
with continuous embeddings for all $\alpha\in(0,1)$, $p\in(1,+\infty)$ and $0<\eps<\min\set*{\alpha,1-\alpha}$. In the particular case $p=2$, by~\cite{SS15}*{Theorem 2.2} we actually have that 
\begin{equation}\label{intro_eq:S_space_rel_2}
S^{\alpha, 2}_0(\R^{n})=W^{\alpha, 2}(\R^{n})
\end{equation}
for all $\alpha \in (0, 1)$. In addition, as observed in~\cite{S70}*{Chapter~V, Section~5.3}, we have
\begin{equation}\label{intro_eq:S_space_rel_3}
W^{\alpha,p}(\R^n)\subset S^{\alpha,p}_0(\R^n)	
\end{equation}
with continuous embedding for all $\alpha\in(0,1)$ and $p\in(1,2]$. For further properties of the space $S^{\alpha,p}_0(\R^n)$, we refer to \cref{subsec:space_S_alpha_p} below. 

The inclusions~\eqref{intro_eq:S_space_rel_1} and~\eqref{intro_eq:S_space_rel_3}, and the identification~\eqref{intro_eq:S_space_rel_2} may suggest that the spaces $S^{\alpha,p}_0(\R^n)$ can be considered as an interesting distributional-type alternative of the usual fractional Sobolev spaces $W^{\alpha,p}(\R^n)$ and thus as a natural setting for the development of a general theory for solutions to PDEs involving the fractional gradient in~\eqref{intro_eq:def_D_alpha}, proceeding similarly as in the classical Sobolev framework. This is the point of view pursued in~\cites{SS15,SS18}.

Another important aspect of the fractional gradient in~\eqref{intro_eq:def_D_alpha} is that it satisfies three natural `qualitative' requirements as a fractional operator: \emph{invariance} under translations and rotations, \emph{homogeneity} of order~$\alpha$ under dilations and some \emph{continuity} properties in an appropriate functional space, e.g.\ Schwartz's space~$\mathcal{S}(\R^n)$. A fundamental result of~\cite{S18} is that these three requirements actually characterise the fractional gradient in~\eqref{intro_eq:def_D_alpha} (up to multiplicative constants), see~\cite{S18}*{Theorem~2.2}. This shows that the definition in~\eqref{intro_eq:def_D_alpha} is well posed not only from a \emph{mathematical} point of view, but also from a \emph{physical} point of view. 

Besides, the same characterisation holds for the following \emph{fractional divergence}
\begin{equation}\label{intro_eq:def_div_alpha}
\div^\alpha\phi(x):= \mu_{n,\alpha}\int_{\R^{n}} \frac{(y - x) \cdot (\phi(y) - \phi(x)) }{|y - x|^{n + \alpha + 1}} \, dy,
\end{equation} 
see~\cite{S18}*{Theorem~2.4}. Moreover, as it is observed in~\cite{S18}*{Section~6}, the operators $\nabla^\alpha$ and $\div^\alpha$ are \emph{dual}, in the sense that
\begin{equation}\label{intro_eq:duality}
\int_{\R^{n}} u\,\div^{\alpha} \varphi \, dx = - \int_{\R^{n}} \varphi \cdot \nabla^{\alpha} u \, dx
\end{equation}
for all $u\in C^\infty_c(\R^n)$ and $\phi\in C^\infty_c(\R^n;\R^n)$. The fractional integration by parts formula in~\eqref{intro_eq:duality} can be thus taken as the starting point for the development of a general theory of fractional differential operators on the space of Schwartz's distributions. This is the direction of research pursued in~\cite{S18}.

\subsection{De Giorgi's distributional approach to perimeter}

In the classical framework, the Sobolev space $W^{1,1}(\R^n)$ is naturally continuously embedded in $BV(\R^n)$, the space of functions with \emph{bounded variation}, i.e.\
\begin{equation*}
BV(\R^n):=\set*{u\in L^1(\R^n) : |D u|(\R^n)<+\infty},
\end{equation*}
endowed with the norm
\begin{equation*}
\|u\|_{BV(\R^n)}=\|u\|_{L^1(\R^n)}+|D u|(\R^n),
\end{equation*}
where 
\begin{equation*}
|D u|(\R^n)=\sup\set*{\int_{\R^n}u\,\div\phi\,dx : \phi\in C^\infty_c(\R^n;\R^n),\ \|\phi\|_{L^\infty(\R^n;\R^n)}\le1}
\end{equation*}
is the \emph{total variation} of the function $u\in BV(\R^n)$. Thanks to Riesz's Representation Theorem, one can see that a function $u\in L^1(\R^n)$ belongs to $BV(\R^n)$ if and only if there exists a finite vector valued Radon measure $Du\in\mathcal{M}(\R^n;\R^n)$ such that
\begin{equation*}
\int_{\R^n} u\,\div\phi\,dx=-\int_{\R^n}\phi\cdot dDu
\end{equation*}
for all $\phi\in C^\infty_c(\R^n;\R^n)$.

Functions of bounded variation have revealed to be the perfect tool for the development of a deep geometric analysis of sets with finite perimeter, starting directly from the seminal and profound works of R.~Caccioppoli and E.~De~Giorgi. For a modern exposition of this vast subject and a detailed list of references, see~\cites{AFP00,EG15,M12}. 

A measurable set $E\subset\R^n$ has finite Caccioppoli perimeter if 
\begin{equation}\label{intro_eq:def_perimeter}
P(E):=|D\chi_E|(\R^n)<+\infty.
\end{equation}
The perimeter functional in~\eqref{intro_eq:def_perimeter} coincides with the classical surface measure when~$E$ has a sufficiently smooth (topological) boundary and, precisely, one can prove that 
\begin{equation}\label{intro_eq:perimeter_is_Hausdorff}
P(E)=\Haus{n-1}(\de E)
\end{equation}
for all sets~$E$ with Lipschitz boundary, where $\Haus{s}$ denotes the $s$-dimensional Hausdorff measure for all $s\ge0$.  

One of the finest De Giorgi's intuitions is that, for a finite perimeter set~$E$ with non-smooth boundary, the right `boundary object' to keep the validity of~\eqref{intro_eq:perimeter_is_Hausdorff} is a special subset of the topological boundary, the so-called \emph{reduced boundary}~$\redb E$. With this notion in hand, a measurable set $E\subset\R^n$ has finite Caccioppoli perimeter if (and only if) $\Haus{n-1}(\redb E)<+\infty$, in which case we have
\begin{equation}\label{intro_eq:perimeter_is_Hausdorff_redb}
P(E)=\Haus{n-1}(\redb E).
\end{equation} 

Besides the validity of~\eqref{intro_eq:perimeter_is_Hausdorff_redb}, an essential feature of De~Giorgi's reduced boundary is the following \emph{blow-up property}: if $x\in\redb E$, then
\begin{equation}\label{intro_eq:blow-up}
\chi_{\frac{E-x}{r}}\to\chi_{H_{\nu_{E}(x)}}
\quad
\text{in } L^1_{\loc}(\R^n)
\end{equation}
as $r\to0$, where 
\begin{equation*}
H_{\nu_{E}(x)}:=\set*{y\in\R^n : y \cdot\nu_E(x) \ge 0},
\qquad
\nu_E(x)=\lim_{r\to0}\frac{D\chi_E(B_r(x))}{|D\chi_E|(B_r(x))}.
\end{equation*}
The function $\nu_E\colon\redb E\to\mathbb{S}^{n-1}$ denotes the so-called \emph{measure theoretic inner unit normal} of~$E$ and coincides with the usual inner unit normal of~$E$ when the boundary of~$E$ is sufficiently smooth. In other words, the blow-up property in~\eqref{intro_eq:blow-up} shows that, in a neighbourhood of a point $x\in\redb E$, the finite perimeter set~$E$ is infinitesimally close to $x + H_{\nu_{E}(x)} = \{y \in \mathbb{R}^{n} : (y - x) \cdot \nu_{E}(x) \ge 0 \}$.

\subsection{Fractional variation and perimeter: a new distributional approach}

In the fractional framework, an analogue of the space $BV(\R^n)$ is completely missing, since no distributional definition of the space $W^{\alpha,1}(\R^n)$ is available. 

Nevertheless, a theory for sets with finite fractional perimeter has been developed in recent years, with a strong interest on minimal fractional surfaces. We refer to~\cite{CF17}*{Section~7} for a detailed exposition of the most recent results in this direction. 

A measurable set $E\subset\R^n$ has finite fractional perimeter if 
\begin{equation}\label{intro_eq:def_frac_perimeter}
P_\alpha(E):= [\chi_E]_{W^{\alpha,1}(\R^n)}
= 2 \int_{\R^n \setminus E}\int_{E}\frac{1}{|x-y|^{n+\alpha}}\,dx\,dy<+\infty.
\end{equation}
The fractional perimeter functional in~\eqref{intro_eq:def_frac_perimeter} has a strong \emph{non-local nature} in the sense that its value depends also on points which are very far from the boundary of the set~$E$. For this reason, it is not clear if such a perimeter measure may be linked with some kind of fractional analogue of De~Giorgi's reduced boundary (which, a posteriori, cannot be expected to be a special subset of the topological boundary of~$E$).

In this paper, we want to combine the functional approach of~\cites{SS15,SS18} with the distributional point of view of~\cite{S18} to develop a satisfactory theory extending De~Giorgi's approach to variation and perimeter in the fractional setting.

The natural starting point is the duality relation~\eqref{intro_eq:duality}, which motivates the definition of the space
\begin{equation}\label{intro_eq:BV_alpha}
BV^\alpha(\R^n):=\set*{u\in L^1(\R^n) : |D^\alpha u|(\R^n)<+\infty},
\end{equation}
where
\begin{equation}\label{intro_eq:def_frac_variation}
|D^\alpha u|(\R^n)=\sup\set*{\int_{\R^n}u\,\div^\alpha\phi\,dx : \phi\in C^\infty_c(\R^n;\R^n),\ \|\phi\|_{L^\infty(\R^n;\R^n)}\le1}
\end{equation}
stands for the \emph{fractional variation} of the function $u\in BV^\alpha(\R^n)$. Note that the fractional variation in~\eqref{intro_eq:def_frac_variation} is well defined, since one can show that $\div^\alpha\phi\in L^\infty(\R^n)$ for all $\phi\in C^\infty_c(\R^n;\R^n)$ (see \cref{prop:frac_div_repr}).

A different approach to fractional variation was developed in~\cite{Z18}. We do not know if the fractional variation defined in~\eqref{intro_eq:def_frac_variation} is linked to the one introduced in~\cite{Z18} and it would be very interesting to establish a connection between the two.

With definition~\eqref{intro_eq:BV_alpha}, we are able to show that 
\begin{equation*}
W^{\alpha,1}(\R^n)\subset BV^\alpha(\R^n)
\end{equation*}
with continuous embedding, in perfect analogy with the classical framework.

Thus, emulating the classical definition in~\eqref{intro_eq:def_perimeter}, it is very natural to define the fractional analogue of the Caccioppoli perimeter using the total variation in~\eqref{intro_eq:def_frac_variation}. Note that this definition is well posed, since $\div^\alpha\phi\in L^1(\R^n)$ for all $\phi\in W^{\alpha, 1}(\R^{n}; \R^{n})$ arguing similarly as in~\eqref{intro_eq:frac_Sobolev_enough}. With this notion, we are able to show that 
\begin{equation} \label{eq:D_alpha_chi_E_P_alpha_intro}
|D^\alpha\chi_E|(\R^n)\le \mu_{n,\alpha} P_\alpha(E)
\end{equation}
for all measurable sets~$E$ with finite fractional perimeter, so that our approach naturally incorporates the current notion of fractional perimeter.

Following the classical framework, the main results concerning the space $BV^\alpha(\R^n)$ we are able to prove are the following:
\begin{itemize}
\item $BV^\alpha(\R^n)$ is a Banach space and its norm is lower semicontinuous with respect to $L^1$-convergence;
\item the inclusion $W^{\alpha,1}(\R^n)\subset BV^\alpha(\R^n)$ is continuous and strict;
\item the sets $C^\infty(\R^n)\cap BV^\alpha(\R^n)$ and $C^\infty_c(\R^n)$ are dense in $BV^\alpha(\R^n)$ with respect to the distance
\begin{equation*}
d(u,v):=\|u-v\|_{L^1(\R^n)} + \abs*{|D^\alpha u|(\R^n)-|D^\alpha v|(\R^n)};
\end{equation*} 
\item a fractional analogue of Gagliardo--Nirenberg--Sobolev inequality holds, i.e.\ for all $n\ge2$ the embedding 
\begin{equation*}
BV^\alpha(\R^n)\subset L^{\frac{n}{n-\alpha}}(\R^n)
\end{equation*}
is continuous;
\item the natural analogue of the coarea formula does not hold in $BV^{\alpha}(\R^{n})$, since there are functions $u \in BV^{\alpha}(\R^{n})$ such that $\int_{\R}|D^\alpha\chi_{\set{u > t}}|(\R^n)\,dt=+\infty$;
\item any uniformly bounded sequence in $BV^\alpha(\R^n)$ admits limit points in $L^1(\R^n)$ with respect the $L^1_{\loc}$-convergence.
\end{itemize}
Concerning sets with finite fractional Caccioppoli $\alpha$-perimeter, the main results we are able to prove are the following:
\begin{itemize}
\item fractional Caccioppoli $\alpha$-perimeter is lower semicontinuous with respect to $L^1_{\loc}$-convergence;
\item a fractional isoperimetric inequality holds, i.e.\
\begin{equation*}
|E|^{\frac{n-\alpha}{n}} \le c_{n, \alpha} |D^{\alpha} \chi_E|(\R^{n});
\end{equation*} 
\item any sequence of sets with uniformly bounded fractional Caccioppoli $\alpha$-perimeter confined in a fixed ball admits limit points with respect to $L^1$-convergence;
\item a natural analogue of De Giorgi's reduced boundary, that we call \emph{fractional reduced boundary} $\redb^\alpha E$, is well posed for any set~$E$ with finite fractional Caccioppoli $\alpha$-perimeter;
\item if~$E$ has finite fractional Caccioppoli $\alpha$-perimeter, then its fractional Caccioppoli $\alpha$-perimeter measure satisfies $|D^\alpha\chi_E|\ll\Haus{n-\alpha}\res\redb^\alpha E$; 
\item if~$E$ has finite fractional Caccioppoli $\alpha$-perimeter and $x\in\redb^\alpha E$, then the family $(\frac{E-x}{r})_{r>0}$ admits limit points in the $L^1_{\loc}$-topology and any such limit point satisfies a rigidity condition. 
\end{itemize}

Some of the results listed above are proved similarly as in the classical framework. Since we believe that our approach might be interesting also to researchers that may be not familiar with the theory of functions of bounded variation, we tried to keep the exposition the most self-contained as possible.

\subsection{Future developments} 

Thanks to this new approach, a large variety of classical results might be extended to the context of functions with bounded fractional variation. Here we just list some of the most intriguing open problems:
\begin{itemize}
\item investigate the case of equality in~\eqref{eq:D_alpha_chi_E_P_alpha_intro};
\item achieve a better characterisation of the blow-ups (possibly, their uniqueness);
\item prove a Structure Theorem for $\redb^{\alpha} E$ in the spirit of De Giorgi's Theorem;
\item study the fractional isoperimetric inequality and its stability, possibly also for its relative version;
\item develop a \emph{calibration theory} for sets of finite fractional Caccioppoli $\alpha$-perimeter as a useful tool for the study of fractional minimal surfaces;
\item consider the asymptotics as $\alpha \to 1^{-}$ and investigate the $\Gamma$-convergence to the classical perimeter;
\item extend the Gauss--Green and integration by parts formulas to sets of finite fractional Caccioppoli $\alpha$-perimeter;
\item give a good definition of $BV^\alpha$ functions on a general open set.
\end{itemize}  
Some of these open problems will be the subject of a forthcoming paper, see~\cite{CS18}.

\subsection{Organisation of the paper}

The paper is organised as follows. In \cref{sec:Silhavy_calculus}, we introduce and study the fractional gradient and divergence, proving generalised Leibniz's rules and representation formulas. In \cref{sec:frac_BV_func}, we define the space $BV^{\alpha}(\R^{n})$ and we study approximation by smooth functions, embeddings and compactness exploiting a fractional version of the Fundamental Theorem of Calculus. In \cref{sec:frac_Caccioppoli_sets}, we define sets of (locally) finite fractional Caccioppoli $\alpha$-perimeter, we prove some compactness results and we introduce the notion of fractional reduced boundary. Finally, in \cref{sec:blow-ups}, we prove existence of blow-ups of sets with locally finite fractional Caccioppoli $\alpha$-perimeter.

\section{\v{S}ilhav\'{y}'s fractional calculus}
\label{sec:Silhavy_calculus}

\subsection{General notation}

We start with a brief description of the main notation used in this paper.

Given an open set $\Omega$, we say that a set $E$ is compactly contained in $\Omega$, and we write $E \Subset \Omega$, if the $\overline{E}$ is compact and contained in $\Omega$.
We denote by $\Leb{n}$ and $\Haus{\alpha}$ the Lebesgue measure and the $\alpha$-dimensional Hausdorff measure on $\R^n$ respectively, with $\alpha \ge 0$. Unless otherwise stated, a measurable set is a $\Leb{n}$-measurable set. We also use the notation $|E|=\Leb{n}(E)$. All functions we consider in this paper are Lebesgue measurable, unless otherwise stated. We denote by $B_r(x)$ the standard open Euclidean ball with center $x\in\R^n$ and radius $r>0$. We let $B_r=B_r(0)$. Recall that $\omega_{n} := |B_1|=\pi^{\frac{n}{2}}/\Gamma\left(\frac{n+2}{2}\right)$ and $\Haus{n-1}(\partial B_{1}) = n \omega_n$, where $\Gamma$ is Euler's \emph{Gamma function}, see~\cite{A64}. 

For $k \in \N_{0} \cup \set{+ \infty}$ and $m \in \N$ we denote by $C^{k}_{c}(\Omega ; \R^{m})$ and, respectively, by $\Lip_c(\Omega; \R^{m})$, the space of $C^{k}$-regular, respectively, Lipschitz regular, $m$-vector valued functions defined on~$\R^n$ with compact support in~$\Omega$.

For any exponent $p\in[1,+\infty]$, we denote by 
\begin{equation*}
L^p(\Omega;\R^m):=\set*{u\colon\Omega\to\R^m : \|u\|_{L^p(\Omega;\,\R^m)}<+\infty}
\end{equation*}
the space of $m$-vector valued Lebesgue $p$-integrable functions on~$\Omega$. We denote by 
\begin{equation*}
W^{1,p}(\Omega;\R^m):=\set*{u\in L^p(\Omega;\,\R^m) : [u]_{W^{1,p}(\Omega;\,\R^m)}:=\|\nabla u\|_{L^p(\Omega;\,\R^{n+m})}<+\infty}
\end{equation*}
the space of $m$-vector valued Sobolev functions on~$\Omega$, see for instance~\cite{L09}*{Chapter~10} for its precise definition and main properties. We also let
\begin{equation*}
w^{1,p}(\Omega;\R^m):=\set*{u\in L^1_{\loc}(\Omega;\R^m) : [u]_{W^{1,p}(\Omega;\,\R^m)}<+\infty}.
\end{equation*}
We denote by 
\begin{equation*}
BV(\Omega;\R^m):=\set*{u\in L^1(\Omega;\R^m) : [u]_{BV(\Omega;\R^m)}:=|Du|(\Omega)<+\infty}
\end{equation*}
the space of $m$-vector valued functions of bounded variation on~$\Omega$, see for instance~\cite{AFP00}*{Chapter~3} or~\cite{EG15}*{Chapter~5} for its precise definition and main properties. We also let 
\begin{equation*}
bv(\Omega;\R^m):=\set*{u\in L^1_{\loc}(\Omega;\R^m) : [u]_{BV(\Omega;\,\R^m)}<+\infty}.
\end{equation*}
For $\alpha\in(0,1)$ and $p\in[1,+\infty)$, we denote by 
\begin{equation*}
W^{\alpha,p}(\Omega;\R^m):=\set*{u\in L^p(\Omega;\R^m) : [u]_{W^{\alpha,p}(\Omega;\,\R^m)}\!:=\left(\int_\Omega\int_\Omega\frac{|u(x)-u(y)|^p}{|x-y|^{n+p\alpha}}\,dx\,dy\right)^{\frac{1}{p}}\!<+\infty}
\end{equation*}
the space of $m$-vector valued fractional Sobolev functions on~$\Omega$, see~\cite{DiNPV12} for its precise definition and main properties. We also let 
\begin{equation*}
w^{\alpha,p}(\Omega;\R^m):=\set*{u\in L^1_{\loc}(\Omega;\R^m) : [u]_{W^{\alpha,p}(\Omega;\,\R^m)}<+\infty}.
\end{equation*}
For $\alpha\in(0,1)$ and $p=+\infty$, we simply let
\begin{equation*}
W^{\alpha,\infty}(\Omega;\R^m):=\set*{u\in L^\infty(\Omega;\R^m) : \sup_{x,y\in \Omega,\, x\neq y}\frac{|u(x)-u(y)|}{|x-y|^\alpha}<+\infty},
\end{equation*}
so that $W^{\alpha,\infty}(\Omega;\R^m)=C^{0,\alpha}_b(\Omega;\R^m)$, the space of $m$-vector valued bounded $\alpha$-H\"older continuous functions on~$\Omega$.

\subsection{Definition of \texorpdfstring{$\nabla^\alpha$}{nablaˆalpha} and \texorpdfstring{$\diverg^\alpha$}{divergenceˆalpha}}

We now recall and study the non-local operators~$\nabla^\alpha$ and~$\diverg^\alpha$ introduced by \v{S}ilhav\'{y} in~\cite{S18}.

Let $\alpha\in(0,1)$ and set  
\begin{equation}\label{eq:def_mu_alpha} 
\mu_{n, \alpha} := 2^{\alpha} \pi^{- \frac{n}{2}} \frac{\Gamma\left ( \frac{n + \alpha + 1}{2} \right )}{\Gamma\left ( \frac{1 - \alpha}{2} \right )}.
\end{equation}
We let
\begin{equation} \label{eq:def_frac_grad} 
\nabla^{\alpha} f(x) := \mu_{n, \alpha} \lim_{\eps \to 0} \int_{\{ |z| > \eps \}} \frac{z f(x + z)}{|z|^{n + \alpha + 1}} \, dz
\end{equation}
be the \emph{$\alpha$-gradient} of $f\in C^\infty_c(\R^n)$ at $x\in\R^n$. We also let
\begin{equation}\label{eq:def_frac_div}
\div^{\alpha} \varphi(x) := \mu_{n, \alpha} \lim_{\eps \to 0} \int_{\{ |z| > \eps \}} \frac{z \cdot \varphi(x + z)}{|z|^{n + \alpha + 1}} \, dz
\end{equation}
be the \emph{$\alpha$-divergence} of $\phi\in C^\infty_c(\R^n;\R^n)$ at $x\in\R^n$. The non-local operators~$\nabla^\alpha$ and~$\diverg^\alpha$ are well defined in the sense that the involved integrals converge and the limits exist, see~\cite{S18}*{Section~7}.

Since 
\begin{equation}\label{eq:cancellation_kernel}
\int_{\set*{|z| > \eps}} \frac{z}{|z|^{n + \alpha + 1}} \, dz=0,
\qquad
\forall\eps>0,
\end{equation}
it is immediate to check that $\nabla^{\alpha}c=0$ for all $c\in\R$. Moreover, the cancellation in~\eqref{eq:cancellation_kernel} yields  
\begin{subequations}
\begin{align}
\nabla^{\alpha} f(x)
\label{eq:def_frac_grad_1}
&=\mu_{n, \alpha} \lim_{\eps \to 0} \int_{\{ |y -x| > \eps \}} \frac{(y - x)}{|y - x|^{n + \alpha + 1}} f(y) \, dy\\
\label{eq:def_frac_grad_2}
&= \mu_{n, \alpha} \lim_{\eps \to 0} \int_{\{ |x - y| > \eps \}} \frac{(y - x)  (f(y) - f(x)) }{|y - x|^{n + \alpha + 1}} \, dy\\
\label{eq:def_frac_grad_3}
&=\mu_{n, \alpha} \int_{\R^{n}} \frac{(y - x)  (f(y) - f(x)) }{|y - x|^{n + \alpha + 1}} \, dy,
\qquad
\forall x\in\R^n,
\end{align}
\end{subequations}
for all $f\in C^\infty_c(\R^n)$. Indeed, \eqref{eq:def_frac_grad_1} follows by a simple change of variables and \eqref{eq:def_frac_grad_2} is a consequence of~\eqref{eq:cancellation_kernel}. To prove~\eqref{eq:def_frac_grad_3} it is enough to apply Lebesgue's Dominated Convergence Theorem. Indeed, we can estimate
\begin{equation}\label{eq:Lip_nabla_estim_1}
\int_{\{|y-x|\le1\}} \abs*{\frac{(y - x)  (f(y) - f(x)) }{|y - x|^{n + \alpha + 1}} }\, dy\le\Lip(f)\int_0^1r^{-\alpha}\, dr
\end{equation}
and
\begin{equation}\label{eq:Lip_nabla_estim_2}
\int_{\{|y-x|>1\}} \abs*{\frac{(y - x)  (f(y) - f(x)) }{|y - x|^{n + \alpha + 1}} }\, dy\le2\|f\|_{L^\infty(\R^n)}\int_1^{+\infty}r^{-(1+\alpha)}\, dr.
\end{equation}
As a consequence, the operator $\nabla^\alpha f$ defined by~\eqref{eq:def_frac_grad_3} is well defined for all $f\in\Lip_c(\R^n)$ and satisfies~\eqref{eq:def_frac_grad}, \eqref{eq:def_frac_grad_1} and~\eqref{eq:def_frac_grad_2}.

By~\cite{S18}*{Theorem 4.3}, $\nabla^{\alpha}$ is invariant by translations and rotations and is $\alpha$-homoge\-neous. Moreover, for all $f\in C^\infty_c(\R^n)$ and $\lambda \in \R$, we have
\begin{equation}\label{eq:alpha_homogeneous} 
(\nabla^{\alpha} f(\lambda \cdot))(x) = |\lambda|^{\alpha} \sgn{(\lambda)} (\nabla^{\alpha} f)(\lambda x),
\qquad
x\in\R^n.
\end{equation}

Arguing similarly as above, we can write
\begin{subequations}
\begin{align}
\div^{\alpha} \varphi(x) 
\label{eq:def_frac_div_1}
&= \mu_{n, \alpha} \lim_{\eps \to 0} \int_{\{ |x - y| > \eps \}} \frac{(y - x) \cdot \varphi(y) }{|y - x|^{n + \alpha + 1}} \, dy,\\
\label{eq:def_frac_div_2}
&= \mu_{n, \alpha} \lim_{\eps \to 0} \int_{\{ |x - y| > \eps \}} \frac{(y - x) \cdot (\varphi(y) - \varphi(x)) }{|y - x|^{n + \alpha + 1}} \, dy,\\
\label{eq:def_frac_div_3}
&= \mu_{n, \alpha} \int_{\R^{n}} \frac{(y - x) \cdot (\varphi(y) - \varphi(x)) }{|y - x|^{n + \alpha + 1}} \, dy,
\qquad
\forall x\in\R^n,
\end{align}
\end{subequations}
for all $\phi\in\Lip_c(\R^n; \R^{n})$.

Exploiting~\eqref{eq:def_frac_grad_3} and~\eqref{eq:def_frac_div_3}, we can extend the operators $\nabla^\alpha$ and $\diverg^\alpha$ to functions with $w^{\alpha,1}$-regularity.

\begin{lemma}[Extension of $\nabla^\alpha$ and $\diverg^\alpha$ to $w^{\alpha,1}$]\label{result:Sobolev_frac_enough}
Let $\alpha\in(0,1)$. If $f\in w^{\alpha,1}(\R^n)$ and $\phi\in w^{\alpha,1}(\R^n;\R^n)$, then the functions $\nabla^\alpha f(x)$ and $\diverg^\alpha f(x)$ given by~\eqref{eq:def_frac_grad_3} and~\eqref{eq:def_frac_div_3} respectively are well defined for $\Leb{n}$-a.e.\ $x\in\R^n$. As a consequence, $\nabla^\alpha f(x)$ and $\diverg^\alpha f(x)$ satisfy~\eqref{eq:def_frac_grad}, \eqref{eq:def_frac_grad_1}, \eqref{eq:def_frac_grad_2} and~\eqref{eq:def_frac_div}, \eqref{eq:def_frac_div_1}, \eqref{eq:def_frac_div_2} respectively for $\Leb{n}$-a.e.\ $x\in\R^n$.
\end{lemma}

\begin{proof}
Let $f\in w^{\alpha,1}(\R^n)$. Then
\begin{equation*}
\int_{\R^n}\int_{\R^n}\left|\frac{(y - x)  (f(y) - f(x)) }{|y - x|^{n + \alpha + 1}} \right|\, dy\, dx
\le
[f]_{W^{\alpha,1}(\R^n)}
\end{equation*} 
and thus the function $\nabla^\alpha f(x)$ given by~\eqref{eq:def_frac_grad_3} is well defined for $\Leb{n}$-a.e.\ $x\in\R^n$ and satisfies~\eqref{eq:def_frac_grad}, \eqref{eq:def_frac_grad_1} and \eqref{eq:def_frac_grad_2} by~\eqref{eq:cancellation_kernel} and by Lebesgue's Dominated Convergence Theorem. A similar argument proves the result for any $\phi\in w^{\alpha,1}(\R^n;\R^n)$.
\end{proof}

\subsection{Equivalent definition of \texorpdfstring{$\nabla^\alpha$}{nablaˆalpha} and \texorpdfstring{$\div^\alpha$}{divergenceˆalpha} via Riesz potential}

We let
\begin{equation} \label{eq:Riesz_potential_def} 
I_{\alpha} f(x) := 
\frac{\Gamma \left ( \frac{n - \alpha}{2} \right )}{2^{\alpha} \pi^{\frac{n}{2}} \Gamma \left ( \frac{\alpha}{2} \right )} \int_{\R^{n}} \frac{f(y)}{|x - y|^{n - \alpha}} \, dy, 
\qquad
x\in\R^n,
\end{equation}
be the \emph{Riesz potential} of order $\alpha\in(0,n)$ of a function $f\in C^\infty_c(\R^n;\R^m)$. 

Let $\alpha\in(0,1)$. Note that $I_{1-\alpha} f\in C^\infty(\R^n;\R^m)$. Recalling~\eqref{eq:def_mu_alpha}, one easily sees that
\begin{equation*} 
I_{1 - \alpha} f (x)
= \frac{\mu_{n, \alpha}}{n + \alpha - 1} \int_{\R^{n}} \frac{f(x+ y)}{|y|^{n + \alpha - 1}} \, dy
\end{equation*}
and 
\begin{equation*} 
\nabla I_{1 - \alpha} f(x) 
=  \frac{\mu_{n, \alpha}}{n + \alpha - 1} \int_{\R^{n}} \frac{\nabla_{x} f(x + y)}{|y|^{n + \alpha - 1}}  \, dy
= \frac{\mu_{n, \alpha}}{n + \alpha - 1} \int_{\R^{n}} \frac{\nabla_{y} f(x + y)}{|y|^{n + \alpha - 1}}  \, dy,
\end{equation*}
so that 
\begin{equation*}
\nabla I_{1-\alpha}f=I_{1-\alpha}\nabla f
\end{equation*}
for all $f\in C^\infty_c(\R^n)$. A similar argument proves that 
\begin{equation*}
\div I_{1-\alpha}\phi=I_{1-\alpha}\div\phi
\end{equation*}
for all $\phi\in C^\infty_c(\R^n;\R^n)$.
 
Thus, accordingly to the approach developed in~\cites{H59,SS18,SSS18,SSS15,SS15,SSVanS17}, we can consider the operators
\begin{equation*}
\widetilde{\nabla^\alpha}:=\nabla I_{1-\alpha}\colon C^\infty_c(\R^n)\to C^\infty(\R^n;\R^n)
\end{equation*}  
and
\begin{equation*}
\widetilde{\div^\alpha}:=\div I_{1-\alpha}\colon C^\infty_c(\R^n;\R^n)\to C^\infty(\R^n).
\end{equation*}
We can prove that these two operators coincide with the operators defined in~\eqref{eq:def_frac_grad} and~\eqref{eq:def_frac_div}. See also~\cite{SS15}*{Theorem~1.2}.  

\begin{proposition}[Equivalence]\label{rem:equiv_def_Riesz_potential} 
Let $\alpha\in(0,1)$. We have $\widetilde{\nabla^\alpha}=\nabla^\alpha$ on~$\Lip_c(\R^n)$ and $\widetilde{\div^\alpha}=\div^\alpha$ on~$\Lip_c(\R^n;\R^n)$.
\end{proposition}

\begin{proof}
Let $f\in\Lip_c(\R^n)$ and fix $x\in\R^n$. Integrating by parts, we can compute
\begin{align*} 
\widetilde{\nabla^\alpha}f(x)
&=\frac{\mu_{n, \alpha}}{n + \alpha - 1} \lim_{\eps\to0}\int_{\set*{|y|>\eps}} \frac{\nabla_{y} f(x + y)}{|y|^{n + \alpha - 1}}  \, dy \\ 
& = \mu_{n, \alpha} \lim_{\eps\to0}\int_{\set*{|y|>\eps}} \frac{y f(y + x)}{|y |^{n + \alpha + 1}} \, dy 
= \nabla^\alpha f(x),
\end{align*}
since we can estimate 
\begin{equation*}
\begin{split}
\left | \int_{\set*{|y|=\eps}} \frac{f(x + y)}{|y|^{n + \alpha - 1}}\frac{y}{|y|}\, d \Haus{n - 1}(y) \right |
& = \left | \int_{\set*{|y|=\eps}}  \frac{(f(x + y) - f(x))}{|y|^{n + \alpha - 1}}\frac{y}{|y|}\,  d \Haus{n - 1}(y) \right | \\
&\le n \omega_{n} \|\nabla f\|_{L^{\infty}(\R^{n}; \R^{n})} \eps^{1 - \alpha}.   
\end{split} 
\end{equation*}
The proof of $\widetilde{\div^\alpha}\phi=\div^\alpha\phi$ for all $\phi\in\Lip_c(\R^n;\R^n)$ follows similarly.
\end{proof}

A useful consequence of the equivalence proved in \cref{rem:equiv_def_Riesz_potential} above is the following result.

\begin{corollary}[Representation formula for $\diverg^\alpha$ and $\nabla^\alpha$] \label{prop:frac_div_repr} 
Let $\alpha\in(0,1)$. If $\phi \in\Lip_c(\R^{n}; \R^{n})$ then $\div^{\alpha} \phi \in L^1(\R^n)\cap L^{\infty}(\R^{n})$ with
\begin{equation}\label{eq:frac_div_repres}
\div^{\alpha} \phi(x) 
= \frac{\mu_{n, \alpha}}{n + \alpha - 1} \int_{\R^{n}} \frac{\div \phi(y)}{|y - x|^{n +\alpha-1}} \, dy
\end{equation}
for all $x\in\R^n$,
\begin{equation}\label{eq:frac_div_repr_Lip_L1_estimate} 
\| \div^{\alpha} \phi \|_{L^1(\R^{n})} \le \mu_{n, \alpha}  [\phi]_{W^{\alpha,1}(\R^{n};\R^n)}
\end{equation}
and
\begin{equation} \label{eq:frac_div_repr_Lip_estimate} 
\| \div^{\alpha} \phi \|_{L^{\infty}(\R^{n})} \le C_{n, \alpha, U}  \|\div \phi \|_{L^{\infty}(\R^{n})}
\end{equation}
for any bounded open set $U\subset\R^n$ such that $\supp(\phi) \subset U$, where 
\begin{equation} \label{eq:constant_estimate} 
C_{n, \alpha, U} := 
\frac{n \mu_{n, \alpha}}{(1 - \alpha)(n + \alpha - 1)}
\left(
	\omega_n\diam(U)^{1 - \alpha} 
	+\left(
		\frac{n \omega_{n}}{n+\alpha-1}
	\right)^\frac{n + \alpha - 1}{n}|U|^\frac{1 - \alpha}{n}
\right). 
\end{equation}

Analogously, if $f\in\Lip_c(\R^{n})$, then $\nabla^{\alpha}f \in L^1(\R^n; \R^{n})\cap L^{\infty}(\R^{n}; \R^{n})$ with
\begin{equation}\label{eq:frac_nabla_repres}
\nabla^{\alpha} f(x) 
= \frac{\mu_{n, \alpha}}{n + \alpha - 1} \int_{\R^{n}} \frac{\nabla f(y)}{|y - x|^{n +\alpha-1}} \, dy
\end{equation}
for all $x\in\R^n$,
\begin{equation}\label{eq:frac_nabla_repr_Lip_L1_estimate} 
\| \nabla^{\alpha} f \|_{L^1(\R^{n};\R^n)} \le \mu_{n, \alpha}  [f]_{W^{\alpha,1}(\R^{n})}
\end{equation}
and
\begin{equation} \label{eq:frac_nabla_repr_Lip_estimate} 
\| \nabla^{\alpha} f \|_{L^{\infty}(\R^{n};\R^n)} \le C_{n, \alpha, U}  \|\nabla f\|_{L^{\infty}(\R^{n};\R^n)}
\end{equation}
for any bounded open set $U\subset\R^n$ such that $\supp(f) \subset U$, where $C_{n, \alpha, U}$ is as in~\eqref{eq:constant_estimate}.
\end{corollary}

\begin{proof}
The representation formula~\eqref{eq:frac_div_repres} follows directly from \cref{rem:equiv_def_Riesz_potential}. The estimate in~\eqref{eq:frac_div_repr_Lip_L1_estimate} is a consequence of \cref{result:Sobolev_frac_enough}. Finally, if $U\subset\R^n$ is a bounded open set such that $\supp(\phi)\subset U$, then
\begin{equation*}
\begin{split}
|\div^{\alpha} \phi(x)| 
&\le \frac{\mu_{n, \alpha}}{n + \alpha - 1} \int_{\R^{n}} |y - x|^{1 - n - \alpha} \, |\div \phi(y)| \, dy\\
&\le \frac{\mu_{n, \alpha}\|\div \phi\|_{L^\infty(\R^n)}}{n + \alpha - 1} \int_U |y - x|^{1 - n - \alpha} \, dy 
\end{split}
\end{equation*}
and~\eqref{eq:frac_div_repr_Lip_estimate} follows by \cref{result:riesz_kernel_estimate} below. The proof of~\eqref{eq:frac_nabla_repres}, \eqref{eq:frac_nabla_repr_Lip_L1_estimate} and~\eqref{eq:frac_nabla_repr_Lip_estimate} is similar and is left to the reader.  
\end{proof}

\begin{lemma} \label{result:riesz_kernel_estimate} 
Let $\alpha \in (0, 1)$ and let $U\subset\R^n$ be a bounded open set. For all $x \in \R^{n}$, we have
\begin{equation} \label{eq:riesz_kernel_estimate} 
\int_{U} |y - x|^{1 - n - \alpha} \, d y 
\le 
\frac{n}{1 - \alpha}
\left(
	\omega_n\diam(U)^{1 - \alpha} 
	+\left(
		\frac{n \omega_{n}}{n+\alpha-1}
	\right)^\frac{n + \alpha - 1}{n}|U|^\frac{1 - \alpha}{n}
\right).
\end{equation}
\end{lemma}

\begin{proof} 
For $\delta > 0$, set $U^{\delta} := \{ x \in \R^{n} : \dist(x, U) < \delta \}$. Since clearly
\begin{equation*}
x \in U^{\delta}
\implies
B_{(\diam{(U)} + \delta)}(x) \supset U,
\end{equation*}
for all $x\in U^\delta$ we can estimate
\begin{align*} 
\int_{U} |y - x|^{1 - n - \alpha} \, d y 
&\le \int_{B_{(\diam (U) + \delta)}(x)} |y - x|^{1 - n - \alpha} \, d y \\
& = n \omega_{n} \int_{0}^{\diam(U) + \delta} r^{- \alpha} \, dr \\
& = \frac{n \omega_{n}}{1 - \alpha} \left (\diam{(U)} + \delta \right)^{1 - \alpha}. 
\end{align*}
On the other hand, it is plain that 
\begin{equation*}
x \notin U^{\delta},\ y \in U 
\implies
|y - x| > \delta,
\end{equation*}
so that for all $x\notin U^\delta$ we can estimate
\begin{align*} 
\int_{U} |y - x|^{1 - n - \alpha} \, d y \le \delta^{1 - n - \alpha} |U|. 
\end{align*}
Thus, for all $\delta > 0$ and $x \in \R^{n}$, we can estimate
\begin{align*} 
\int_{U} |y - x|^{1 - n - \alpha} \, d y 
& \le \frac{n \omega_{n}}{1 - \alpha} \left (\diam{(U)} + \delta \right)^{1 - \alpha} + \delta^{1 - n - \alpha} |U| \\
& \le \frac{n \omega_{n}}{1 - \alpha} \left (\diam{(U)}^{1 - \alpha} + \delta^{1 - \alpha} \right) +  \delta^{1 - n - \alpha} |U|
\end{align*}
since the function $s \mapsto s^{1 - \alpha}$ is subadditive for all $s>0$. Thus~\eqref{eq:riesz_kernel_estimate} follows minimising in~$\delta>0$ the right-hand side.
\end{proof}

\subsection{Duality and Leibniz's rules}

We now study the properties of the operators $\nabla^\alpha$ and $\diverg^\alpha$. We begin with the following duality relation, see~\cite{S18}*{Section~6}.

\begin{lemma}[Duality]\label{result:duality}
Let $\alpha\in(0,1)$. For all $f\in\Lip_c(\R^n)$ and $\phi\in\Lip_c(\R^n;\R^n)$ it holds
\begin{equation}\label{eq:duality_smooth}
\int_{\R^{n}} f\,\div^{\alpha} \varphi \, dx = - \int_{\R^{n}} \varphi \cdot \nabla^{\alpha} f \, dx.
\end{equation}
\end{lemma}

\begin{proof} 
Recalling \cref{result:Sobolev_frac_enough}  and exploiting~\eqref{eq:def_frac_grad_1} and~\eqref{eq:def_frac_div_1}, we can write
\begin{align*} 
\int_{\R^{n}} f\,\div^{\alpha} \varphi \, dx 
& = \mu_{n, \alpha}\int_{\R^{n}} f(x) \lim_{\eps \to 0} \int_{ \{ |x - y| > \eps \}} \frac{(y - x) \cdot \varphi(y) }{|y - x|^{n + \alpha + 1}} \, dy \, dx \\
& = \mu_{n, \alpha} \lim_{\eps \to 0} \int_{\R^{n}} \int_{\{ |x - y| > \eps \}}  f(x) \, \frac{(y - x) \cdot \varphi(y) }{|y - x|^{n + \alpha + 1}} \, dy \, dx  \\
& = - \mu_{n, \alpha} \lim_{\eps \to 0} \int_{\R^{n}}  \int_{\{ |x - y| > \eps \}}  \phi(y)\cdot \frac{(x - y)\,f(x) }{|x - y|^{n + \alpha + 1}} \, dx \, d y \\
& = - \int_{\R^{n}} \varphi(y) \cdot \nabla^{\alpha} f(y) \, dy 
\end{align*}
by the Lebesgue's Dominated Convergence Theorem and Fubini's Theorem.
\end{proof}

We now prove two Leibniz-type rules for the operators $\nabla^\alpha$ and $\diverg^\alpha$, which in particular show the strong non-local nature of these two operators.

\begin{lemma}[Leibniz's rule for $\nabla^\alpha$] \label{lem:Leibniz_frac_grad}
Let $\alpha\in(0,1)$. For all $f,g\in\Lip_c(\R^n)$ it holds
\begin{equation*}
\nabla^{\alpha}(f g)=f \nabla^{\alpha} g + g \nabla^{\alpha} f+\nabla^{\alpha}_{\rm NL}(f,g),
\end{equation*}
where
\begin{equation*}
\nabla^{\alpha}_{\rm NL}(f, g)(x):=\mu_{n, \alpha} \int_{\R^{n}} \frac{(y - x)  (f(y) - f(x)) ( g(y) - g(x) ) }{|y - x|^{n + \alpha + 1}} \, dy,
\quad
\forall x\in\R^n,
\end{equation*}
with~$\mu_{n, \alpha}$ as in~\eqref{eq:def_mu_alpha}. Moreover, it holds
\begin{equation*} 
\|\nabla^{\alpha}_{\rm NL} (f, g)\|_{L^{1}(\R^{n};\R^{n})}\le \mu_{n, \alpha} [f]_{W^{\frac{\alpha}{p}, p}(\R^{n})} [g]_{W^{\frac{\alpha}{q}, q}(\R^{n})} 
\end{equation*} 
with $p,q\in(1,\infty)$ such that $\frac{1}{p} +\frac{1}{q}=1$ and similarly
\begin{equation*} 
\| \nabla^{\alpha}_{\rm NL} (f, g)\|_{L^{1}(\R^{n}; \R^{n})}\le 2 \mu_{n, \alpha} \|f\|_{L^{\infty}(\R^{n})} [g]_{W^{\alpha, 1}(\R^{n})}. 
\end{equation*}
\end{lemma}

\begin{proof}
Given $f, g \in\Lip_c(\R^{n})$, by \cref{result:Sobolev_frac_enough} and by \eqref{eq:def_frac_grad_3} we have
\begin{align*} 
\nabla^{\alpha}(f g)(x) & = \mu_{n, \alpha} \int_{\R^{n}} \frac{(y - x)  (f(y) g(y) - f(x) g(x)) }{|y - x|^{n + \alpha + 1}} \, dy \\
& = \mu_{n, \alpha} \int_{\R^{n}} \frac{(y - x)  (f(y) g(y) - f(y) g(x) + f(y) g(x) - f(x) g(x)  ) }{|y - x|^{n + \alpha + 1}} \, dy \\
& = \mu_{n, \alpha} \int_{\R^{n}} \frac{(y - x)  f(y) ( g(y) - g(x) ) }{|y - x|^{n + \alpha + 1}} \, dy + g(x) \nabla^{\alpha} f(x) \\
& = \mu_{n, \alpha} \int_{\R^{n}} \frac{(y - x)  (f(y) - f(x)) ( g(y) - g(x) ) }{|y - x|^{n + \alpha + 1}} \, dy + f(x) \nabla^{\alpha} g(x) + g(x) \nabla^{\alpha} f(x). 
\end{align*}
We also have that
\begin{align*} 
\| \nabla^{\alpha}_{\rm NL} (f, g) &\|_{L^{1}(\R^{n}; \R^{n})} 
\le \mu_{n, \alpha} \int_{\R^{n}} \int_{\R^{n}} \frac{|f(y) - f(x)|}{|x - y|^{\frac{n + \alpha}{p}}} \frac{|g(y) - g(x)|}{|y - x|^{\frac{n + \alpha}{q}}} \, dy \, dx, \\
& \le \mu_{n, \alpha} \left ( \int_{\R^{n}} \int_{\R^{n}} \frac{|f(y) - f(x)|^{p}}{|x - y|^{n + \alpha}} \, dy \, dx \right )^{\frac{1}{p}} \left ( \int_{\R^{n}} \int_{\R^{n}} \frac{|g(y) - g(x)|^{q}}{|x - y|^{n + \alpha}} \, dy \, dx \right )^{\frac{1}{q}} 
\end{align*}
for any $p, q \in (1, \infty)$ such that $\frac{1}{p} + \frac{1}{q} = 1$. The case $p = \infty$, $q = 1$ follows similarly.
\end{proof}

\begin{lemma}[Leibniz's rule for $\div^\alpha$] \label{lem:Leibniz_frac_div}
Let $\alpha\in(0,1)$. For all $f \in\Lip_c(\R^n)$ and $\phi\in\Lip_c(\R^{n};\R^{n})$ it holds
\begin{equation*}
\div^{\alpha}(f \phi)=f \div^{\alpha} \phi + \phi \cdot \nabla^{\alpha} f+\div^{\alpha}_{\rm NL}(f, \phi),
\end{equation*}
where
\begin{equation} \label{eq:div_NL_term}
\div^{\alpha}_{\rm NL}(f, \phi)(x):=\mu_{n, \alpha} \int_{\R^{n}} \frac{(y - x)  \cdot ( \phi(y) - \phi(x) ) (f(y) - f(x)) }{|y - x|^{n + \alpha + 1}} \, dy,
\quad
\forall x\in\R^n,
\end{equation}
with~$\mu_{n, \alpha}$ as in~\eqref{eq:def_mu_alpha}. Moreover, it holds
\begin{equation*} 
\|\div^{\alpha}_{\rm NL} (f, \phi)\|_{L^{1}(\R^{n})}\le \mu_{n, \alpha} [f]_{W^{\frac{\alpha}{p}, p}(\R^{n})} [\phi]_{W^{\frac{\alpha}{q}, q}(\R^{n}; \R^{n})} 
\end{equation*} 
with $p,q\in(1,\infty)$ such that $\frac{1}{p}+\frac{1}{q}=1$ and similarly
\begin{align*} 
\| \div^{\alpha}_{\rm NL} (f, \phi)\|_{L^{1}(\R^{n})} & \le 2 \mu_{n, \alpha} \|f\|_{L^{\infty}(\R^{n})} [\phi]_{W^{\alpha, 1}(\R^{n}; \R^{n})}, \\
\| \div^{\alpha}_{\rm NL} (f, \phi)\|_{L^{1}(\R^{n})} & \le 2 \mu_{n, \alpha} \|\phi \|_{L^{\infty}(\R^{n}; \R^{n})} [f]_{W^{\alpha, 1}(\R^{n})}. 
\end{align*}
\end{lemma}

\begin{proof}
Given $f \in\Lip_c(\R^n)$ and $\phi\in\Lip_c(\R^n;\R^n)$, by \cref{result:Sobolev_frac_enough} and by \eqref{eq:def_frac_grad_3} we have
\begin{align*} 
\div^{\alpha}(f \phi)(x) & = \mu_{n, \alpha} \int_{\R^{n}} \frac{(y - x) \cdot  (f(y) \phi(y) - f(x) \phi(x)) }{|y - x|^{n + \alpha + 1}} \, dy \\
& = \mu_{n, \alpha} \int_{\R^{n}} \frac{(y - x) \cdot (f(y) \phi(y) - f(y) \phi(x) + f(y) \phi(x) - f(x) \phi(x)  ) }{|y - x|^{n + \alpha + 1}} \, dy \\
& = \mu_{n, \alpha} \int_{\R^{n}} \frac{(y - x) \cdot ( \phi(y) - \phi(x) )  f(y) }{|y - x|^{n + \alpha + 1}} \, dy + \phi(x) \cdot \nabla^{\alpha} f(x) \\
& = \mu_{n, \alpha} \int_{\R^{n}} \frac{(y - x)  \cdot ( \phi(y) - \phi(x) ) (f(y) - f(x))}{|y - x|^{n + \alpha + 1}} \, dy + f(x) \div^{\alpha} \phi(x) + \\
& \quad + \phi(x) \cdot \nabla^{\alpha} f(x). 
\end{align*}
We also have that
\begin{align*} 
\| \div^{\alpha}_{\rm NL} (f, \phi) &\|_{L^{1}(\R^{n})}
\le \mu_{n, \alpha} \int_{\R^{n}} \int_{\R^{n}} \frac{|f(y) - f(x)|}{|x - y|^{\frac{n + \alpha}{p}}} \frac{|\phi(y) - \phi(x)|}{|y - x|^{\frac{n + \alpha}{q}}} \, dy \, dx, \\
& \le \mu_{n, \alpha} \left ( \int_{\R^{n}} \int_{\R^{n}} \frac{|f(y) - f(x)|^{p}}{|x - y|^{n + \alpha}} \, dy \, dx \right )^{\frac{1}{p}} \left ( \int_{\R^{n}} \int_{\R^{n}} \frac{|\phi(y) - \phi(x)|^{q}}{|x - y|^{n + \alpha}} \, dy \, dx \right )^{\frac{1}{q}} 
\end{align*}
for any $p, q \in (1, \infty)$ such that $\frac{1}{p} + \frac{1}{q} = 1$. The case $p = \infty$, $q = 1$ follows similarly.
\end{proof}

\begin{remark}[Extension of $\nabla^\alpha_{\mathrm{NL}}$ and $\diverg^\alpha_{\mathrm{NL}}$ to fractional Sobolev spaces] \label{rem:div_nabla_NL_extension} 
Thanks to the estimates in \cref{lem:Leibniz_frac_grad}, for all $\alpha \in (0, 1)$ the bilinear operator 
\begin{equation*}
\nabla^{\alpha}_{\rm NL} \colon
\Lip_c(\R^{n}) \times \Lip_c(\R^{n}) 
\to 
L^{1}(\R^{n}; \R^{n})
\end{equation*}
can be continuously extended to a bilinear operator 
\begin{equation*}
\nabla^{\alpha}_{\rm NL} \colon 
w^{\frac{\alpha}{p}, p}(\R^{n})\times w^{\frac{\alpha}{q}, q}(\R^{n}) 
\to 
L^{1}(\R^{n}; \R^{n})
\end{equation*}
for any $p, q \in[1,\infty]$ such that $\frac{1}{p} + \frac{1}{q} = 1$, for which we retain the same notation (we tacitly adopt the convention $w^{\frac{\alpha}{\infty},\infty}=L^\infty$). Analogously, because of the estimates in \cref{lem:Leibniz_frac_div}, the bilinear operator 
\begin{equation*}
\div^{\alpha}_{\rm NL} \colon 
\Lip_c(\R^{n}) \times \Lip_c(\R^{n}; \R^{n}) 
\to 
L^{1}(\R^{n})
\end{equation*}
can be continuously extended to a bilinear operator 
\begin{equation*}
\div^{\alpha}_{\rm NL} \colon 
w^{\frac{\alpha}{p}, p}(\R^{n}) \times w^{\frac{\alpha}{q}, q}(\R^{n}; \R^{n})
\to 
L^{1}(\R^{n})
\end{equation*} 
for any $p, q \in [1, \infty]$ such that $\frac{1}{p} + \frac{1}{q} = 1$, for which we retain the same notation. 
\end{remark}

\section{Fractional \texorpdfstring{$BV$}{BV} functions}
\label{sec:frac_BV_func}

In this section we introduce and study the fractional $BV$ space naturally induced by the operators $\nabla^\alpha$ and $\diverg^\alpha$ defined in \cref{sec:Silhavy_calculus} following De Giorgi's distributional approach. In the presentation of the results, we will frequently refer to~\cite{EG15}*{Chapter~5}.

\subsection{Definition of \texorpdfstring{$BV^\alpha(\R^n)$}{BVˆalpha(Rˆn)} and Structure Theorem}

In analogy with the classical case (see~\cite{EG15}*{Definition~5.1} for instance), we start with the following definition. 

\begin{definition}[$BV^\alpha(\R^n)$ space]\label{def:BV_alpha_space}
Let $\alpha\in(0,1)$. A function $f\in L^1(\R^n)$ belongs to the space $BV^\alpha(\R^n)$ if 
\begin{equation*}
\sup\set*{\int_{\R^n} f\,\div^\alpha\phi\ dx : \phi\in C^\infty_c(\R^n;\R^n),\ \|\phi\|_{L^\infty(\R^n;\R^n)}\le1}<+\infty.
\end{equation*}
\end{definition}

We can now state the following fundamental result relating non-local distributional gradients of $BV^\alpha$ functions to vector valued Radon measures. 

\begin{theorem}[Structure Theorem for $BV^\alpha$ functions]\label{th:structure_BV_alpha}
Let $\alpha\in(0,1)$ and $f \in L^{1}(\R^{n})$. Then, $f \in BV^\alpha(\R^n)$ if and only if there exists a finite vector valued Radon measure $D^{\alpha} f \in \M(\R^n; \R^{n})$ such that 
\begin{equation}\label{eq:BV_alpha_duality} 
\int_{\R^{n}} f\, \div^{\alpha} \varphi \, dx = - \int_{\R^n} \varphi \cdot d D^{\alpha} f 
\end{equation}
for all $\varphi \in C^{\infty}_{c}(\R^n; \R^{n})$. In addition, for any open set $U\subset\R^n$ it holds
\begin{equation}\label{eq:BV_alpha_weak_grad_tot_var}
|D^{\alpha} f|(U) = \sup\set*{\int_{\R^n}f\,\div^\alpha\phi\ dx : \phi\in C^\infty_c(U;\R^n),\ \|\phi\|_{L^\infty(U;\R^n)}\le1}.
\end{equation}
\end{theorem}

\begin{proof} 
If $f \in L^{1}(\R^{n})$ and if there exists a finite vector valued Radon measure $D^{\alpha} f \in \M(\R^n; \R^{n})$ such that \eqref{eq:BV_alpha_duality} holds, then $f \in BV^{\alpha}(\R^{n})$ by \cref{def:BV_alpha_space}.

If $f \in BV^{\alpha}(\R^{n})$, then the proof is identical to the one of~\cite{EG15}*{Theorem~5.1}, with minor modifications. Define the linear functional $L\colon C_{c}^\infty(\R^n; \R^{n})\to\R$ setting
\begin{equation*} 
L(\phi) := - \int_{\R^{n}} f\, \div^{\alpha}\phi \, dx
\qquad
\forall\phi \in C^{\infty}_{c}(\R^n; \R^{n}).
\end{equation*}
Note that $L$ is well defined thanks to \cref{prop:frac_div_repr}. Since $f\in BV^\alpha(\R^n)$, we have
\begin{equation*}
C(U):=\sup\set*{L(\phi) : \phi\in C^\infty_c(U;\R^n),\ \|\phi\|_{L^\infty(U;\R^n)}\le 1}<+\infty
\end{equation*}
for each open set $U\subset\R^n$, so that
\begin{equation*} 
\left | L(\phi) \right | \le C(U) \|\phi\|_{L^\infty(U;\R^n)}
\qquad
\forall \phi \in C^{\infty}_{c}(U; \R^{n}).
\end{equation*}
Thus, by the density of $C^{\infty}_{c}(\R^n; \R^{n})$ in $C_{c}(\R^n; \R^{n})$, the functional $L$ can be uniquely extended to a continuous linear functional $\tilde{L}\colon C_{c}(\R^n; \R^{n})\to\R$ and the conclusion follows by Riesz's Representation Theorem.
\end{proof}

\subsection{Lower semicontinuity of fractional variation}

Similarly to the classical case, the \emph{fractional variation measure} given by \cref{th:structure_BV_alpha} in~\eqref{eq:BV_alpha_weak_grad_tot_var} is lower semicontinuous with respect to $L^1$-convergence.

\begin{proposition}[Lower semicontinuity of fractional variation measure]\label{result:frac_var_meas_is_lsc}
Let $\alpha\in(0,1)$. If $(f_k)_{k\in\N}\subset BV^\alpha(\R^n)$ and $f_k\to f$ in $L^1(\R^n)$ as $k\to+\infty$, then $f\in BV^\alpha(\R^n)$ with
\begin{equation*}
|D^\alpha f|(U)\le\liminf_{k\to+\infty}|D^\alpha f_k|(U)
\end{equation*}
for any open set $U\subset\R^n$.
\end{proposition}

\begin{proof}
Let $\phi\in C^\infty_c(\R^n;\R^n)$ with $\|\phi\|_{L^\infty(\R^n; \R^{n})}\le1$. Then $\div^\alpha\phi\in L^\infty(\R^n)$ by \cref{prop:frac_div_repr} and so we can estimate
\begin{equation*}
\begin{split}
\int_{\R^n}f\,\div^\alpha\phi\ dx
=\lim_{k\to+\infty}\int_{\R^n}f_k\,\div^\alpha\phi\ dx
=-\lim_{k\to+\infty}\int_{\R^n}\phi\ dD^\alpha f_k
\le\liminf_{k\to+\infty} |D^\alpha f_k|(\R^n).
\end{split}
\end{equation*}
This shows that 
\begin{equation*} 
|D^{\alpha} f|(\R^{n}) \le \liminf_{k \to + \infty} |D^\alpha f_k|(\R^n),
\end{equation*}
thanks to \cref{th:structure_BV_alpha}. Finally, if $U$ is an open set in $\R^{n}$, it is enough to take $\phi \in C^\infty_c(U;\R^n)$ and to argue as above, applying \eqref{eq:BV_alpha_weak_grad_tot_var}.
\end{proof}

From \cref{result:frac_var_meas_is_lsc} we immediately deduce the following result, whose standard proof is left to the reader.

\begin{corollary}[$BV^\alpha$ is a Banach space]
Let $\alpha\in(0,1)$. The linear space $BV^{\alpha}(\R^n)$ equipped with the norm
\begin{equation*}
\|f\|_{BV^\alpha(\R^n)}:=\|f\|_{L^1(\R^n)}+|D^\alpha f|(\R^n),
\qquad
f\in BV^\alpha(\R^n),
\end{equation*}
where $D^\alpha f$ is given by \cref{th:structure_BV_alpha}, is a Banach space.
\end{corollary}

\subsection{Approximation by smooth functions}

Here and in the following, we let $\rho\in C^\infty_c(\R^n)$ be a function such that
\begin{equation}\label{eq:def_rho}
\supp\rho\subset B_1,
\qquad
\rho\ge0,
\qquad
\int_{\R^n}\rho(x)\ dx=1,
\end{equation}
see~\cite{EG15}*{Section~4.2.1} for an example. We thus let $(\rho_\eps)_{\eps>0}\subset C^\infty_c(\R^n)$ be defined as
\begin{equation}\label{eq:def_rho_eps}
\rho_\eps(x):=\frac{1}{\eps^n}\rho\left(\frac{x}{\eps}\right)
\quad
\forall x\in\R^n.
\end{equation}  
We call $(\rho_\eps)_{\eps>0}$ a family of \emph{standard mollifiers}. We have the following result.

\begin{lemma}[Convolution with standard mollifiers] \label{result:commutation_mollifier} 
Let $\alpha\in(0,1)$ and let $(\rho_\eps)_{\eps>0}$ as in~\eqref{eq:def_rho_eps}. If $\phi \in \Lip_{c}(\R^{n}; \R^{n})$, then
\begin{equation} \label{eq:commutation_mollifier_frac_div} 
\div^{\alpha} (\rho_{\eps} \ast \phi) = \rho_{\eps} \ast \div^{\alpha} \phi 
\end{equation}
for any $\eps>0$. Thus, if $f \in BV^{\alpha}(\R^n)$, then
\begin{equation} \label{eq:commutation_mollifier_frac_grad} 
D^{\alpha} (\rho_{\eps} \ast f) = (\rho_{\eps} \ast D^{\alpha} f) \Leb{n}
\end{equation}
for any $\eps>0$, and
\begin{equation} \label{eq:commutation_mollifier_weak_conv} 
D^{\alpha} (\rho_{\eps} \ast f) \weakto D^{\alpha} f 
\end{equation}
in $\M(\R^n; \R^{n})$ as $\eps\to0$.
\end{lemma}

\begin{proof} 
Let $\phi \in \Lip_{c}(\R^n; \R^{n})$ and $x \in \R^{n}$. Recalling~\eqref{eq:frac_div_repres}, we can write
\begin{equation*}
\div^\alpha\phi=K_{n,\alpha}*\div\phi,
\end{equation*}
where
\begin{equation*}
K_{n,\alpha}(x)=\frac{\mu_{n,\alpha}}{n+\alpha-1}\,|x|^{1-n-\alpha},
\quad
x\in\R^n \setminus \{0\}.
\end{equation*}
Since $\rho_{\eps}*\phi\in \Lip_{c}(\R^n; \R^{n})$, we can compute
\begin{equation*}
\begin{split}
\div^\alpha(\rho_{\eps}*\phi)
&=K_{n,\alpha}*\div(\rho_{\eps}*\phi)\\
&=K_{n,\alpha}*(\rho_{\eps}*\div\phi)\\
&=\rho_{\eps}*(K_{n,\alpha}*\div\phi)\\
&=\rho_{\eps}*\div^\alpha\phi
\end{split}
\end{equation*} 
and~\eqref{eq:commutation_mollifier_frac_div} follows.
Now let $f \in BV^{\alpha}(\R^n)$ and $\phi \in C^{\infty}_{c}(\R^n; \R^{n})$. By~\eqref{eq:BV_alpha_duality} and~\eqref{eq:commutation_mollifier_frac_div}, for all $\eps>0$ we can compute
\begin{align*} 
- \int_{\R^{n}} (\rho_{\eps} \ast f)\, \div^{\alpha} \phi \, dx 
&= - \int_{\R^{n}} f\, (\rho_{\eps} \ast \div^{\alpha} \phi) \, dx \\
& = - \int_{\R^{n}} f \,\div^{\alpha} (\rho_{\eps} \ast \phi) \, dx \\
&= \int_{\R^{n}} (\rho_{\eps} \ast \phi) \, d D^{\alpha} f \\
& = \int_{\R^{n}} \phi \cdot (\rho_{\eps} \ast D^{\alpha} f) \, dx, 
\end{align*}
proving~\eqref{eq:commutation_mollifier_frac_grad}. The convergence in~\eqref{eq:commutation_mollifier_weak_conv} thus follows from standard properties of the mollification of Radon measures, see~\cite{AFP00}*{Theorem 2.2} for instance. 
\end{proof}

As an immediate application of \cref{result:commutation_mollifier}, we can prove that a function in $BV^\alpha(\R^n)$ can be tested against the fractional divergence of any $\Lip_c$-regular vector field.

\begin{proposition}[$\Lip_c$-regular test] \label{result:Lip_test} 
Let $\alpha \in (0, 1)$. If $f \in BV^{\alpha}(\R^{n})$, then~\eqref{eq:BV_alpha_duality} holds for all $\phi \in \Lip_{c}(\R^{n}; \R^{n})$.
\end{proposition}

\begin{proof} 
Fix $\phi\in\Lip_c(\R^n;\R^n)$ and let $(\rho_\eps)_{\eps>0}\subset C^\infty_c(\R^n)$ be as in~\eqref{eq:def_rho_eps}. Then $\rho_\eps*\phi\in C^\infty_c(\R^n;\R^n)$ and so, by~\cref{result:commutation_mollifier} and~\eqref{eq:BV_alpha_duality}, we have
\begin{equation}\label{eq:duality_Lip_eps} 
\int_{\R^{n}} (\rho_{\eps} \ast f)\, \div^{\alpha}  \phi \, dx 
=\int_{\R^{n}} f\, \div^{\alpha} (\rho_{\eps} \ast \phi) \, dx 
= - \int_{\R^{n}} (\rho_{\eps} \ast \phi) \cdot d D^{\alpha} f. 
\end{equation}
Since $\rho_{\eps} *\phi \to \phi$ uniformly and $\rho_{\eps} \ast f\to f$ in $L^1(\R^n)$ as $\eps\to0$, and $\diverg^\alpha\phi\in L^\infty(\R^n)$ by \cref{prop:frac_div_repr}, we can pass to the limit as $\eps\to0$ in~\eqref{eq:duality_Lip_eps} getting
\begin{equation*}
\int_{\R^{n}} f\, \div^{\alpha} \phi \, dx 
= - \int_{\R^{n}} \phi \cdot d D^{\alpha} f
\end{equation*}
for any $\phi\in\Lip_c(\R^n;\R^n)$.
\end{proof}

As in the classical case, we can prove the density of $C^\infty(\R^n)\cap BV^\alpha(\R^n)$ in $BV^\alpha(\R^n)$.

\begin{theorem}[Approximation by $C^\infty\cap BV^\alpha$ functions]\label{result:approx_by_smooth_BV}
Let $\alpha\in(0,1)$. If $f\in BV^\alpha(\R^n)$, then there exists $(f_k)_{k\in\N}\subset BV^\alpha(\R^n)\cap C^\infty(\R^n)$ such that
\begin{enumerate}[\indent(i)]
\item $f_{k} \to f$ in $L^{1}(\R^{n})$;
\item $|D^{\alpha} f_{k}|(\R^{n}) \to |D^{\alpha} f|(\R^{n})$.
\end{enumerate}
\end{theorem}

\begin{proof}
Let $(\rho_\eps)_{\eps>0}\subset C^\infty_c(\R^n)$ be as in~\eqref{eq:def_rho_eps}. Fix $f\in BV^\alpha(\R^n)$ and consider $f_\eps:=f*\rho_\eps$ for all $\eps>0$. Since $f_\eps\to f$ in $L^1(\R^n)$, by \cref{result:frac_var_meas_is_lsc} we get that
\begin{equation*}
|D^\alpha f|(\R^n)\le\liminf_{\eps\to0}|D^\alpha f_\eps|(\R^n).
\end{equation*}
By \cref{result:commutation_mollifier} we also have that
\begin{equation*}
|D^\alpha f_\eps|(\R^n)
=\int_{\R^n}|\rho_\eps*D^\alpha f|\,dx
\le|D^\alpha f|(\R^n)
\end{equation*}
and the proof is complete.
\end{proof}

Let $(\eta_R)_{R>0}\subset C^\infty_c(\R^n)$ be such that
\begin{equation}\label{eq:def_cut-off}
0\le\eta_R\le 1,
\qquad
\eta_R=1\text{ on $B_R$},
\qquad
\supp(\eta_R)\subset B_{R+1},
\qquad
\Lip(\eta_R)\le 2.
\end{equation} 
We call $\eta_R$ a \emph{cut-off function}. As in the classical case, we can prove the density of $C^\infty_c(\R^n)$ in $BV^\alpha(\R^n)$.

\begin{theorem}[Approximation by $C^\infty_c$ functions]\label{result:approx_by_smooth_c_BV}
Let $\alpha\in(0,1)$. If $f\in BV^\alpha(\R^n)$, then there exists $(f_k)_{k\in\N}\subset C^\infty_c(\R^n)$ such that
\begin{enumerate}[\indent(i)]
\item $f_{k} \to f$ in $L^{1}(\R^{n})$;
\item $|D^{\alpha} f_{k}|(\R^{n}) \to |D^{\alpha} f|(\R^{n})$.
\end{enumerate}
\end{theorem}

\begin{proof}
Let $(\eta_R)_{R>0}\subset C^\infty_c(\R^n)$ be as in~\eqref{eq:def_cut-off}. Thanks to \cref{result:approx_by_smooth_BV}, it is enough to prove that $f\eta_R\to f$ in $BV^\alpha(\R^n)$ as $R\to+\infty$ for all $f\in C^\infty(\R^n)\cap BV^\alpha(\R^n)$. Clearly, $f\eta_R\to f$ in $L^1(\R^n)$ as $R\to+\infty$. Thus, by \cref{result:frac_var_meas_is_lsc}, we just need to prove that 
\begin{equation}\label{eq:approx_by_smooth_c_BV_limsup}
\limsup_{R\to+\infty}|D^\alpha (f\eta_R)|(\R^n)\le |D^\alpha f|(\R^n).
\end{equation}
Fix $\phi\in C^\infty_c(\R^n;\R^n)$. Then, by \cref{lem:Leibniz_frac_div}, we get
\begin{equation*}
\begin{split}
\int_{\R^n}f\eta_R\,\div^\alpha\phi\,dx
&=\int_{\R^n}f\,\div^\alpha(\eta_R\phi)\,dx
-\int_{\R^n}f\,\phi\cdot\nabla^\alpha\eta_R\,dx
-\int_{\R^n}f\,\div^\alpha_{\mathrm{NL}}(\eta_R, \phi)\,dx.
\end{split}
\end{equation*}
Since $f\in BV^\alpha(\R^n)$ and $0\le\eta_R\le 1$, we have
\begin{equation*}
\abs*{\int_{\R^n}f\,\div^\alpha(\eta_R\phi)\,dx}
\le\|\phi\|_{L^\infty(\R^n;\R^n)}|D^\alpha f|(\R^n).
\end{equation*}
Moreover, we have
\begin{equation*}
\abs*{\int_{\R^n}f\,\phi\cdot\nabla^\alpha\eta_R\,dx}
\le 
\mu_{n, \alpha} \|\phi\|_{L^\infty(\R^n;\R^n)}\int_{\R^n}|f(x)|\int_{\R^n}\frac{|\eta_R(y)-\eta_R(x)|}{|y-x|^{n+\alpha}}\,dy\,dx
\end{equation*}
and, similarly,
\begin{equation*}
\abs*{\int_{\R^n}f\,\div^\alpha_{\mathrm{NL}}(\eta_R,\phi)\,dx}
\le 
2\mu_{n, \alpha}\|\phi\|_{L^\infty(\R^n;\R^n)}\int_{\R^n}|f(x)|\int_{\R^n}\frac{|\eta_R(y)-\eta_R(x)|}{|y-x|^{n+\alpha}}\,dy\,dx.
\end{equation*}
Combining these three estimates, we conclude that
\begin{equation*}
\begin{split}
\left | \int_{\R^n}f\eta_R\,\div^\alpha\phi\,dx \right |
&\le
\|\phi\|_{L^\infty(\R^n;\R^n)}|D^\alpha f|(\R^n)\\
&\quad+3\mu_{n, \alpha}\|\phi\|_{L^\infty(\R^n;\R^n)}\int_{\R^n}|f(x)|\int_{\R^n}\frac{|\eta_R(y)-\eta_R(x)|}{|y-x|^{n+\alpha}}\,dy\,dx
\end{split}
\end{equation*}
and~\eqref{eq:approx_by_smooth_c_BV_limsup} follows by \cref{th:structure_BV_alpha}. Indeed, we have
\begin{equation*}
\lim_{R\to+\infty}\int_{\R^n}|f(x)|\int_{\R^n}\frac{|\eta_R(y)-\eta_R(x)|}{|y-x|^{n+\alpha}}\,dy\,dx=0
\end{equation*}
combining~\eqref{eq:Lip_nabla_estim_1}, \eqref{eq:Lip_nabla_estim_2} and~\eqref{eq:def_cut-off} with Lebesgue's Dominated Convergence Theorem.
\end{proof}

\subsection{Gagliardo--Nirenberg--Sobolev inequality}

Thanks to \cref{result:approx_by_smooth_c_BV}, we are able to prove the analogous of the Gagliardo--Nirenberg--Sobolev inequality for the space $BV^{\alpha}(\R^{n})$.

\begin{theorem}[Gagliardo--Nirenberg--Sobolev inequality] \label{thm:GNS_immersion} 
Let $\alpha \in (0, 1)$ and $n \ge 2$. There exists a constant $c_{n, \alpha} > 0$ such that
\begin{equation} \label{eq:GNS_inequality} 
\|f\|_{L^{\frac{n}{n - \alpha}}(\R^{n})} \le c_{n, \alpha} |D^{\alpha} f|(\R^{n}) 
\end{equation}
for any $f \in BV^{\alpha}(\R^{n})$. As a consequence, $BV^{\alpha}(\R^{n})$ is continuously embedded in $L^q(\R^n)$ for any $q\in[1,\frac{n}{n-\alpha}]$.
\end{theorem}

\begin{proof}
By \cite{SSVanS17}*{Theorem A'}, we know that~\eqref{eq:GNS_inequality} holds for any $f \in C^{\infty}_{c}(\R^{n})$. So let $f \in BV^{\alpha}(\R^{n})$ and let $(f_k)_{k\in\N}\subset C^\infty_c(\R^n)$ be as in \cref{result:approx_by_smooth_c_BV}.  By Fatou's Lemma and \cref{result:frac_var_meas_is_lsc}, we thus obtain
\begin{equation*} 
\|f\|_{L^{\frac{n}{n - \alpha}}(\R^{n})} \le \liminf_{k \to + \infty} \|f_{k}\|_{L^{\frac{n}{n - \alpha}}(\R^{n})} \le c_{n, \alpha} \lim_{k \to + \infty} |D^{\alpha} f_{k}|(\R^{n}) = c_{n, \alpha} |D^{\alpha} f|(\R^{n})
\end{equation*}
and the proof is complete.
\end{proof}

\begin{remark}
We stress the fact that \cref{thm:GNS_immersion} does not hold for $n = 1$, as will be shown in \cref{rem:no_GNS_one_dim} below. It is worth to notice that an analogous restriction holds for~\cite{SSVanS17}*{Theorem A}, for which the authors provide a counterexample in the case $n = 1$ (see~\cite{SSVanS17}*{Counterexample 3.2}). The authors then derive~\cite{SSVanS17}*{Theorem A'} as a consequence of~\cite{SSVanS17}*{Theorem A}, without proving the necessity of the restriction to $n \ge 2$ in this second case, as we do in \cref{rem:no_GNS_one_dim}.
\end{remark}

\subsection{Coarea inequality} 

In analogy with the classical case, we can prove a coarea inequality formula for functions in $BV^{\alpha}(\R^{n})$.

\begin{theorem}[Coarea inequality] \label{th:coarea_inequality} 
Let $\alpha \in (0, 1)$. If $f \in BV^{\alpha}(\R^{n})$ is such that
\begin{equation}\label{eq:coarea_int_finite}
\int_{\R}|D^\alpha\chi_{\set{f>t}}|(\R^n)\,dt<+\infty,
\end{equation}
then
\begin{equation} \label{eq:weak_coarea_formula} 
D^{\alpha} f = \int_{\R} D^{\alpha} \chi_{\{ f > t \}} \, dt 
\end{equation}
and
\begin{equation} \label{eq:tot_var_coarea_inequality} 
|D^{\alpha} f| \le \int_{\R} |D^{\alpha} \chi_{\{f > t \}}| \, d t. 
\end{equation}
\end{theorem}

\begin{proof} 
Let $\phi \in C^{\infty}_{c}(\R^{n}; \R^{n})$. By~\eqref{eq:coarea_int_finite} and applying Fubini's Theorem twice, we can compute
\begin{align*} 
\int_{\R^{n}} \phi \cdot \, d D^{\alpha} f 
& = - \int_{\R^{n}} f\, \div^{\alpha} \phi(x) \, dx \\
& = - \int_{\R^{n}} \div^{\alpha}\phi(x) \left ( \int_\R \chi_{(- \infty, f(x))}(t) - \chi_{(- \infty, 0)}(t) \, dt \right )  dx \\
& = - \int_{\R} \int_{\R^{n}} \div^{\alpha} \phi(x) \left ( \chi_{\{f > t \}}(x) - \chi_{(- \infty, 0)}(t) \right ) \, dx \, dt \\
& = \int_{\R} \int_{\R^{n}} \phi\cdot \, d D^{\alpha} \chi_{\{f > t \}} \, dt \\
& = \int_{\R^{n}}\phi \cdot d\left(\int_{\R} D^{\alpha} \chi_{\{ f > t \}} \, dt\right) 
\end{align*}
proving~\eqref{eq:weak_coarea_formula}. Thus
\begin{equation*} 
|D^{\alpha} f |
=\abs*{\int_{\R} D^{\alpha} \chi_{\{f > t \}} \, d t}
\le\int_{\R} |D^{\alpha} \chi_{\{f > t \}}| \, d t 
\end{equation*}
and the proof is complete.
\end{proof}

\subsection{A fractional version of the Fundamental Theorem of Calculus}

Let $\alpha\in(0,1)$ and let $\mu_{n, -\alpha}$ be given by~\eqref{eq:def_mu_alpha} (note that the expression in~\eqref{eq:def_mu_alpha} makes sense for all $\alpha\in(-1,1)$). We let
\begin{equation}\label{eq:def_T_space}
\mathcal{T}(\R^n):=\set*{f\in C^\infty(\R^n): D^\mathsf{a} f\in L^1(\R^n)\cap C_0(\R^n)\text{ for all multi-indices } \mathsf{a}\in\N^n_0} 
\end{equation}
and 
\begin{equation*}
\mathcal{T}(\R^n;\R^n):=\set*{\phi\in C^\infty(\R^n;\R^n) : \phi_i\in\mathcal{T}(\R^n),\ i=1,\dots, n}.
\end{equation*}
By~\cite{S18}*{Section~5}, the operator
\begin{equation} \label{eq:div_negative_alpha_def}
\div^{- \alpha} \phi(x) := \mu_{n, - \alpha} \int_{\R^{n}} \frac{z\cdot \phi(x + z)}{|z|^{n + 1 - \alpha}}  \, dz
\end{equation}
is well defined for any $\phi \in \mathcal{T}(\R^{n}; \R^{n})$. Moreover, by \cite{S18}*{Theorem 5.3}, we have the following \emph{inversion formula} 
\begin{equation} \label{eq:div_-alpha_nabla_alpha_id}
- \div^{- \alpha} \nabla^{\alpha} =\mathrm{id}_{\mathcal{T}(\R^n)}.
\end{equation} 
Exploiting~\eqref{eq:div_negative_alpha_def} and~\eqref{eq:div_-alpha_nabla_alpha_id} we can prove the following fractional version of the Fundamental Theorem of Calculus. See~\cite{SS15}*{Theorem~2.1} for a similar approach.

\begin{theorem}[Fractional Fundamental Theorem of Calculus] \label{thm:fund_theorem_calculus_frac}
Let $\alpha \in (0, 1)$. If $f \in C^{\infty}_{c}(\R^{n})$, then
\begin{equation} \label{eq:fund_theorem_calculus_frac}
f(y) - f(x) = \mu_{n, - \alpha} \int_{\R^{n}} \left ( \frac{z-x}{|z-x|^{n + 1 - \alpha}} - \frac{z - y}{|z - y|^{n + 1 - \alpha}} \right ) \cdot \nabla^{\alpha} f(z) \, dz
\end{equation}
for any $x, y \in \R^{n}$.
\end{theorem}

\begin{proof} 
Since clearly $C^\infty_c(\R^n)\subset\mathcal{T}(\R^n)$, we have $\nabla^\alpha f\in\mathcal{T}(\R^n;\R^n)$ by~\cite{S18}*{Theorem~4.3}.  Applying~\eqref{eq:div_-alpha_nabla_alpha_id}, we have
\begin{align*}
f(y) - f(x) & = (- \div^{- \alpha} \nabla^{\alpha} f)( y) - (- \div^{- \alpha} \nabla^{\alpha} f)(x) \\
& = \mu_{n, - \alpha} \int_{\R^{n}} \frac{z}{|z|^{n + 1 - \alpha}} \cdot \big( \nabla^{\alpha} f(x + z) - \nabla^{\alpha} f(y + z) \big ) \, dz 
\end{align*}
for all $x, y \in \R^{n}$. Then~\eqref{eq:fund_theorem_calculus_frac} follows splitting the integral and changing variables. 
\end{proof}

An easy consequence of \cref{thm:fund_theorem_calculus_frac} is that the distributional $\alpha$-divergence of the kernel appearing in~\eqref{eq:fund_theorem_calculus_frac} is a difference of Dirac deltas.

\begin{proposition} \label{prop:div_alpha_delta}
Let $\alpha \in (0, 1)$. If $x, y \in \R^{n}$, then
\begin{equation} \label{eq:div_alpha_delta} 
\mu_{n, - \alpha} \div^{\alpha} \left ( \frac{\cdot - y}{|\cdot - y|^{n + 1 - \alpha}} - \frac{\cdot - x}{|\cdot - x|^{n + 1 - \alpha}} \right ) = \delta_{y} - \delta_{x}
\end{equation} 
in the sense of Radon measures.
\end{proposition}
\begin{proof} 
It follows immediately from~\eqref{eq:fund_theorem_calculus_frac}.
\end{proof}

\subsection{Compactness}

We start with the following H\"older estimate on the $L^1$-norm of translations of functions in $C^\infty_c(\R^n)$.

\begin{proposition}[$L^1$-estimate on translations] \label{prop:Holder_estimate} 
Let $\alpha \in (0,1)$. If $f \in C^{\infty}_{c}(\R^{n})$, then
\begin{equation} \label{eq:Holder_estimate}
\int_{\R^n}|f(x + y) - f(x)|\,dx \le \gamma_{n, \alpha}\, |y|^{\alpha}\, \|\nabla^{\alpha} f\|_{L^{1}(\R^{n}; \R^{n})}
\end{equation}
for all $y \in \R^{n}$, where
\begin{equation}\label{eq:def_gamma_n,alpha}
\gamma_{n, \alpha} := \mu_{n, - \alpha} \int_{\R^{n}} \left | \frac{z}{|z|^{n + 1 - \alpha}} - \frac{z - \mathrm{e}_{1}}{|z - \mathrm{e}_{1}|^{n + 1- \alpha}} \right | \, dz.
\end{equation}
\end{proposition}

\begin{proof}
By \eqref{eq:fund_theorem_calculus_frac}, we have
\begin{align*}
\int_{\R^n}|f(x + y) - f(x)|\,dx
& \le \mu_{n, - \alpha} \int_{\R^{n}} \int_{\R^{n}} \left | \frac{z}{|z|^{n + 1 - \alpha}} - \frac{z - y}{|z - y|^{n + 1- \alpha}} \right | |\nabla^{\alpha} f(x + z)| \, dz \, dx \\
& = \mu_{n, - \alpha} \|\nabla^{\alpha} f\|_{L^{1}(\R^{n}; \R^{n})} \int_{\R^{n}} \left | \frac{z}{|z|^{n + 1 - \alpha}} - \frac{z - y}{|z - y|^{n + 1- \alpha}} \right | \, dz. 
\end{align*}
Now we notice that the integral appearing in the last term is actually a radial function of~$y$. Indeed, let $\mathrm{R} \in {\rm SO}(n)$ be such that $\mathrm{R} y = |y| \nu$, for some $\nu \in \mathbb{S}^{n - 1}$. Making the change of variable $z = |y|\transp{\mathrm{R}}w$, we obtain
\begin{align*} 
\int_{\R^{n}} \left | \frac{z}{|z|^{n + 1 - \alpha}} - \frac{z - y}{|z - y|^{n + 1- \alpha}} \right | \, dz 
& = |y|^{\alpha} \int_{\R^{n}} \left | \frac{\transp{\mathrm{R}} w}{|w|^{n + 1 - \alpha}} - \frac{\transp{\mathrm{R}}(w - \nu)}{|w - \nu|^{n + 1- \alpha}} \right | \, dw \\
& = |y|^{\alpha} \int_{\R^{n}} \left | \frac{w}{|w|^{n + 1 - \alpha}} - \frac{(w - \nu)}{|w - \nu|^{n + 1- \alpha}} \right | \, dw.
\end{align*}
Since $\nu$ is arbitrary, we may choose $\nu = \mathrm{e}_{1}$. We now prove that 
\begin{equation*}
\int_{\R^{n}} \left | \frac{z}{|z|^{n + 1 - \alpha}} - \frac{z - \mathrm{e}_{1}}{|z - \mathrm{e}_{1}|^{n + 1- \alpha}} \right | \, dz < + \infty.
\end{equation*}
To this purpose, we notice that
\begin{align*}
\int_{B_{2}} \left | \frac{z}{|z|^{n + 1 - \alpha}} - \frac{z - \mathrm{e}_{1}}{|z - \mathrm{e}_{1}|^{n + 1- \alpha}} \right | \, dz & \le \int_{B_{2}} \frac{1}{|z|^{n - \alpha}} \, dz + \int_{B_{2}} \frac{1}{|z - \mathrm{e}_{1}|^{n - \alpha}} \, dz \\
& \le 2 \int_{B_{3}} \frac{1}{|z|^{n - \alpha}}\, dz = 2 n \omega_{n} \frac{3^{\alpha}}{\alpha}. 
\end{align*}
On the other hand, for all $z\in\R^n\setminus B_2$ we have
\begin{align*}
\frac{z - \mathrm{e}_{1}}{|z - \mathrm{e}_{1}|^{n + 1- \alpha}} - \frac{z}{|z|^{n + 1 - \alpha}} & = \int_{0}^{1} \frac{d}{dt} \left ( \frac{(z - t \mathrm{e}_{1})}{|z - t \mathrm{e}_{1}|^{n + 1 - \alpha}} \right) \, dt \\
& = \int_{0}^{1} - \frac{\mathrm{e}_{1}}{|z - t \mathrm{e}_{1}|^{n + 1 - \alpha}} + (n + 1 - \alpha) (z_{1} - t) \frac{(z - t \mathrm{e}_{1})}{|z - t \mathrm{e}_{1}|^{n + 3 - \alpha}} \, dt
\end{align*}
so that
\begin{align*}
\int_{\R^{n} \setminus B_{2}} \left | \frac{z}{|z|^{n + 1 - \alpha}} - \frac{z - \mathrm{e}_{1}}{|z - \mathrm{e}_{1}|^{n + 1- \alpha}} \right | \, dz & \le  \int_{\R^{n} \setminus B_{2}} \int_{0}^{1} \frac{|z - t \mathrm{e}_{1}|+  (n - \alpha +1)|z_{1} - t|}{|z - t \mathrm{e}_{1}|^{n + 2 - \alpha}} \, dt \, dz \\
& \le (n - \alpha + 2) \int_{0}^{1} \int_{\R^{n} \setminus B_{2}} \frac{1}{|z - t \mathrm{e}_{1}|^{n + 1 - \alpha}} \, dz \, dt \\
& \le (n - \alpha + 2) \int_{0}^{1} \int_{\R^{n} \setminus B_{1}} \frac{1}{|z|^{n + 1 - \alpha}} \, dz \, dt \\
& = (n - \alpha + 2)  \frac{n \omega_{n}}{1 - \alpha}.
\end{align*}
We conclude that
\begin{equation*}
\int_{\R^{n}} \left | \frac{z}{|z|^{n + 1 - \alpha}} - \frac{z - \mathrm{e}_{1}}{|z - \mathrm{e}_{1}|^{n + 1- \alpha}} \right | \, dz \le n \omega_{n} \left ( 2 \frac{3^{\alpha}}{\alpha} + \frac{(n - \alpha + 2)}{1 - \alpha} \right ) < + \infty.
\end{equation*}
Thus, the proof is complete.
\end{proof} 

Similarly to the classical case, as a consequence of the previous result we can prove the following key  estimate of the $L^1$-distance of a function in $BV^\alpha(\R^n)$ and its convolution with a mollifier.

\begin{corollary}[$L^1$-distance with convolution] \label{result:mollifier_conv_estimate_Holder}
Let $\alpha \in (0, 1)$. If $f \in BV^{\alpha}(\R^{n})$, then
\begin{equation} \label{eq:mollifier_conv_estimate_Holder}
\|\rho_{\eps} \ast f - f\|_{L^{1}(\R^{n})} 
\le \gamma_{n, \alpha}\, \eps^{\alpha} |D^{\alpha} f|(\R^{n})
\end{equation}
for all $\eps > 0$, where $(\rho_\eps)_{\eps>0}\subset C^\infty_c(\R^n)$ is as in~\eqref{eq:def_rho_eps} and $\gamma_{n, \alpha}$ as in \cref{prop:Holder_estimate}.
\end{corollary}

\begin{proof}
By \cref{result:approx_by_smooth_c_BV}, it is enough to prove~\eqref{eq:mollifier_conv_estimate_Holder} for $f \in C^{\infty}_{c}(\R^{n})$. By~\eqref{eq:Holder_estimate}, we get
\begin{align*}
\|\rho_{\eps} \ast f - f\|_{L^{1}(\R^{n})} 
& \le \int_{\R^{n}} \int_{\R^{n}} \rho(y) |f(x - \eps y) - f(x)| \, dy \, dx \\
& = \int_{\R^{n}} \rho(y)\int_{\R^{n}}  |f(x - \eps y) - f(x)| \, dx \, dy \\
& \le \gamma_{n, \alpha}\, \eps^{\alpha} \|\nabla^{\alpha}f\|_{L^{1}(\R^{n}; \R^{n})} \int_{B_1} \rho(y)  |y|^{\alpha} \, dy\\
& \le \gamma_{n, \alpha}\, \eps^{\alpha} \|\nabla^{\alpha}f\|_{L^{1}(\R^{n}; \R^{n})}
\end{align*}
and the proof is complete.
\end{proof}

We are now ready to prove following compactness result for the space $BV^\alpha(\R^n)$.

\begin{theorem}[Compactness for $BV^\alpha(\R^n)$]\label{result:compactness_BV_alpha}
Let $\alpha \in (0, 1)$. If $(f_k)_{k\in\N}\subset BV^\alpha(\R^n)$ satisfies
\begin{equation*}
\sup_{k\in\N}\|f_k\|_{BV^\alpha(\R^n)}<+\infty,
\end{equation*}
then there exists a subsequence $(f_{k_j})_{j\in\N}\subset BV^\alpha(\R^n)$ and a function $f\in L^1(\R^n)$ such that
\begin{equation*}
f_{k_j}\to f
\text{ in } L^1_{\loc}(\R^n)
\end{equation*}
as $j\to+\infty$.
\end{theorem}

\begin{proof}
We follow the line of the proof of~\cite{AFP00}*{Theorem~3.23}. Let $(\rho_\eps)_{\eps>0}\subset C^\infty_c(\R^n)$ be as in~\eqref{eq:def_rho_eps} and set $f_{k,\eps}:=\rho_\eps*f_k$. Clearly $f_{k,\eps}\in C^\infty(\R^n)$ and 
\begin{equation*}
\|f_{k,\eps}\|_{L^\infty(U)}\le\|\rho_\eps\|_{L^\infty(\R^n)} \|f_k\|_{L^1(\R^n)},
\qquad
\|\nabla f_{k,\eps}\|_{L^\infty(U; \R^{n})}\le\|\nabla\rho_{\eps}\|_{L^\infty(\R^n; \R^{n})}\|f_k\|_{L^1(\R^n)}
\end{equation*}
for any open set $U\Subset\R^n$. Thus $(f_{k,\eps})_{k\in\N}$ is locally equibounded and locally equicontinuous for each $\eps>0$ fixed. By a diagonal argument, we can find a sequence $(k_j)_{j\in\N}$ such that $(f_{k_j,\eps})_{j\in\N}$ converges in $C(U)$ for any open set $U\Subset\R^n$ with $\eps=1/p$ for all $p\in\N$. By \cref{result:mollifier_conv_estimate_Holder}, we thus get
\begin{equation*}
\begin{split}
\limsup_{h,j\to+\infty}\int_U|f_{k_h}-f_{k_j}|\,dx
&=\limsup_{h,j\to+\infty}\int_U|f_{{k_h},1/p}-f_{{k_j},1/p}|\,dx\\
&\quad+\limsup_{h,j\to+\infty}\int_U|f_{k_h}-f_{{k_h},1/p}|+|f_{k_j}-f_{{k_j},1/p}|\,dx\\
&\le\frac{2\gamma_{n,\alpha}}{p^\alpha}\,\sup_{k\in\N}|D^\alpha f_k|(\R^n)
\end{split}
\end{equation*}
for all open set $U\Subset\R^n$. Since $p\in\N$ is arbitrary and $L^1(U)$ is a Banach space, this shows that $(f_{k_j})_{j\in\N}$ converges in $L^1(U)$ for all open set $U\Subset\R^n$. Up to extract a further subsequence (which we do not relabel for simplicity), we also have that $f_{k_j}(x)\to f(x)$ for $\Leb{n}$-a.e.\ $x\in\R^n$. By Fatou's Lemma, we can thus infer that
\begin{equation*}
\|f\|_{L^1(\R^n)}
\le\liminf_{j\to+\infty}\|f_{k_j}\|_{L^1(\R^n)}
\le\sup_{k\in\N}\|f_k\|_{BV^\alpha(\R^n)}.
\end{equation*}
Hence $f\in L^1(\R^n)$ and the proof is complete. 
\end{proof}

\begin{remark}[Improvement of~\cite{SS15}*{Theorem~2.1}]
The argument presented above can be used to extend the validity of~\cite{SS15}*{Theorem~2.1} to all exponents $p\in[1,\frac{n}{\alpha})$, since our strategy does not rely on the boundedness of Riesz's transform but only on the inversion formula~\eqref{eq:div_-alpha_nabla_alpha_id}. We leave the details of the proof of this improvement of~\cite{SS15}*{Theorem~2.1} to the interested reader. 
\end{remark}

\subsection{The inclusion \texorpdfstring{$W^{\alpha,1}(\R^n)\subset BV^\alpha(\R^n)$}{Wˆ{alpha,1}(Rˆn) in BVˆalpha(Rˆn)}}

As in the classical case, fractional $BV$ functions naturally include fractional Sobolev functions. 

\begin{theorem}[$W^{\alpha,1}(\R^n)\subset BV^\alpha(\R^n)$]\label{result:Sobolev_subset_BV} 
Let $\alpha\in(0,1)$. If $f\in W^{\alpha,1}(\R^n)$ then $f\in BV^\alpha(\R^n)$, with
\begin{equation}\label{eq:BV_subset_Sobolev} 
|D^{\alpha} f|(\R^n) \le \mu_{n, \alpha} [ f ]_{W^{\alpha, 1}(\R^n)}
\end{equation}
and
\begin{equation}\label{eq:duality_Sobolev_frac}
\int_{\R^n}f\,\div^\alpha\phi\, dx=-\int_{\R^n}\phi\cdot\nabla^\alpha f\, dx
\end{equation}
for all $\phi\in\Lip_c(\R^n;\R^n)$, so that $D^\alpha f=\nabla^\alpha f\,\Leb{n}$. 

Moreover, if $f \in BV(\R^{n})$, then $f \in W^{\alpha, 1}(\R^{n})$ for any $\alpha \in (0, 1)$, with 
\begin{equation} \label{eq:W_alpha_norm_bound_BV}
\|f\|_{W^{\alpha,1}(\R^n)} \le c_{n,\alpha}\|f\|_{BV(\R^n)} 
\end{equation}
for some $c_{n, \alpha} > 0$ and
\begin{equation} \label{eq:weak_frac_grad_repr} 
\nabla^{\alpha} f(x) 
= \frac{\mu_{n, \alpha}}{n + \alpha - 1} \int_{\R^n} \frac{d D f(y)}{|y - x|^{n + \alpha - 1}} 
\end{equation}
for $\Leb{n}$-a.e.\ $x\in\R^n$.  
\end{theorem}

\begin{proof} 
Let $f\in W^{\alpha,1}(\R^n)$. For any $\phi\in\Lip_{c}(\R^n;\R^{n})$, by Lebesgue's Dominated Convergence Theorem, Fubini's Theorem and \cref{result:Sobolev_frac_enough}, and recalling~\eqref{eq:cancellation_kernel}, we can compute
\begin{align*} 
\int_{\R^{n}} f \div^{\alpha} \varphi \, dx  
& = \mu_{n, \alpha} \lim_{\eps \to 0} \int_{\R^{n}} \int_{\{|x -y| > \eps \}} f(x) \frac{(y - x) \cdot \varphi(y) }{|y - x|^{n + \alpha + 1}} \, dy \, dx \\
& = \mu_{n, \alpha} \lim_{\eps \to 0} \int_{\R^{n}} \int_{\{|x -y| > \eps \}} \varphi(y) \cdot \frac{(y - x) f(x) }{|y - x|^{n + \alpha + 1}} \, dx \, dy\\
& = \mu_{n, \alpha} \lim_{\eps \to 0} \int_{\R^{n}} \int_{\{|x -y| > \eps \}} \varphi(y) \cdot \frac{(y - x) (f(x) - f(y)) }{|y - x|^{n + \alpha + 1}} \, dx \, dy\\
& = - \int_{\R^{n}} \varphi(y) \cdot\nabla^\alpha f(y) \, dy.
\end{align*}
This proves~\eqref{eq:duality_Sobolev_frac}, so that $f\in BV^\alpha(\R^n)$. Inequality~\eqref{eq:BV_subset_Sobolev} follows as in \cref{result:Sobolev_frac_enough}.

Now let $f \in BV(\R^{n})$. We claim that $f \in W^{\alpha, 1}(\R^{n})$. Indeed, take $(f_k)_{k\in\N}\subset C^\infty(\R^n)\cap BV(\R^n)$ such that $f_k\to f$ in $L^1(\R^n)$ and $\|\nabla f_k\|_{L^1(\R^n; \R^{n})}\to|Df|(\R^n)$ as $k\to+\infty$ (for instance, see~\cite{EG15}*{Theorem~5.3}). Since $W^{1,1}(\R^n) \subset W^{\alpha,1}(\R^n)$ (the proof of this inclusion is similar to the one of~\cite{DiNPV12}*{Proposition~2.2}, for example), by Fatou's Lemma we get that
\begin{equation*}
\begin{split}
\|f\|_{W^{\alpha,1}(\R^n)}
&\le\liminf_{k\to+\infty}\|f_k\|_{W^{\alpha,1}(\R^n)}\\
&\le c_{n,\alpha}\liminf_{k\to+\infty}\|f_k\|_{W^{1,1}(\R^n)}\\
&=c_{n,\alpha}\lim_{k\to+\infty}(\|f_k\|_{L^1(\R^n)}+|Df_k|(\R^n))\\
&=c_{n,\alpha}\|f\|_{BV(\R^n)}.
\end{split}
\end{equation*} 
Since $|Df|(\R^n)<+\infty$, by \cref{result:riesz_kernel_estimate} the function in~\eqref{eq:weak_frac_grad_repr} is well defined in~$L^1_{\loc}(\R^n)$. Fix $\phi \in C^{\infty}_{c}(\R^n; \R^{n})$. By \cref{prop:frac_div_repr}, we can write
\begin{equation*}
\int_{\R^{n}} f(x) \, \div^{\alpha} \phi(x) \, dx  
= \frac{\mu_{n, \alpha}}{n + \alpha - 1} \int_{\R^{n}} \int_{\R^n} f(x)\,\frac{\div \phi(y)}{|y - x|^{n + \alpha-1}} \, dy \, dx.
\end{equation*}
Recalling \cref{result:riesz_kernel_estimate}, applying Fubini's Theorem twice and integrating by parts, we obtain
\begin{align*}
\int_{\R^{n}} \int_{\R^n} f(x)\,\frac{\div \phi(y)}{|y - x|^{n + \alpha-1}} \, dy \, dx 
& = \int_{\R^{n}} \int_{\R^n} f(x)\,\frac{\div_y \phi(x+y)}{|y|^{n + \alpha-1}} \, dy \, dx \\ 
& = \int_{\R^{n}} \int_{\R^n} f(x)\,\frac{\div_x \phi(x+y)}{|y|^{n + \alpha-1}} \, dy \, dx \\ 
& = \int_{\R^{n}} |y|^{1-n - \alpha}\int_{\R^n} f(x)\,\div\phi(x+y) \, dx \, dy \\ 
& = - \int_{\R^n} |y|^{1 - n - \alpha} \int_{\R^{n}} \phi(y + x) \cdot \, d Df(x) \, dy \\
& = - \int_{\R^{n}} \int_{\R^n} \frac{\phi(y)}{|y - x|^{n + \alpha -1}} \, dy \cdot \, d Df(x) \\
& = - \int_{\R^{n}} \phi(y) \cdot \int_{\R^n} \frac{d Df(x)}{|x-y|^{n + \alpha -1}} \, dy.
\end{align*}
Thus we conclude that
\begin{equation*}  
\int_{\R^{n}} f(x) \, \div^{\alpha} \phi(x) \, dx 
= 
- \frac{\mu_{n, \alpha}}{n + \alpha - 1}  \int_{\R^{n}} \phi(y) \cdot \int_{\R^n} \frac{d Df(x)}{|x-y|^{n + \alpha -1}} \, dy.
\end{equation*}
Recalling~\eqref{eq:duality_Sobolev_frac}, this proves~\eqref{eq:weak_frac_grad_repr} and the proof is complete.
\end{proof}

\subsection{The space \texorpdfstring{$S^{\alpha,p}(\R^n)$}{Sˆ{alpha,p}(Rˆn)} and the inclusion \texorpdfstring{$S^{\alpha,1}(\R^n)\subset BV^\alpha(\R^n)$}{Sˆ{alpha,1}(Rˆn) in BVˆalpha(Rˆn)}}
\label{subsec:space_S_alpha_p}
 
It is now tempting to approach fractional Sobolev spaces from a distributional point of view. Recalling \cref{prop:frac_div_repr}, we can give the following definition.

\begin{definition}[Weak $\alpha$-gradient]
Let $\alpha\in(0,1)$, $p\in[1,+\infty]$, $f\in L^p(\R^n)$. We say that $g\in L^1_{\loc}(\R^n;\R^n)$ is a \emph{weak $\alpha$-gradient} of $f$, and we write $g=\nabla^\alpha_w f$, if 
\begin{equation*}
\int_{\R^n}f\,\div^\alpha\phi\, dx
=-\int_{\R^n}g\cdot\phi\, dx
\end{equation*} 
for all $\phi\in C^\infty_c(\R^n;\R^n)$. 
\end{definition} 

For $\alpha\in(0,1)$ and $p\in[1,+\infty]$, we can thus introduce the \emph{distributional fractional Sobolev space} $(S^{\alpha,p}(\R^n),\|\cdot\|_{S^{\alpha,p}(\R^n)})$ letting
\begin{equation}\label{eq:def_distrib_frac_Sobolev}
S^{\alpha,p}(\R^n):=\set*{f\in L^p(\R^n) : \exists\, \nabla^\alpha_w f \in L^p(\R^n;\R^n)}
\end{equation}
and
\begin{equation}\label{eq:def_distrib_frac_Sobolev_norm}
\|f\|_{S^{\alpha,p}(\R^n)}:=\|f\|_{L^p(\R^n)}+\|\nabla^\alpha_w f\|_{L^p(\R^n; \R^{n})},
\qquad
\forall f\in S^{\alpha,p}(\R^n).
\end{equation}

We omit the standard proof of the following result.

\begin{proposition}[$S^{\alpha,p}$ is a Banach space]
Let $\alpha\in(0,1)$ and $p\in[1,+\infty]$. The space $(S^{\alpha,p}(\R^n),\|\cdot\|_{S^{\alpha,p}(\R^n)})$ is a Banach space.
\end{proposition}

We leave the proof of the following interpolation result to the reader.

\begin{lemma}[Interpolation]
Let $\alpha\in(0,1)$ and $p_1,p_2\in[1,+\infty]$, with $p_1\le p_2$. Then
\begin{equation*}
S^{\alpha,p_1}(\R^n)\cap S^{\alpha,p_2}(\R^n)\subset S^{\alpha,q}(\R^n)
\end{equation*}
with continuous embedding for all $q\in[p_1,p_2]$.
\end{lemma}

Taking advantage of the techniques developed in the study of the space $BV^\alpha(\R^n)$ above, we are able to prove the following approximation result.

\begin{theorem}[Approximation by $C^\infty\cap S^{\alpha,p}$ functions]\label{result:approx_by_smooth_S_alpha_p}
Let $\alpha\in(0,1)$ and $p\in[1,+\infty)$. The set $C^\infty(\R^n)\cap S^{\alpha,p}(\R^n)$ is dense in $S^{\alpha,p}(\R^n)$.
\end{theorem}

\begin{proof}
Let $(\rho_\eps)_{\eps>0}\subset C^\infty_c(\R^n)$ be as in~\eqref{eq:def_rho_eps}. Fix $f\in S^{\alpha,p}(\R^n)$ and consider $f_\eps:=f*\rho_\eps$ for all $\eps>0$. By \cref{result:commutation_mollifier}, it is easy to check that $f_\eps\in C^\infty(\R^n)\cap S^{\alpha,p}(\R^n)$ with $\nabla^\alpha_w f_\eps=\rho_\eps*\nabla^\alpha_w f$ for all $\eps>0$, so that the conclusion follows by standard properties of the convolution.
\end{proof}

Given $\alpha\in(0,1)$ and $p\in[1,+\infty]$, it is easy to see that, if $f\in C^\infty_c(\R^n)$, then, by \cref{result:duality}, $f\in S^{\alpha,p}(\R^n)$ with $\nabla^\alpha_w f=\nabla^\alpha f$. In the case $p=1$, we can prove that $C^\infty_c(\R^n)$ is also a dense subset of~$S^{\alpha,1}(\R^n)$.

\begin{theorem}[Approximation by $C^\infty_c$ functions]\label{result:approx_by_smooth_c_S_alpha_1}
Let $\alpha\in(0,1)$. The set $C^\infty_c(\R^n)$ is dense in $S^{\alpha,1}(\R^n)$.
\end{theorem}

\begin{proof}
Let $(\eta_R)_{R>0}\subset C^\infty_c(\R^n)$ be as in~\eqref{eq:def_cut-off}. Thanks to \cref{result:approx_by_smooth_S_alpha_p}, it is enough to prove that $f\eta_R\to f$ in $S^{\alpha,1}(\R^n)$ as $R\to+\infty$ for all $f\in C^\infty(\R^n)\cap S^{\alpha,1}(\R^n)$. Clearly, $f\eta_R\to f$ in $L^1(\R^n)$ as $R\to+\infty$. We now argue as in the proof of \cref{result:approx_by_smooth_c_BV}. Fix $\phi\in C^\infty_c(\R^n;\R^n)$. Then, by \cref{lem:Leibniz_frac_div}, we get
\begin{equation*}
\begin{split}
\int_{\R^n}f\eta_R\,\div^\alpha\phi\,dx
&=\int_{\R^n}f\,\div^\alpha(\eta_R\phi)\,dx
-\int_{\R^n}f\,\phi\cdot\nabla^\alpha\eta_R\,dx
-\int_{\R^n}f\,\div^\alpha_{\mathrm{NL}}(\eta_R, \phi)\,dx.
\end{split}
\end{equation*}
Since $f\in S^{\alpha,1}(\R^n)$, we have
\begin{equation*}
\int_{\R^n}f\,\div^\alpha(\eta_R\phi)\,dx
=-\int_{\R^n}\eta_R\phi\cdot\nabla^\alpha_w f\,dx.
\end{equation*}
Since $f\eta_R\in C^\infty_c(\R^n)$, we also have
\begin{equation*}
\int_{\R^n}f\eta_R\,\div^\alpha\phi\,dx
=-\int_{\R^n}\phi\cdot\nabla^\alpha(\eta_R f)\,dx.
\end{equation*}
Thus we can write
\begin{equation*}
\begin{split}
\int_{\R^n}(\nabla^\alpha_w f-\nabla^\alpha(\eta_R f))\cdot\phi\,dx
& = \int_{\R^n}(1-\eta_R)\phi\cdot\nabla^\alpha_w f\,dx\\
&\quad-\int_{\R^n}f\,\phi\cdot\nabla^\alpha\eta_R\,dx
-\int_{\R^n}f\,\div^\alpha_{\mathrm{NL}}(\eta_R, \phi)\,dx.
\end{split}
\end{equation*}
Moreover, we have
\begin{equation*}
\abs*{\int_{\R^n}f\,\phi\cdot\nabla^\alpha\eta_R\,dx}
\le 
\mu_{n, \alpha} \|\phi\|_{L^\infty(\R^n;\R^n)}\int_{\R^n}|f(x)|\int_{\R^n}\frac{|\eta_R(y)-\eta_R(x)|}{|y-x|^{n+\alpha}}\,dy\, dx
\end{equation*}
and, similarly,
\begin{equation*}
\abs*{\int_{\R^n}f\,\div^\alpha_{\mathrm{NL}}(\eta_R, \phi)\,dx}
\le 
2\mu_{n, \alpha}\|\phi\|_{L^\infty(\R^n;\R^n)}\int_{\R^n}|f(x)|\int_{\R^n}\frac{|\eta_R(y)-\eta_R(x)|}{|y-x|^{n+\alpha}}\,dy\,dx.
\end{equation*}
Combining these two estimates, we get that
\begin{equation*}
\begin{split}
\left | \int_{\R^n}(\nabla^\alpha_w f-\nabla^\alpha(\eta_R f))\cdot\phi\,dx\right |
&\le
\|\phi\|_{L^\infty(\R^n;\R^n)}\int_{\R^n}(1-\eta_R)|\nabla^\alpha_w f|\,dx\\
&\quad+3\mu_{n, \alpha}\|\phi\|_{L^\infty(\R^n;\R^n)}\int_{\R^n}|f(x)|\int_{\R^n}\frac{|\eta_R(y)-\eta_R(x)|}{|y-x|^{n+\alpha}}\,dy\,dx.
\end{split}
\end{equation*}
We thus conclude that 
\begin{align*}
\|\nabla^\alpha_w f-\nabla^\alpha(\eta_R f)\|_{L^1(\R^n; \R^n)}
\le & \int_{\R^n}(1-\eta_R)|\nabla^\alpha_w f|\,dx \\
& +3\mu_{n, \alpha}\int_{\R^n}|f(x)|\int_{\R^n}\frac{|\eta_R(y)-\eta_R(x)|}{|y-x|^{n+\alpha}}\,dy\,dx.
\end{align*}
Therefore $\nabla^\alpha(\eta_R f)\to\nabla^\alpha_w f$ in $L^1(\R^n; \R^{n})$ as $R\to+\infty$. Indeed, we have
\begin{equation*}
\lim_{R\to+\infty}\int_{\R^n}(1-\eta_R)|\nabla^\alpha_w f|\,dx=0
\end{equation*}
combining~\eqref{eq:def_cut-off} with Lebesgue's Dominated Convergence Theorem and
\begin{equation*}
\lim_{R\to+\infty}\int_{\R^n}|f(x)|\int_{\R^n}\frac{|\eta_R(y)-\eta_R(x)|}{|y-x|^{n+\alpha}}\,dy\,dx=0
\end{equation*}
combining~\eqref{eq:Lip_nabla_estim_1}, \eqref{eq:Lip_nabla_estim_2} and~\eqref{eq:def_cut-off} with Lebesgue's Dominated Convergence Theorem.
\end{proof}

We do not know if $C^\infty_c(\R^n)$ is a dense subset of $S^{\alpha,p}(\R^n)$ for $\alpha\in(0,1)$ and $p\in(1,+\infty)$. In other words, defining
\begin{equation*}
S^{\alpha,p}_0(\R^n):=\closure[-1]{C^\infty_c(\R^n)}^{\|\cdot\|_{S^{\alpha,p}(\R^n)}},
\end{equation*}
we do not know if the (continuous) inclusion $S^{\alpha,p}_0(\R^n)\subset S^{\alpha,p}(\R^n)$ is strict. 

The space $(S^{\alpha,p}_0(\R^n),\|\cdot\|_{S^{\alpha,p}(\R^n)})$ was introduced in~\cite{SS15} (with a different, but equivalent, norm). Thanks to~\cite{SS15}*{Theorem 1.7}, for all $\alpha \in (0, 1)$ and $p \in (1, + \infty)$ we have $S^{\alpha,p}_0(\R^n) = L^{\alpha, p}(\R^{n})$, where $L^{\alpha, p}(\R^{n})$ is the \emph{Bessel potential space}, see~\cite{SS15}*{Definition~2.1}. It is known that $L^{\alpha+\eps,p}(\R^n)\subset W^{\alpha, p}(\R^{n})\subset L^{\alpha-\eps,p}(\R^n)$ with continuous embeddings for all $\alpha\in(0,1)$, $p\in(1,+\infty)$ and $0<\eps<\min\set*{\alpha,1-\alpha}$, see~\cite{SS15}*{Theorem~2.2}. In the particular case $p=2$, it holds that $L^{\alpha, 2}(\R^{n})=W^{\alpha, 2}(\R^{n})$ for all $\alpha \in (0, 1)$, see~\cite{SS15}*{Theorem 2.2}. In addition, $W^{\alpha,p}(\R^n)\subset L^{\alpha,p}(\R^n)$ with continuous embedding for all $\alpha\in(0,1)$ and $p\in(1,2]$, see~\cite{S70}*{Chapter~V, Section~5.3}.

\begin{proposition}[Relation between $W^{\alpha,p}$ and $S^{\alpha,p}$]
The following properties hold.
\begin{enumerate}[\indent (i)]
\item\label{item:embedding_W_S_1} If $\alpha\in(0,1)$ and $p\in[1,2]$, then $W^{\alpha,p}(\R^n)\subset S^{\alpha,p}(\R^n)$ with continuous embedding.
\item\label{item:embedding_W_S_2} If $0<\alpha<\beta<1$ and $p\in(2,+\infty]$, then $W^{\beta,p}(\R^n)\subset S^{\alpha,p}(\R^n)$ with continuous embedding.
\end{enumerate}
\end{proposition}

\begin{proof}
Property~\eqref{item:embedding_W_S_1} follows from the discussion above for the case $p\in(1,2]$ and from \cref{result:Sobolev_subset_BV} for the case $p=1$. Property~\eqref{item:embedding_W_S_2} follows from the discussion above for the case $p\in(2,+\infty)$, while for the case $p=+\infty$ it is enough to observe that
\begin{equation*}
\begin{split}
\|\nabla^\alpha f\|_{L^\infty(\R^n; \R^{n})}
&\le \mu_{n, \alpha} \sup_{x\in\R^n}\int_{\R^n}\frac{|f(y)-f(x)|}{|y-x|^{n+\alpha}}\,dy\\
&\le2 \mu_{n, \alpha} \|f\|_{L^\infty(\R^n)}\int_{\set{|y|>1}}\frac{dy}{|y|^{n+\alpha}}
+ \mu_{n, \alpha} [f]_{W^{\beta,\infty}(\R^n)}\int_{\set{|y|\le1}}\frac{dy}{|y|^{n+\alpha-\beta}}\\
&\le c_{n,\alpha,\beta}\|f\|_{W^{\beta,\infty}(\R^n)} 
\end{split}
\end{equation*}
for all $f\in W^{\beta,\infty}(\R^n)$.
\end{proof}

As in the classical case, we have $S^{\alpha,1}(\R^n)\subset BV^\alpha(\R^n)$ with continuous embedding.

\begin{theorem}[$S^{\alpha,1}(\R^n)\subset BV^\alpha(\R^n)$]\label{result:S_alpha_1_in_BV_alpha}
Let $\alpha\in(0,1)$. If $f\in BV^\alpha(\R^n)$, then $f\in S^{\alpha,1}(\R^n)$ if and only if $|D^\alpha f|\ll\Leb{n}$, in which case 
\begin{equation*}
D^\alpha f=\nabla^\alpha_w f\,\Leb{n} 
\quad
\text{in $\M(\R^n;\R^n)$}.
\end{equation*}
\end{theorem}

\begin{proof}
Let $f\in BV^\alpha(\R^n)$ and assume that $|D^\alpha f|\ll\Leb{n}$. Then $D^\alpha f=g\,\Leb{n}$ for some $g\in L^1(\R^n;\R^n)$. But then, by \cref{th:structure_BV_alpha}, we must have
\begin{equation*}
\int_{\R^n}f\,\div^\alpha\phi\,dx=-\int_{\R^n} g\cdot\phi\,dx 
\end{equation*}
for all $\phi\in C^\infty_c(\R^n;\R^n)$, so that $f\in S^{\alpha,1}(\R^n)$ with $\nabla^\alpha_w f=g$. Viceversa, if $f\in S^{\alpha,1}(\R^n)$ then
\begin{equation*}
\int_{\R^n}f\,\div^\alpha\phi\,dx=-\int_{\R^n} \phi\cdot\nabla^\alpha_w f\,dx
\end{equation*}
for all $\phi\in C^\infty_c(\R^n;\R^n)$, so that $f\in BV^\alpha(\R^n)$ with $D^\alpha f=\nabla^\alpha_w f\,\Leb{n}$ in $\M(\R^n;\R^n)$.
\end{proof}

\subsection{The inclusion \texorpdfstring{$S^{\alpha,1}(\R^n)\subset BV^\alpha(\R^n)$}{Sˆ{alpha,1}(Rˆn) in BVˆalpha(Rˆn)} is strict}

It seems natural to ask whether the inclusion $S^{\alpha,1}(\R^n)\subset BV^\alpha(\R^n)$ is strict as in the classical case. We start to solve this problem in the case $n = 1$.

\begin{theorem}[$BV^{\alpha}(\R)\setminus S^{\alpha, 1}(\R)\neq\varnothing$] \label{prop:BV_alpha_W_alpha_inclusion_R}
Let $\alpha \in (0, 1)$. The inclusion $S^{\alpha, 1}(\R) \subset BV^{\alpha}(\R)$ is strict, since for any $a, b \in \R$, with $a \neq b$, the function 
\begin{equation*}
f_{a, b, \alpha}(x) := |x - b|^{\alpha - 1} \sgn(x - b) - |x - a|^{\alpha - 1} \sgn(x - a)
\end{equation*}
satisfies $f_{a, b, \alpha} \in BV^{\alpha}(\R)$ with 
\begin{equation} \label{eq:strict_inclusion_example_nabla_1}
D^{\alpha} f_{a, b, \alpha} = \frac{\delta_{b} - \delta_{a}}{\mu_{1, - \alpha}}
\end{equation}
in the sense of finite Radon measures.
\end{theorem}

\begin{proof} 
Let $a, b \in \R$ be fixed with $a \neq b$. One can easily check that $f_{a, b, \alpha} \in L^{1}(\R)$. Since $n=1$, we have $\nabla^{\alpha}=\div^{\alpha}$. Thus, \eqref{eq:strict_inclusion_example_nabla_1} follows from~\eqref{eq:div_alpha_delta}, proving that $f\in BV^\alpha(\R)$. But $|D^\alpha f_{a, b, \alpha}|\perp\Leb{1}$, so that $f_{a, b, \alpha}\notin S^{\alpha,1}(\R)$ by \cref{result:S_alpha_1_in_BV_alpha}.
\end{proof}

\begin{remark} \label{rem:no_GNS_one_dim}
Note that $f_{a, b, \alpha} \in BV^{\alpha}(\R)\setminus L^{\frac{1}{1 - \alpha}}(\R)$, since
\begin{equation*}
|f_{a, b, \alpha}(x)|^{\frac{1}{1 - \alpha}} 
\sim
\begin{cases}
|x - a|^{-1} & \text{as $x \to a$},\\[3mm]
|x - b|^{-1} & \text{as $x \to b$}.
\end{cases}
\end{equation*}
Thus, \cref{thm:GNS_immersion} cannot hold for $n = 1$.
\end{remark}

For the case $n>1$, we need to recall the definition of the \emph{fractional Laplacian operator~$(- \Delta)^{\frac{\alpha}{2}}$} and some of its properties. 

Following~\cite{S18}, for any $f \in C^{\infty}_{c}(\R^{n})$ we set
\begin{equation} \label{eq:fract_Laplacian}
(- \Delta)^{\frac{\alpha}{2}} f(x) := 
\begin{cases} 
\displaystyle \nu_{n, \alpha} \int_{\R^{n}} \frac{f(x + h)}{|h|^{n + \alpha}} \, dh & \text{if} \ \alpha \in (- 1, 0), \\[5mm]
f(x) & \text{if} \ \alpha = 0, \\[5mm]
\displaystyle \nu_{n, \alpha} \int_{\R^{n}} \frac{f(x + h) - f(x)}{|h|^{n + \alpha}}\,dh & \text{if} \ \alpha \in (0, 1), \\[5mm]
\displaystyle \nu_{n, \alpha} \lim_{\eps \to 0} \int_{\{ |h| > \eps \}} \frac{f(x + h) - f(x)}{|h|^{n + \alpha}}\,dh & \text{if} \ \alpha \in [1, 2),
\end{cases}
\end{equation}
where
\begin{equation} \label{eq:nu_alpha_constant}
\nu_{n, \alpha} := 2^{\alpha} \pi^{- \frac{n}{2}} \frac{\Gamma \left ( \frac{n + \alpha}{2} \right )}{\Gamma \left ( - \frac{ \alpha}{2} \right )}.
\end{equation}

We stress the fact that this definition is consistent with the previous definitions of fractional gradient and divergence in the sense that
\begin{equation*}
-\div^{\alpha} \nabla^{\beta} = (- \Delta)^{\frac{\alpha + \beta}{2}}
\end{equation*}
for any $\alpha \in (-1, 1)$ and $\beta \in (0, 1)$ (see \cite{S18}*{Theorem 5.3}), so that, in particular,
\begin{equation*}
-\div^{\alpha} \nabla^{\alpha} = (- \Delta)^{\alpha}
\end{equation*}
for any $\alpha \in (0, 1)$.

In the case $\alpha\in(-1,0)$, 
we have 
\begin{equation*}
(- \Delta)^{\frac{\alpha}{2}}=I_{-\alpha}
\quad
\text{on $C^\infty_c(\R^n)$},
\end{equation*}
where $I_\alpha$ is as in~\eqref{eq:Riesz_potential_def}.

In the case $\alpha\in(0,1)$, notice that
\begin{equation} \label{eq:fract_Laplacian_L_1}
\| (- \Delta)^{\frac{\alpha}{2}} f \|_{L^{1}(\R^{n})} \le \nu_{n, \alpha} [f]_{W^{\alpha, 1}(\R^{n})}
\end{equation}
for all $f \in C^{\infty}_{c}(\R^{n})$. Thus the linear operator 
\begin{equation*}
(- \Delta)^{\frac{\alpha}{2}}\colon C^{\infty}_{c}(\R^{n}) \to L^1(\R^n) 
\end{equation*}
can be continuously extended to a linear operator 
\begin{equation*}
(- \Delta)^{\frac{\alpha}{2}} \colon W^{\alpha, 1}(\R^{n}) \to L^{1}(\R^{n}),
\end{equation*}
for which we retain the same notation. 

Given $\alpha\in(0,1)$ and $\eps>0$, for all $f\in W^{\alpha,1}(\R^n)$ we also set
\begin{equation*}
(- \Delta)^{\alpha/2}_\eps f(x):=
\nu_{n, \alpha} \int_{\set{|h|>\eps}} \frac{f(x + h) - f(x)}{|h|^{n + \alpha}}\,dh.
\end{equation*}
By Lebesgue's Dominate Convergence Theorem, we have that
\begin{equation*}
\lim_{\eps\to0}\|(- \Delta)^{\alpha/2}_\eps f-(- \Delta)^{\frac{\alpha}{2}}f\|_{L^1(\R^n)}=0
\end{equation*}
for all $f\in W^{\alpha,1}(\R^n)$. Thus, arguing as in the proof of~\cite{S98}*{Lemma~2.4} (see also~\cite{SKM93}*{Section~25.1}), for all $f\in W^{\alpha,1}(\R^n)$ we have 
\begin{equation}\label{eq:riesz_inversion_laplacian_sobolev}
I_{\alpha}(- \Delta)^{\frac{\alpha}{2}}f= f 
\quad
\text{in $L^1(\R^n)$}.
\end{equation}

Taking advantage of the identity in~\eqref{eq:riesz_inversion_laplacian_sobolev}, we can prove the following result. 

\begin{lemma}[Relation between $BV^\alpha(\R^n)$ and $bv(\R^n)$]\label{result:correspondence_BV_alpha_bv}
Let $\alpha\in(0,1)$. The following properties hold.
\begin{enumerate}[\indent (i)]
\item\label{item:BV_alpha_bv_1} If $f\in BV^\alpha(\R^n)$, then $u:=I_{1-\alpha}f\in bv(\R^n)$ with $Du=D^\alpha f$ in $\M(\R^{n}; \R^{n})$.
\item\label{item:BV_alpha_bv_2} If $u\in BV(\R^n)$, then $f:= (-\Delta)^{\frac{1-\alpha}{2}}u\in BV^\alpha(\R^n)$ with 
\begin{equation*}
\|f\|_{L^1(\R^n)}\le c_{n,\alpha}\|u\|_{BV(\R^n)}
\quad\text{and}\quad
D^\alpha f=D u
\quad\text{in $\M(\R^{n}; \R^{n})$}.
\end{equation*}
As a consequence, the operator $(-\Delta)^{\frac{1-\alpha}{2}}\colon BV(\R^n)\to BV^\alpha(\R^n)$ is continuous.
\end{enumerate}
\end{lemma}

\begin{proof}
We prove the two properties separately.

\smallskip

\textit{Proof of~\eqref{item:BV_alpha_bv_1}}. Let $f\in BV^\alpha(\R^n)$. Since $f \in L^{1}(\R^{n})$, we have $I_{1 - \alpha} f \in L^{1}_{\rm loc}(\R^{n})$. By Fubini's Theorem, for any $\phi\in C^\infty_c(\R^n;\R^n)$ we have
\begin{equation} \label{eq:proof_correspondence_BV_alpha_bv}
\int_{\R^n}f\,\div^\alpha\phi\,dx
=\int_{\R^n}f\,I_{1-\alpha}\div\phi\,dx
=\int_{\R^n}u\,\div\phi\,dx,
\end{equation}
proving that $u:=I_{1-\alpha}f\in bv(\R^n)$ with $Du=D^\alpha f$ in $\M(\R^{n}; \R^{n})$.

\smallskip

\textit{Proof of~\eqref{item:BV_alpha_bv_2}}. Let $u\in BV(\R^n)$. By \cref{result:Sobolev_subset_BV}, we know that $u\in W^{1-\alpha,1}(\R^n)$, so that $f:= (-\Delta)^{\frac{1-\alpha}{2}}u\in L^1(\R^n)$ with $\|f\|_{L^1(\R^n)}\le c_{n,\alpha}\|u\|_{BV(\R^n)}$ by~\eqref{eq:W_alpha_norm_bound_BV} and~\eqref{eq:fract_Laplacian_L_1}. Then, arguing as before, for any $\phi\in C^\infty_c(\R^n;\R^n)$ we get \eqref{eq:proof_correspondence_BV_alpha_bv}, since we have $I_{1-\alpha}f=u$ in~$L^1(\R^n)$ by \eqref{eq:riesz_inversion_laplacian_sobolev}. The proof is complete. 
\end{proof}

\begin{remark}[Integrability issues]
Note that the inclusion $I_{1-\alpha}(BV^\alpha(\R^n))\subset L^1_{\loc}(\R^n)$ in \cref{result:correspondence_BV_alpha_bv} above is sharp. Indeed, by Tonelli's Theorem it is easily seen that $I_{1-\alpha}\chi_E\notin L^1(\R^n)$ whenever $\chi_E\in W^{\alpha,1}(\R^n)$. However, when $n\ge2$, by \cref{thm:GNS_immersion} and by Hardy--Littlewood--Sobolev inequality (see~\cite{S70}*{Chapter~V, Section~1.2} for instance), the map $I_{1-\alpha}\colon BV^\alpha\to L^p(\R^n)$ is continuous for each $p\in\left(\frac{n}{n-1+\alpha},\frac{n}{n-1}\right]$.
\end{remark}

As a consequence of \cref{result:correspondence_BV_alpha_bv}, we can prove that the inclusion $S^{\alpha, 1}(\R^{n})\subset BV^{\alpha}(\R^{n})$ is strict for all $\alpha\in(0,1)$ and $n\ge1$. 

\begin{theorem}[$BV^{\alpha}(\R^{n}) \setminus S^{\alpha, 1}(\R^{n}) \neq \varnothing$] \label{thm:strict_inclusion_W_BV_alpha}
Let $\alpha \in (0, 1)$. The inclusion $S^{\alpha, 1}(\R^{n}) \subset BV^{\alpha}(\R^{n})$ is strict.
\end{theorem}

\begin{proof}
Let $u \in BV(\R^n)\setminus W^{1,1}(\R^n)$. By \cref{result:correspondence_BV_alpha_bv}, we know that $f:=(-\Delta)^{\frac{1-\alpha}{2}}u\in BV^\alpha(\R^n)$ with $Du=D^\alpha f$ in $\M(\R^{n}; \R^{n})$. But then $|D^\alpha f|$ is not absolutely continuous with respect to $\Leb{n}$, so that $f\notin S^{\alpha, 1}(\R^{n})$ by \cref{result:S_alpha_1_in_BV_alpha}.
\end{proof}

\subsection{The inclusion \texorpdfstring{$W^{\alpha, 1}(\R^{n}) \subset S^{\alpha,1}(\R^{n})$}{Wˆ{alpha,1}(Rˆn) in Sˆ{alpha,1}(Rˆn)} is strict}

By \cref{thm:strict_inclusion_W_BV_alpha}, we know that the inclusion $W^{\alpha, 1}(\R^{n}) \subset BV^{\alpha}(\R^{n})$ is strict. In the following result we prove that also the inclusion $W^{\alpha, 1}(\R^{n}) \subset S^{\alpha,1}(\R^{n})$ is strict. 

\begin{theorem}[$S^{\alpha,1}(\R^{n}) \setminus W^{\alpha, 1}(\R^{n}) \neq \varnothing$] \label{thm:strict_inclusion_W_S_alpha_1}
Let $\alpha \in (0, 1)$. The inclusion $W^{\alpha, 1}(\R^{n}) \subset S^{\alpha,1}(\R^{n})$ is strict.
\end{theorem}

\begin{proof}
We argue by contradiction. If $W^{\alpha, 1}(\R^{n}) = S^{\alpha,1}(\R^{n})$, then the inclusion map $W^{\alpha, 1}(\R^{n})\hookrightarrow S^{\alpha,1}(\R^{n})$ is a linear and continuous bijection. Thus, by the Inverse Mapping Theorem, there must exist a constant $C>0$ such that
\begin{equation}\label{eq:contradiction_S_W}
[g]_{W^{\alpha,1}(\R^n)}\le C\|g\|_{S^{\alpha,1}(\R^n)}
\end{equation}
for all $g \in S^{\alpha, 1}(\R^{n})$. Now let $f\in BV^{\alpha}(\R^{n}) \setminus S^{\alpha, 1}(\R^{n})$ be given by \cref{thm:strict_inclusion_W_BV_alpha}. By \cref{result:approx_by_smooth_c_BV}, there exists $(f_k)_{k\in\N}\subset C^\infty_c(\R^n)$ such that $f_k\to f$ in $L^1(\R^n)$ and $|D^\alpha f_k|(\R^n)\to|D^\alpha f|(\R^n)$ as $k\to+\infty$. Up to extract a subsequence (which we do not relabel for simplicity), we can assume that $f_k(x)\to f(x)$ as $k\to+\infty$ for $\Leb{n}$-a.e.\ $x\in\R^n$. By~\eqref{eq:contradiction_S_W} and Fatou's Lemma, we have that
\begin{equation*}
\begin{split}
[f]_{W^{\alpha,1}(\R^n)}
&\le\liminf_{k\to+\infty}[f_k]_{W^{\alpha,1}(\R^n)}\\
&\le C\liminf_{k\to+\infty}\|f_k\|_{S^{\alpha,1}(\R^n)}\\
&= C\lim_{k\to+\infty}\|f_k\|_{BV^{\alpha}(\R^n)}\\
&=C\,\|f\|_{BV^{\alpha}(\R^n)}<+\infty.
\end{split}
\end{equation*} 
Therefore $f\in W^{\alpha,1}(\R^n)$, in contradiction with \cref{thm:strict_inclusion_W_BV_alpha}. We thus must have that the inclusion map $W^{\alpha, 1}(\R^{n})\hookrightarrow S^{\alpha,1}(\R^{n})$ cannot be surjective.
\end{proof}

\subsection{The inclusion \texorpdfstring{$BV^{\alpha}(\R^{n})\subset W^{\beta, 1}(\R^{n})$}{BVˆalpha(Rˆn) in Wˆ{beta,1}(Rˆn)} for \texorpdfstring{$\beta<\alpha$}{beta<alpha}}

Even though the inclusion $W^{\alpha, 1}(\R^{n}) \subset BV^{\alpha}(\R^{n})$ is strict, it is interesting to notice that $BV^{\alpha}(\R^{n})\subset W^{\beta, 1}(\R^{n})$ for all $0<\beta<\alpha < 1$ with continuous embedding.

\begin{theorem}[$BV^{\alpha}(\R^{n}) \subset W^{\beta, 1}(\R^{n})$ for $\beta < \alpha$] \label{result:BV_alpha_W_beta_embedding}
Let $\alpha,\beta\in(0,1)$ with $\beta<\alpha$. Then $BV^{\alpha}(\R^{n}) \subset W^{\beta, 1}(\R^{n})$, with
\begin{equation} \label{eq:BV_alpha_W_beta_embedding}
[f]_{W^{\beta, 1}(\R^{n})} \le C_{n, \alpha, \beta} \, \|f\|_{BV^{\alpha}(\R^{n})},
\end{equation}
for all $f\in BV^{\alpha}(\R^{n})$, where 
\begin{equation} \label{eq:BV_alpha_W_beta_embedding_constant}
C_{n, \alpha, \beta} := n \omega_{n}  \frac{\alpha2^{\frac{\alpha - \beta}{\beta}} \gamma_{n, \alpha}^{\beta/\alpha}}{\beta (\alpha - \beta)}
\end{equation}
and $\gamma_{n,\alpha}$ is as in~\eqref{eq:def_gamma_n,alpha}.
\end{theorem}

\begin{proof}
Let $f \in C^{\infty}_{c}(\R^{n})$ and $r>0$. By~\eqref{eq:Holder_estimate}, we get
\begin{align*}
[f]_{W^{\beta, 1}(\R^{n})} 
& = \int_{\R^{n}} \int_{\R^{n}} \frac{|f(x + y) - f(x)|}{|y|^{n + \beta}} \, dx \, dy\\
& \le \int_{\R^{n}} \frac{1}{|y|^{n + \beta}} \left ( 2 \|f\|_{L^{1}(\R^{n})} \chi_{\R^{n} \setminus B_{r}}(y) + \gamma_{n, \alpha} |y|^{\alpha} \|\nabla^{\alpha} f\|_{L^{1}(\R^{n}; \R^{n})} \chi_{B_{r}}(y) \right ) \, dy \\
& = 2 \frac{n \omega_{n}}{\beta} r^{- \beta} \|f\|_{L^{1}(\R^{n})} + \frac{n \omega_{n}}{\alpha - \beta} \gamma_{n, \alpha} r^{\alpha - \beta} \|\nabla^{\alpha} f\|_{L^{1}(\R^{n}; \R^{n})} \\
& \le \left ( 2 \frac{n \omega_{n}}{\beta} r^{- \beta} + \frac{n \omega_{n}}{\alpha - \beta} \gamma_{n, \alpha} r^{\alpha - \beta} \right ) \| f\|_{BV^{\alpha}(\R^{n})},
\end{align*}
so that both~\eqref{eq:BV_alpha_W_beta_embedding} and~\eqref{eq:BV_alpha_W_beta_embedding_constant} are proved by minimising in $r > 0$ for all $f\in C^\infty_c(\R^n)$. Now let $f\in BV^\alpha(\R^n)$. By \cref{result:approx_by_smooth_c_BV}, there exists $(f_k)_{k\in\N}\subset C^\infty_c(\R^n)$ such that $\|f_k\|_{BV^\alpha(\R^n)}\to\|f\|_{BV^\alpha(\R^n)}$ and $f_k\to f$ a.e.\ as $k\to+\infty$. Thus, by Fatou's Lemma, we get that
\begin{equation*}
[f]_{W^{\beta, 1}(\R^{n})}
\le\liminf_{k\to+\infty}\, [f_k]_{W^{\beta, 1}(\R^{n})}
\le\lim_{k\to+\infty} C_{n, \alpha, \beta}\, \|f_k\|_{BV^{\alpha}(\R^{n})}
=C_{n, \alpha, \beta}\, \|f\|_{BV^{\alpha}(\R^{n})}
\end{equation*}
and the conclusion follows.
\end{proof}

\noindent
Note that the constant in~\eqref{eq:BV_alpha_W_beta_embedding_constant} satisfies
\begin{equation*}
\lim_{\beta\to\alpha^-}C_{n, \alpha, \beta}=+\infty,
\end{equation*}
accordingly to the strict inclusion $W^{\alpha, 1}(\R^{n}) \subset BV^{\alpha}(\R^{n})$. In particular, the function in \cref{prop:BV_alpha_W_alpha_inclusion_R} is such that $f_{a, b, \alpha} \in W^{\beta, 1}(\R)$ for all $\beta \in (0, \alpha)$.

As an immediate consequence of  \cref{result:BV_alpha_W_beta_embedding}, we have the following result.

\begin{corollary}
Let $0<\beta<\alpha<1$. Then $BV^\alpha(\R^n)\subset BV^\beta(\R^n)$ and $S^{\alpha,1}(\R^n)\subset S^{\beta,1}(\R^n)$ with continuous embeddings.
\end{corollary}

\section{Fractional Caccioppoli sets} 
\label{sec:frac_Caccioppoli_sets}

\subsection{Definition of fractional Caccioppoli sets and the Gauss--Green formula}

As in the classical case (see~\cite{AFP00}*{Definition~3.3.5} for instance), we start with the following definition.

\begin{definition}[Fractional Caccioppoli set]\label{def:frac_Caccioppoli_set}
Let $\alpha\in(0,1)$ and let $E\subset\R^n$ be a measurable set. For any open set $\Omega\subset\R^n$, the \emph{fractional Caccioppoli $\alpha$-perimeter in $\Omega$} is the \emph{fractional variation} of $\chi_E$ in $\Omega$, i.e.\ 
\begin{equation*}
|D^\alpha\chi_E|(\Omega)=\sup\set*{\int_E\div^\alpha\phi\,dx : \phi\in C^\infty_c(\Omega;\R^n),\ \|\phi\|_{L^\infty(\Omega;\R^n)}\le1}.
\end{equation*}
We say that $E$ is a set with \emph{finite fractional Caccioppoli $\alpha$-perimeter in $\Omega$} if $|D^\alpha\chi_E|(\Omega)<+\infty$. We say that $E$ is a set with \emph{locally finite fractional Caccioppoli $\alpha$-perimeter in $\Omega$} if $|D^\alpha\chi_E|(U)<+\infty$ for any $U\Subset\Omega$.
\end{definition}

We can now state the following fundamental result relating non-local distributional gradients of characteristic functions of fractional Caccioppoli sets and vector valued Radon measures.

\begin{theorem}[Gauss--Green formula for fractional Caccioppoli sets]\label{result:Gauss-Green}
Let $\alpha\in(0,1)$ and let $\Omega\subset\R^n$ be an open set. A measurable set $E \subset \R^{n}$ is a set with finite fractional Caccioppoli $\alpha$-perimeter in $\Omega$ if and only if $D^{\alpha} \chi_E \in \M(\Omega; \R^{n})$ and 
\begin{equation}\label{eq:Gauss-Green} 
\int_E \div^{\alpha} \phi \, dx = - \int_{\Omega} \varphi \cdot d D^{\alpha} \chi_E 
\end{equation}
for all $\phi \in C^{\infty}_{c}(\Omega; \R^{n})$. In addition, for any open set $U\subset\Omega$ it holds
\begin{equation}\label{eq:Caccioppoli-measure}
|D^{\alpha} \chi_{E}|(U) = \sup\set*{\int_E\,\div^\alpha\phi\ dx : \phi\in C^\infty_c(U;\R^n),\ \|\phi\|_{L^\infty(U;\R^n)}\le1}.
\end{equation}
\end{theorem}

\begin{proof}
The proof is similar to the one of \cref{th:structure_BV_alpha}. If $D^{\alpha} \chi_E \in \M(\Omega; \R^{n})$ and~\eqref{eq:Gauss-Green} holds, then~$E$ has finite fractional Caccioppoli $\alpha$-perimeter in~$\Omega$ by \cref{def:frac_Caccioppoli_set}. 

If~$E$ is a set with finite fractional Caccioppoli $\alpha$-perimeter in $\Omega$, then define the linear functional $L\colon C_{c}^\infty(\Omega; \R^{n})\to\R$ setting
\begin{equation*} 
L(\phi) := - \int_E \div^{\alpha}\phi \, dx
\qquad
\forall\phi \in C^{\infty}_{c}(\Omega; \R^{n}).
\end{equation*}
Note that $L$ is well defined thanks to \cref{prop:frac_div_repr}. Since $E$ has finite fractional Caccioppoli $\alpha$-perimeter in $\Omega$, we have
\begin{equation*}
C(U):=\sup\set*{L(\phi) : \phi\in C^\infty_c(U;\R^n),\ \|\phi\|_{L^\infty(U;\R^n)}\le 1}<+\infty
\end{equation*}
for each open set $U\subset\Omega$, so that
\begin{equation*} 
\left | L(\phi) \right | \le C(U) \|\phi\|_{L^\infty(U;\R^n)}
\qquad
\forall \phi \in C^{\infty}_{c}(U; \R^{n}).
\end{equation*}
Thus, by the density of $C^{\infty}_{c}(\Omega; \R^{n})$ in $C_{c}(\Omega; \R^{n})$, the functional $L$ can be uniquely extended to a continuous linear functional $\tilde{L}\colon C_{c}(\Omega; \R^{n})\to\R$ and the conclusion follows by Riesz's Representation Theorem.
\end{proof}

\subsection{Lower semicontinuity of fractional variation}

As in the classical case, the variation measure of a set with finite fractional Caccioppoli $\alpha$-perimeter is lower semicontinuous with respect to the local convergence in measure. We also achieve a weak convergence result.

\begin{proposition}[Lower semicontinuity of fractional variation measure]
\label{result:frac_Caccioppoli_perimeter_is_lsc}
Let $\alpha\in(0,1)$ and let $\Omega\subset\R^n$ be an open set. If $(E_k)_{k\in\N}$ is a sequence of sets with finite fractional Caccioppoli $\alpha$-perimeter in $\Omega$ and $\chi_{E_k}\to\chi_E$ in $L^1_{\loc}(\R^n)$, then 
\begin{equation} \label{eq:weak_conv}
D^{\alpha} \chi_{E_{k}} \weakto D^{\alpha} \chi_{E} \ \ \text{in} \ \mathcal{M}(\Omega; \R^{n}),
\end{equation}
and
\begin{equation}\label{eq:lsc_frac_Caccioppoli_perim}
|D^\alpha\chi_E|(\Omega)\le\liminf_{k\to+\infty}|D^\alpha\chi_{E_k}|(\Omega).
\end{equation}
\end{proposition}

\begin{proof}
Up to extract a further subsequence, we can assume that $\chi_{E_k}(x)\to\chi_E(x)$ as $k\to+\infty$ for $\Leb{n}$-a.e.\  $x\in\R^n$. Now let $\phi\in C^\infty_c(\Omega;\R^n)$ be such that $\|\phi\|_{L^{\infty}(\Omega; \R^{n})}\le1$. Then $\div^\alpha\phi\in L^1(\R^n)$ by \cref{prop:frac_div_repr} and so, by Lebesgue's Dominated Convergence Theorem, we have
\begin{align*}
\int_E\div^\alpha\phi\ dx 
=\lim_{k\to+\infty}\int_{E_k}\div^\alpha\phi\ dx
=-\lim_{k\to+\infty}\int_\Omega\phi \cdot \ dD^\alpha\chi_{E_k}
\le \liminf_{k\to+\infty} |D^\alpha\chi_{E_k}|(\Omega).
\end{align*}
By \cref{result:Gauss-Green}, we get~\eqref{eq:lsc_frac_Caccioppoli_perim}. The convergence in~\eqref{eq:weak_conv} easily follows.
\end{proof}

\subsection{Fractional isoperimetric inequality}

As a simple application of \cref{thm:GNS_immersion}, we can prove the following fractional isoperimetric inequality.

\begin{theorem}[Fractional isoperimetric inequality] \label{result:isoperimetric_ineq} 
Let $\alpha \in (0, 1)$ and $n \ge 2$. There exists a constant $c_{n, \alpha} > 0$ such that
\begin{equation}\label{eq:insoperimetric_ineq} 
|E|^{\frac{n-\alpha}{n}} \le c_{n, \alpha} |D^{\alpha} \chi_E|(\R^{n}) 
\end{equation}
for any set $E\subset\R^n$ such that $|E|<+\infty$ and $|D^{\alpha} \chi_E|(\R^{n})<+\infty$.
\end{theorem}

\begin{proof}
Since $\chi_E\in BV^\alpha(\R^n)$, the result follows directly by \cref{thm:GNS_immersion}.
\end{proof}

\subsection{Compactness}

As an application of \cref{result:compactness_BV_alpha}, we can prove the following compactness result for sets with finite fractional Caccioppoli $\alpha$-perimeter in~$\R^n$ (see for instance~\cite{M12}*{Theorem~12.26} for the analogous result in the classical case).

\begin{theorem}[Compactness for sets with finite fractional Caccioppoli $\alpha$-perimeter]\label{result:compactness_Caccioppoli}
Let $\alpha\in(0,1)$ and $R>0$. If $(E_k)_{k\in\N}$ is a sequence of sets with finite fractional Caccioppoli $\alpha$-perimeter in~$\R^n$ such that
\begin{equation*}
\sup_{k\in\N}|D^\alpha\chi_{E_k}|(\R^n)<+\infty
\quad\text{and}\quad
E_k\subset B_R\quad \forall k\in\N,
\end{equation*} 
then there exist a subsequence $(E_{k_j})_{j\in\N}$ and a set $E\subset B_R$ with finite fractional Caccioppoli $\alpha$-perimeter in~$\R^n$ such that 
\begin{equation*}
\chi_{E_{k_j}}\to\chi_E\text{ in~$L^1(\R^n)$}
\end{equation*}
as $j\to+\infty$. 
\end{theorem}

\begin{proof}
Since $E_k\subset B_R$ for all $k\in\N$, we clearly have that $(\chi_{E_k})_{k\in\R^n}\subset BV^\alpha(\R^n)$. By \cref{result:compactness_BV_alpha}, there exist a subsequence $(E_{k_j})_{j\in\N}$ and a function $f\in L^1(\R^n)$ such that $\chi_{E_{k_j}}\to f$ in $L^1_{\loc}(\R^n)$ as $j\to+\infty$. Since again $E_{k_j}\subset B_R$ for all $j\in\N$, we have that $\chi_{E_{k_j}}\to f$ in $L^1(\R^n)$ as $j\to+\infty$. Up to extract a further subsequence (which we do not relabel for simplicity), we can assume that $\chi_{E_{k_j}}(x)\to f(x)$ for $\Leb{n}$-a.e.\ $x\in\R^n$ as $j\to+\infty$, so that $f=\chi_E$ for some $E\subset B_R$. By \cref{result:frac_Caccioppoli_perimeter_is_lsc} we conclude that $E$ has finite fractional Caccioppoli $\alpha$-perimeter in~$\R^n$.
\end{proof}

\cref{result:compactness_Caccioppoli} can be applied to prove the following compactness result for sets with locally finite fractional Caccioppoli $\alpha$-perimeter.

\begin{corollary}[Compactness for locally finite fractional Caccioppoli $\alpha$-perimeter sets] \label{result:compactness_Caccioppoli_local}
Let $\alpha \in (0,1)$. If $(E_k)_{k\in\N}$ is a sequence of sets with locally finite fractional Caccioppoli $\alpha$-perimeter in~$\R^n$ such that
\begin{equation}\label{eq:uniform_bound_compactness_loc_Caccioppoli}
\sup_{k\in\N}|D^\alpha\chi_{E_k}|(B_R)<+\infty
\quad \forall R>0,
\end{equation} 
then there exist a subsequence $(E_{k_j})_{j\in\N}$ and a set $E$ with locally finite fractional Caccioppoli $\alpha$-perimeter in~$\R^n$ such that 
\begin{equation*}
\chi_{E_{k_j}}\to\chi_E\text{ in~$L^1_{\loc}(\R^n)$}
\end{equation*}
as $j\to+\infty$. 
\end{corollary}

\begin{proof}
We divide the proof into two steps, essentially following the strategy presented in the proof of~\cite{M12}*{Corollary~12.27}.

\smallskip

\textit{Step~1}. Let $F\subset\R^n$ be a set with locally finite fractional Caccioppoli $\alpha$-perimeter in~$\R^n$. We claim that
\begin{equation}\label{eq:claim_perim_intersection_ball}
|D^\alpha\chi_{F\cap B_R}|(\R^n)\le|D^\alpha\chi_F|(B_R)+3\mu_{n, \alpha} P_\alpha(B_R) 
\quad \forall R>0.
\end{equation}
Indeed, let $R'<R$ and, recalling \cref{remark:density_test_in_frac_Sobolev}, let $(u_k)_{k\in\N}\subset C^\infty_c(\R^n)$ be such that $\supp(u_k)\Subset B_R$ and $0\le u_k\le 1$ for all $k\in\N$ and also $u_k\to\chi_{B_{R'}}$ in $W^{\alpha,1}(\R^n)$ as $k\to+\infty$. If $\phi\in C^\infty_c(\R^n;\R^n)$ with $\|\phi\|_{L^\infty(\R^n;\R^n)}\le1$, then
\begin{equation*}
\begin{split}
\int_F u_k\,\div^\alpha\phi\,dx
&=\int_F \div^\alpha(u_k\phi)\,dx
-\int_F \phi\cdot\nabla^\alpha u_k\,dx
-\int_F \div^\alpha_{\mathrm{NL}}(u_k,\phi)\,dx\\
&\le\int_F \div^\alpha(u_k\phi)\,dx
+3\mu_{n, \alpha} [u_k]_{W^{\alpha,1}(\R^n)}\\
&\le|D^\alpha\chi_F|(B_{R'})
+3\mu_{n, \alpha} [u_k]_{W^{\alpha,1}(\R^n)}\\
&\le|D^\alpha\chi_F|(B_R)
+3\mu_{n, \alpha} [u_k]_{W^{\alpha,1}(\R^n)}
\end{split}
\end{equation*}
by \cref{lem:Leibniz_frac_div}. Passing to the limit as $k\to+\infty$, we conclude that
\begin{equation*}
\int_{F\cap B_{R'}} \div^\alpha\phi\,dx
\le|D^\alpha\chi_F|(B_R)
+3\mu_{n, \alpha} P_\alpha(B_{R'})
\end{equation*}
and thus
\begin{equation*}
|D^\alpha\chi_{F\cap B_{R'}}|(\R^n)\le|D^\alpha\chi_F|(B_R)+3\mu_{n, \alpha} P_\alpha(B_R) 
\end{equation*}
by \cref{result:Gauss-Green}. Since $\chi_{F\cap B_{R'}}\to\chi_{F\cap B_R}$ in $L^1(\R^n)$ as $R'\to R$, the claim in~\eqref{eq:claim_perim_intersection_ball} follows by \cref{result:frac_Caccioppoli_perimeter_is_lsc}.

\smallskip

\textit{Step~2}. By~\eqref{eq:uniform_bound_compactness_loc_Caccioppoli} and~\eqref{eq:claim_perim_intersection_ball}, we can apply \cref{result:compactness_Caccioppoli} to $(E_k\cap B_j)_{k\in\N}$ for each fixed $j\in\N$. By a standard diagonal argument, we find a subsequence $(E_{k_h})_{h\in\N}$ and a sequence $(F_j)_{j\in\N}$ of sets with finite fractional Caccioppoli $\alpha$-perimeter such that $\chi_{E_{k_h}\cap B_j}\to\chi_{F_j}$ in $L^1(\R^n)$ as $h\to+\infty$ for each $j\in\N$. Up to null sets, we have $F_j\subset F_{j+1}$, so that $\chi_{E_{k_h}}\to\chi_E$ in $L^1_{\loc}(\R^n)$ with $E:=\bigcup_{j\in\N}F_j$. The conclusion thus follows by \cref{result:frac_Caccioppoli_perimeter_is_lsc}.
\end{proof}

\subsection{Fractional reduced boundary}

Thanks to the scaling property of the fractional divergence, we have 
\begin{equation} \label{scaling_prop_eq} 
D^{\alpha} \chi_{\lambda E} = \lambda^{n - \alpha} (\delta_{\lambda})_{\#}D^{\alpha} \chi_{E}
\quad
\text{on}\ \lambda\Omega, 
\end{equation} 
where $\delta_{\lambda}(x) = \lambda x$ for all $x\in\R^n$ and $\lambda>0$. Indeed, we can compute 
\begin{equation*} 
\int_{\lambda E} \div^{\alpha} \phi \, dx 
= \lambda^{n} \int_E (\div^{\alpha} \varphi)\circ\delta_\lambda\, dx 
= \lambda^{n - \alpha} \int_E \div^{\alpha} (\varphi\circ\delta_\lambda) \, dx 
\end{equation*}
for all $\phi\in C^{\infty}_{c}(\Omega; \R^{n})$. In analogy with the classical case, we are thus led to the following definition.

\begin{definition}[Fractional reduced boundary] \label{def:redb_alpha} 
Let $\alpha\in (0, 1)$ and let $\Omega\subset\R^n$ be an open set. If $E\subset\R^n$ is a set with finite fractional Caccioppoli $\alpha$-perimeter in $\Omega$, then we say that a point $x\in\Omega$ belongs to the \emph{fractional reduced boundary} of $E$ (inside $\Omega$), and we write $x\in\redb^\alpha E$, if 
\begin{equation*}
x\in\supp(D^\alpha \chi_{E})
\qquad\text{and}\qquad
\exists\lim_{r\to0}\frac{D^\alpha\chi_E(B_r(x))}{|D^\alpha\chi_E|(B_r(x))}\in\mathbb{S}^{n-1}.
\end{equation*}
We thus let
\begin{equation*}
\nu_E^\alpha\colon\Omega\cap\redb^\alpha E\to\mathbb{S}^{n-1},
\qquad
\nu_E^\alpha(x):=\lim_{r\to0}\frac{D^\alpha\chi_E(B_r(x))}{|D^\alpha\chi_E|(B_r(x))},
\quad
x\in\Omega\cap\redb^\alpha E,
\end{equation*}
be the (\emph{measure theoretic}) \emph{inner unit fractional normal} to~$E$ (inside $\Omega$).
\end{definition}

As a consequence of \cref{def:redb_alpha} and arguing similarly as in the proof of \cref{result:Lip_test}, if $E\subset\R^n$ is a set with finite fractional Caccioppoli $\alpha$-perimeter in $\Omega$, then the following Gauss--Green formula
\begin{equation}\label{eq:Caccioppoli_Gauss-Green}
\int_E\div^\alpha\phi\, dx
=-\int_{\Omega\cap\redb^\alpha E}\phi\cdot\nu^\alpha_E\ d|D^\alpha\chi_E|,
\end{equation}
holds for any $\varphi \in \Lip_{c}(\Omega; \R^{n})$.

\subsection{Sets of finite fractional perimeter are fractional Caccioppoli sets}

In analogy with the classical case and with the inclusion $W^{\alpha,1}(\R^n)\subset BV^\alpha(\R^n)$, we can show that sets with finite fractional $\alpha$-perimeter have finite fractional Caccioppoli $\alpha$-perimeter. Recall that the \emph{fractional $\alpha$-perimeter} of a set~$E\subset\R$ in an open set $\Omega\subset\R^n$ is defined as
\begin{equation*}
P_{\alpha}(E; \Omega) 
:= \int_{\Omega} \int_{\Omega} \frac{|\chi_{E}(x) - \chi_{E}(y)|}{|x - y|^{n + \alpha}} \, dx \, dy + 2 \int_{\Omega} \int_{\R^{n} \setminus \Omega} \frac{|\chi_{E}(x) - \chi_{E}(y)|}{|x - y|^{n + \alpha}} \, dx \, dy,
\end{equation*}
see~\cite{CF17} for an account on this subject.

\begin{proposition}[Sets of finite fractional perimeter are fractional Caccioppoli sets]\label{result:Sobolev_is_Caccioppoli}
Let $\alpha\in(0,1)$ and let $\Omega\subset\R^n$ be an open set. If $E\subset\R^n$ satisfies $P_\alpha(E;\Omega)<+\infty$, then $E$ is a set with finite fractional Caccioppoli $\alpha$-perimeter in $\Omega$ with
\begin{equation}\label{eq:FFPS_is_Caccioppoli_estim}
|D^{\alpha}\chi_E|(\Omega) \le \mu_{n, \alpha} P_\alpha(E;\Omega)
\end{equation}
and
\begin{equation}\label{eq:FFPS_is_Caccioppoli_Gauss-Green}
\int_E\div^\alpha\phi\, dx = - \int_\Omega\phi\cdot\nabla^\alpha\chi_E\, dx
\end{equation}
for all $\phi\in\Lip_c(\Omega;\R^n)$, so that $D^\alpha\chi_E=\nu^\alpha_E\,|D^\alpha\chi_E|=\nabla^\alpha\chi_E\,\Leb{n}$. Moreover, if $E$ is such that $|E|<+\infty$ and $P(E)<+\infty$, then $\chi_{E} \in W^{\alpha, 1}(\R^{n})$ for any $\alpha \in (0, 1)$, and
\begin{equation}\label{eq:FFPS_is_Caccioppoli_repres_formula} 
\nabla^{\alpha} \chi_{E}(x) = \frac{\mu_{n, \alpha}}{n + \alpha - 1} \int_{\R^n} \frac{\nu_{E}(y)}{|y - x|^{n + \alpha - 1}} \, d |D \chi_{E}|(y)
\end{equation}
for $\Leb{n}$-a.e.\ $x\in\R^n$.
\end{proposition}

\begin{proof}
Note that $\nabla^\alpha\chi_E\in L^1(\Omega;\R^n)$, because
\begin{equation*}
\begin{split}
\int_\Omega|\nabla^\alpha\chi_E|\,dx
&\le\mu_{n, \alpha} \int_\Omega\int_{\R^n}\frac{|\chi_E(y)-\chi_E(x)|}{|y-x|^{n+\alpha}}\,dy\,dx\\
&\le\mu_{n, \alpha} \int_\Omega\int_\Omega\frac{|\chi_E(y)-\chi_E(x)|}{|y-x|^{n+\alpha}}\,dy\,dx
+\mu_{n, \alpha} \int_\Omega\int_{\R^{n} \setminus \Omega}\frac{|\chi_E(y)-\chi_E(x)|}{|y-x|^{n+\alpha}}\,dy\,dx\\
&\le\mu_{n, \alpha} P_\alpha(E;\Omega).
\end{split}
\end{equation*}
Now let $\phi \in \Lip_c(\Omega;\R^{n})$ be fixed. By Lebesgue's Dominated Convergence Theorem, by~\eqref{eq:cancellation_kernel} and by Fubini's Theorem (applied for each fixed $\eps > 0$), we can compute
\begin{align*} 
\int_E \div^{\alpha} \varphi \, dx  
& = \mu_{n, \alpha} \lim_{\eps \to 0} \int_E \int_{\{|x -y| > \eps \}} \frac{(y - x) \cdot \varphi(y) }{|y - x|^{n + \alpha + 1}} \, dy \, dx \\
& = \mu_{n, \alpha} \lim_{\eps \to 0} \int_{\Omega} \int_{\{|x -y| > \eps \}} \varphi(y) \cdot \frac{(y - x)\, \chi_E(x) }{|y - x|^{n + \alpha + 1}} \, dx \, dy\\
& = -\mu_{n, \alpha} \lim_{\eps \to 0} \int_\Omega \int_{\{|x -y| > \eps \}} \varphi(y) \cdot \frac{(y - x) (\chi_E(y) - \chi_E(x)) }{|y - x|^{n + \alpha + 1}} \, dx \, dy\\
&=- \int_\Omega \varphi \cdot \nabla^\alpha\chi_E\, dy.
\end{align*}
Thus~\eqref{eq:FFPS_is_Caccioppoli_estim} and~\eqref{eq:FFPS_is_Caccioppoli_Gauss-Green} follow by \cref{result:Gauss-Green} and \cref{def:redb_alpha}. Finally, \eqref{eq:FFPS_is_Caccioppoli_repres_formula} follows from~\eqref{eq:weak_frac_grad_repr}, since $\chi_E\in BV(\R^n)$.
\end{proof}

\noindent
At the present moment, we do not know if $|D^{\alpha}\chi_E|(\Omega)<+\infty$ implies that $P_\alpha(E;\Omega)<+\infty$.

\begin{remark}[$\redb^\alpha E$ is not $\Leb{n}$-negligible in general]
It is important to notice that, by \cref{result:Sobolev_is_Caccioppoli}, we have 
\begin{equation*}
P_\alpha(E;\Omega)<+\infty
\implies
\Leb{n}(\Omega\cap\redb^\alpha E)>0
\end{equation*}
including even the case $\chi_E\in BV(\R^n)$. This shows a substantial difference between the standard \emph{local} De Giorgi's perimeter measure $|D\chi_E|$ and the \emph{non-local} fractional De Giorgi's perimeter measure $|D^\alpha\chi_E|$: the former is supported on a $\Leb{n}$-negligible set contained in the topological boundary of~$E$, while the latter, in general, can be supported on a set of positive Lebesgue measure and, for this reason, cannot be expected to be contained in the topological boundary of~$E$. 
\end{remark}

\begin{remark}[Fractional reduced boundary and precise representative]
We let
\begin{equation*}
u^{*}(x) := 
\begin{cases} 
\displaystyle \lim_{r \to 0} \frac{1}{|B_{r}(x)|} \int_{B_{r}(x)} u(y) \, dy & \text{if the limit exists and is finite}, \\[5mm]
0 & \text{otherwise}, 
\end{cases}
\end{equation*}
be the \emph{precise representative} of a function $u \in L^{1}_{\loc}(\R^{n}; \R^{m})$. Note that $u^{*}$ is well defined at any Lebesgue point of~$u$.
By \cref{result:Sobolev_is_Caccioppoli}, if $P_{\alpha}(E; \Omega) < +\infty$ then $D^{\alpha} \chi_{E} = \nabla^{\alpha} \chi_{E} \Leb{n}$ with $\nabla^{\alpha} \chi_{E} \in L^{1}(\Omega; \R^{n})$. Therefore the set
\begin{equation*}
\mathcal{R}^\alpha _\Omega E :=\set*{x \in \Omega : |(\nabla^{\alpha} \chi_{E})^{*}(x)| = |\nabla^{\alpha} \chi_{E}|^{*}(x) \neq 0}
\end{equation*}
is such that
\begin{equation} \label{eq:redb_alpha_Lebesgue_inclusion}
\mathcal{R}^\alpha _\Omega E \subset \Omega \cap \redb^{\alpha} E
\end{equation}
and
\begin{equation*}
\nu_{E}^{\alpha}(x) = \frac{(\nabla^{\alpha} \chi_{E})^{*}(x)}{|\nabla^{\alpha} \chi_{E}|^{*}(x)}
\qquad
\text{for all $x\in \mathcal{R}^\alpha _\Omega E$}.
\end{equation*}
\end{remark}

The following simple example shows that the inclusion in~\eqref{eq:redb_alpha_Lebesgue_inclusion} and the inequality in \eqref{eq:FFPS_is_Caccioppoli_estim} can be strict.

\begin{example}\label{example:interval_a_b} 
Let $n = 1$, $\alpha \in (0, 1)$ and $a, b \in \R$, with $a < b$. It is easy to see that $\chi_{(a, b)} \in W^{\alpha, 1}(\R)$. By~\eqref{eq:FFPS_is_Caccioppoli_repres_formula}, for any $x \neq a, b$ we have that
\begin{align*} 
\nabla^{\alpha} \chi_{(a, b)}(x) & = \frac{\mu_{1, \alpha}}{\alpha} \int_{\R} \frac{1}{|x - y|^{\alpha}} \, d \left (\delta_{a} - \delta_{b} \right )(y) \\
& = \frac{2^{\alpha}} {\alpha \sqrt{\pi}} \frac{\Gamma\left ( 1 + \frac{\alpha}{2} \right )}{\Gamma\left ( \frac{1 - \alpha}{2} \right )} \left ( \frac{1}{|x - a|^{\alpha}} - \frac{1}{|x - b|^{\alpha}} \right ).
\end{align*}
We claim that 
\begin{equation}\label{eq:frac_redb_a_b}
\redb^{\alpha} (a, b) = \R \setminus\set*{\frac{a + b}{2}}
\end{equation}
while
\begin{equation}\label{eq:special_frac_bound_a_b}
\mathcal{R}^\alpha_{\R} (a,b) = \R \setminus \set*{ a, \frac{a + b}{2}, b},
\end{equation}
so that inclusion~\eqref{eq:redb_alpha_Lebesgue_inclusion} is strict. Finally, we also claim that
\begin{equation}\label{eq:strict_ineq_one_dim}
\|\nabla^{\alpha} \chi_{(a, b)}\|_{L^{1}(\R)} < \mu_{1, \alpha} P_{\alpha}((a, b)).
\end{equation}
Indeed, notice that
\begin{equation*} 
\nabla^{\alpha} \chi_{(a, b)}(x) \ge 0 
\end{equation*}
if and only if $x \le \frac{a + b}{2}$, so that
\begin{equation*}
\lim_{r \to 0} \frac{\displaystyle\int_{x - r}^{x + r} \nabla^{\alpha} \chi_{(a, b)}(y) \, dy}{\displaystyle\int_{x - r}^{x + r} |\nabla^{\alpha} \chi_{(a, b)}(y)| \, dy}
=
\begin{cases}
1 & \text{if $x < \displaystyle\frac{a + b}{2}$},\\[5mm]
-1 & \text{if $x > \displaystyle\frac{a + b}{2}$}.
\end{cases}
\end{equation*}
If $x = \frac{a + b}{2}$, then
\begin{equation*} 
\int_{\frac{a + b}{2} - r}^{\frac{a + b}{2} + r} \nabla^{\alpha} \chi_{(a, b)}(y) \, dy = 0
\qquad
\forall r>0,
\end{equation*}
and claim~\eqref{eq:frac_redb_a_b} follows. In particular, we have
\begin{equation*} 
\nu_{(a, b)}^{\alpha}(x) = 
\begin{cases} 
1 & \text{if} \ \ x < \displaystyle\frac{a + b}{2}, \\[5mm]
- 1 & \text{if} \ \ x > \displaystyle\frac{a + b}{2}. 
\end{cases}
\end{equation*}
On the other hand, it is clear that 
\begin{equation*}
\lim_{r \to 0} \frac{1}{2r} \int_{a - r}^{a + r} \nabla^{\alpha} \chi_{(a, b)}(y) \, dy = + \infty
\end{equation*}
and
\begin{equation*}
\lim_{r \to 0} \frac{1}{2r} \int_{b - r}^{b + r} \nabla^{\alpha} \chi_{(a, b)}(y) \, dy = - \infty,
\end{equation*}
so that claim~\eqref{eq:special_frac_bound_a_b} follows. To prove~\eqref{eq:strict_ineq_one_dim}, note that
\begin{equation} \label{eq:frac_per_a_b}
P_{\alpha}((a, b)) = \frac{4}{\alpha (1 - \alpha)} (b - a)^{1 - \alpha}
\end{equation}
since $P_{\alpha}((a, b)) = (b - a)^{1 - \alpha} P_{\alpha}((0, 1))$ by the scaling property of the fractional perimeter and 
\begin{align*}
P_{\alpha}((0, 1)) & = 2 \int_{\R \setminus (0, 1)} \int_{0}^{1} \frac{1}{|y - x|^{1 + \alpha}} \, dy \, dx \\
& = \frac{2}{\alpha} \int_{\R \setminus (0, 1)} \left [ \frac{\sgn(x - y)}{|y - x|^\alpha} \right ]_{y=0}^{y=1} \, dx \\
& = \frac{2}{\alpha} \int_{\R \setminus (0, 1)} \frac{\sgn(x - 1)}{|1 - x|^{\alpha}}  - \frac{\sgn(x)}{|x|^{\alpha}}  \, dx \\
& = \frac{2}{\alpha} \int_{1}^{\infty} \frac{1}{(x - 1)^{\alpha}} - \frac{1}{x^{\alpha}} \, dx + \frac{2}{\alpha}\int_{- \infty}^{0} \frac{1}{(- x)^{\alpha}} - \frac{1}{(1 - x)^{\alpha}} \, dx  \\
& = \frac{4}{\alpha} \int_{0}^{\infty} \frac{1}{x^{\alpha}} - \frac{1}{(1 + x)^{\alpha}} \, dx = \frac{4}{\alpha (1 - \alpha)}.
\end{align*}
On the other hand, we have
\begin{equation} \label{eq:L_1_nabla_alpha_chi_a_b}
\|\nabla^{\alpha} \chi_{(a, b)}\|_{L^{1}(\R)} 
= \frac{2^{1 + \alpha} \mu_{1, \alpha}}{\alpha (1 - \alpha)} (b - a)^{1 - \alpha}.
\end{equation}
Indeed, $\|\nabla^{\alpha} \chi_{(a, b)}\|_{L^{1}(\R)} =  (b - a)^{1 - \alpha} \|\nabla^{\alpha} \chi_{(0, 1)}\|_{L^{1}(\R)}$ by~\eqref{scaling_prop_eq} and 
\begin{align*}
\frac{\alpha}{\mu_{1, \alpha}} \|\nabla^{\alpha} \chi_{(0, 1)}\|_{L^{1}(\R)}  & = \int_{\R} \left | \frac{1}{|x|^{\alpha}} - \frac{1}{|x - 1|^{\alpha}} \right | \, dx \\
& = \int_{1}^{\infty} \left | \frac{1}{x^{\alpha}} - \frac{1}{(x - 1)^{\alpha}} \right | \, dx + \int_{0}^{1} \left | \frac{1}{x^{\alpha}} - \frac{1}{(1 - x)^{\alpha}} \right | \, dx  \\
&\quad + \int_{- \infty}^{0} \left | \frac{1}{(-x)^{\alpha}} - \frac{1}{(1 - x)^{\alpha}} \right | \, dx \\
& = \int_{1}^{\infty} \frac{1}{(x - 1)^{\alpha}} - \frac{1}{x^{\alpha}} \, dx + \int_{\frac{1}{2}}^{1}  \frac{1}{(1 - x)^{\alpha}} - \frac{1}{x^{\alpha}} \, dx\\
& \quad + \int_{0}^{\frac{1}{2}} \frac{1}{x^{\alpha}} - \frac{1}{(1 - x)^{\alpha}} \, dx + \int_{- \infty}^{0} \frac{1}{(-x)^{\alpha}} - \frac{1}{(1 - x)^{\alpha}} \, dx \\
& = 2 \int_{0}^{\infty} \frac{1}{x^{\alpha}} - \frac{1}{(1 + x)^{\alpha}} \, dx + 2 \int_{0}^{\frac{1}{2}} \frac{1}{x^{\alpha}} - \frac{1}{(1 - x)^{\alpha}} \, dx \\
& = \frac{2}{1 - \alpha} \left ( 1 + 2^{\alpha - 1} + 2^{\alpha - 1} - 1 \right ) = \frac{2^{1 + \alpha}}{1 - \alpha}.
\end{align*}
Combining~\eqref{eq:frac_per_a_b} and~\eqref{eq:L_1_nabla_alpha_chi_a_b}, we get \eqref{eq:strict_ineq_one_dim}.
\end{example}

Thanks to \cref{example:interval_a_b} above, we know that inequality~\eqref{eq:D_alpha_chi_E_P_alpha_intro} is strict for $E=(a,b)$ with $a,b\in\R$, $a<b$. We conclude this section proving that this fact holds for all sets $E\subset\R$ such that $\chi_E\in W^{\alpha,1}(\R)$.

\begin{proposition}
Let $\alpha\in(0,1)$. If $\chi_E\in W^{\alpha,1}(\R)$, then $|D^\alpha\chi_E|(\R)<\mu_{1,\alpha}P_\alpha(E)$.
\end{proposition}

\begin{proof}
We argue by contradiction. Assume $\chi_E\in W^{\alpha,1}(\R)$ is such that $|D^\alpha\chi_E|(\R)=\mu_{1,\alpha}P_\alpha(E)$. Then
\begin{equation}\label{eq:equality_absurd_strict_ineq}
\int_{\R}\int_{\R} \abs*{f_E(x,y)}\,dy\,dx
=\int_{\R}\abs*{\int_{\R} f_E(x,y)\sgn(y-x)\,dy}\,dx
\end{equation}
where
\begin{equation*}
f_E(x,y):=\frac{\chi_E(y)-\chi_E(x)}{|y-x|^{1+\alpha}} \qquad\forall x,y\in\R,\ x\ne y.
\end{equation*}
From~\eqref{eq:equality_absurd_strict_ineq} we deduce that
\begin{equation*}
\int_{\R} \abs*{f_E(x,y)}\,dy
=\abs*{\int_{\R} f_E(x,y)\sgn(y-x)\,dy}
\end{equation*}
for a.e.\ $x\in\R$. If $x\in E$, then $f_E(x,y)\le0$ for all $y\in\R$, $y\ne x$, and thus
\begin{equation*}
\int_x^{+\infty}|f_E(x,y)|\,dy
+\int_{-\infty}^x|f_E(x,y)|\,dy
=
\abs*{\int_x^{+\infty}|f_E(x,y)|\,dy
-\int_{-\infty}^x |f_E(x,y)|\,dy}
\end{equation*}
for a.e.\ $x\in E$. Squaring both sides and simplifying, we get that
\begin{equation*}
\left(\int_x^{+\infty}|f_E(x,y)|\,dy\right)
\left(\int_{-\infty}^x|f_E(x,y)|\,dy\right)=0,
\end{equation*}
so that either $|E^c\cap(x,+\infty)|=0$ or $|E^c\cap(-\infty,x)|=0$ for a.e.\ $x\in E$, contradicting the fact that $|E|<+\infty$.
\end{proof}

\section{Existence of blow-ups for fractional Caccioppoli sets}
\label{sec:blow-ups}

In this section we prove existence of blow-ups for sets with locally finite fractional Caccioppoli $\alpha$-perimeter. We follow the approach presented in~\cite{EG15}*{Section~5.7}. 

We start with the following technical preliminary result.

\begin{lemma}\label{lem:frac_grad_cutoff_ball} 
Let $\alpha \in (0, 1)$. For all $\eps, r > 0$ and $x \in \R^{n}$ we define
\begin{equation*} 
h_{\eps, r, x}(y) := 
\begin{cases} 1 & \text{if} \ \ 0 \le |y - x| \le r, \\[3mm]
\dfrac{r + \eps - |y - x|}{\eps} & \text{if} \ \ r < |y - x| < r + \eps, \\[4mm] 
0 & \text{if} \ \ |y -x| \ge r + \eps. 
\end{cases} 
\end{equation*}
Then $\nabla^\alpha h_{\eps,r,x}\in L^1(\R^n;\R^n)$ with
\begin{equation} \label{eq:frac_grad_cutoff_ball} 
\nabla^{\alpha} h_{\eps, r, x}(y) 
= \frac{\mu_{n, \alpha}}{\eps (n + \alpha - 1)} \int_{B_{r + \eps}(x) \setminus B_r(x)} \frac{x - z}{|x - z|}|z - y|^{1 - n - \alpha}  \, dz 
\end{equation}
for $\Leb{n}$-a.e. $y \in \R^{n}$.
\end{lemma}

\begin{proof} 
Clearly $h_{\eps, r, x} \in\Lip_c(\R^{n})$ and 
\begin{equation*} 
\nabla h_{\eps, r, x}(y) =  - \frac{1}{\eps} \frac{y - x}{|y - x|}\, \chi_{B_{r + \eps}(x) \setminus B_r(x)}(y). 
\end{equation*}
Therefore by~\eqref{eq:weak_frac_grad_repr} we get 
\begin{equation*} 
\nabla^{\alpha} h_{\eps, r, x}(y) 
= - \frac{1}{\eps}\frac{\mu_{n, \alpha}}{n + \alpha - 1} \int_{\R^{n}} \frac{1}{|z - y|^{n + \alpha - 1}} \frac{z - x}{|z - x|} \chi_{B_{r+\eps}(x) \setminus B_r(x)}(z) \, dz 
\end{equation*}
for $\Leb{n}$-a.e.\ $y \in \R^{n}$. By \cref{result:Sobolev_subset_BV}, we get $\nabla^\alpha h_{\eps,r,x}\in L^1(\R^n;\R^n)$. 
\end{proof}

We now proceed with the following formula for integration by parts on balls, see~\cite{EG15}*{Lemma~5.2} for the analogous result in the classical setting.

\begin{theorem}[Integration by parts on balls] \label{th:int_by_parts_E_ball} 
Let $\alpha \in (0, 1)$. If $E\subset\R^n$ is a set with locally finite fractional Caccioppoli $\alpha$-perimeter in~$\R^n$, then 
\begin{equation} \label{eq:int_by_parts_E_ball} 
\int_{E \cap B_r(x)} \div^{\alpha} \phi \, dy 
+ \int_{E} \phi\, \cdot \nabla^{\alpha} \chi_{B_r(x)}\, dy  
+ \int_{E} \div^{\alpha}_{\rm NL} (\chi_{B_r(x)}, \phi) \, dy
=-\int_{B_r(x)} \phi \cdot \, d D^{\alpha} \chi_{E}
\end{equation}
for all $\phi \in\Lip_c(\R^{n}; \R^{n})$, $x \in \redb^{\alpha} E$ and for $\Leb{1}$-a.e.\ $r > 0$.
\end{theorem}

\begin{proof} 
Fix $\eps, r > 0$, $x \in \redb^{\alpha} E$ and $\phi \in \Lip_{c}(\R^{n}; \R^{n})$ and let $h_{\eps, r, x}$ be as in \cref{lem:frac_grad_cutoff_ball}. On the one hand, by~\eqref{eq:Caccioppoli_Gauss-Green} we have 
\begin{equation} \label{eq:IBP_E_1} 
\int_{E} \div^{\alpha}(\phi\, h_{\eps, r, x}) \, dy 
= - \int_{\redb^{\alpha} E} (h_{\eps, r, x}\,\phi) \cdot \, d D^{\alpha} \chi_{E}.
\end{equation}
Since $h_{\eps, r, x}(y) \to \chi_{\closure{B_r(x)}}(y)$ as $\eps \to 0$ for any $y \in \R^{n}$ and $|D^{\alpha} \chi_{E}|(\de B_r(x)) = 0$ for $\Leb{1}$-a.e.\ $r > 0$, we can compute 
\begin{equation*}
\lim_{\eps\to0}\int_{\redb^{\alpha} E} (h_{\eps, r, x}\,\phi) \cdot \, d D^{\alpha} \chi_{E}
=\int_{B_r(x)} \phi \cdot \, d D^{\alpha} \chi_{E}.
\end{equation*}
On the other hand, by \cref{result:Sobolev_frac_enough} and \cref{lem:Leibniz_frac_div}, we have
\begin{equation} \label{eq:Leibniz_rule_h_phi} 
\div^{\alpha}(\phi \, h_{\eps, r, x}) 
= h_{\eps, r, x}\, \div^{\alpha}\phi 
+ \phi \cdot \nabla^{\alpha} h_{\eps, r, x} 
+ \div_{\rm NL}^{\alpha}(h_{\eps, r, x}, \phi). 
\end{equation}
We deal with each term of the right-hand side of~\eqref{eq:Leibniz_rule_h_phi} separately. For the first term, since $0\le h_{\eps, r, x}\le\chi_{B_{r+1}(x)}$ for all $\eps\in(0,1)$ and $h_{\eps, r, x} \to \chi_{B_r(x)}$ in $L^{1}(\R^{n})$ as $\eps \to 0$, by \cref{prop:frac_div_repr} and Lebesgue's Dominated Convergence Theorem we can compute
\begin{equation} \label{eq:IBP_conv_1} 
\lim_{\eps\to0}\int_{E} h_{\eps, r, x} \,\div^{\alpha}\phi \, dy 
=\int_{E \cap B_r(x)} \div^{\alpha}\phi \, dy. 
\end{equation}
For the second term, by~\eqref{eq:frac_grad_cutoff_ball} we have
\begin{equation*} 
\int_{E} \phi(y) \cdot \nabla^{\alpha} h_{\eps, r, x}(y) \, dy 
= \frac{\mu_{n, \alpha}}{\eps (n + \alpha - 1)} \int_{E} \phi(y) \cdot \int_{B_{r+\eps}(x) \setminus B_r(x)} \frac{x - z}{|x - z|} |z - y|^{1 - n - \alpha}\, dz \, dy. 
\end{equation*}
By Fubini's Theorem, we can compute
\begin{align*} 
\int_E & \phi(y)\cdot \int_{B_{r+\eps}(x) \setminus B_r(x)} \frac{x - z}{|x - z|} |z - y|^{1 - n - \alpha}\, dz \, dy \\
& = \int_{B_{r+\eps}(x) \setminus B_r(x)}\frac{x - z}{|x - z|}\cdot \int_E \phi(y)\,|z - y|^{1 - n - \alpha} \, dy \, dz \\
& = \int_{r}^{r + \eps} \int_{\partial B_{\rho}(x)} \frac{x - z}{|x - z|}\cdot \int_E \phi(y)\,|z - y|^{1 - n - \alpha} \, dy \, d \Haus{n - 1}(z) \, d \rho.
\end{align*}
By Lebesgue's Differentiation Theorem, we have
\begin{equation*}
\begin{split}
\lim_{\eps\to0}\frac{1}{\eps}\int_E & \phi(y)\cdot \int_{B_{r+\eps}(x) \setminus B_r(x)} \frac{x - z}{|x - z|}\, |z - y|^{1 - n - \alpha}\, dz \, dy\\
&=\lim_{\eps\to0}\frac{1}{\eps}\int_{r}^{r + \eps} \int_{\partial B_{\rho}(x)} \frac{x - z}{|x - z|}\cdot \int_E \phi(y)\,|z - y|^{1 - n - \alpha} \, dy \, d \Haus{n - 1}(z) \, d \rho\\
&=\int_{\partial B_r(x)} \frac{x - z}{|x - z|}\cdot \int_E \phi(y)\,|z - y|^{1 - n - \alpha} \, dy \, d \Haus{n - 1}(z)\\
&=\int_E \phi(y)\cdot\int_{\partial B_r(x)} \frac{x - z}{|x - z|}\,|z - y|^{1 - n - \alpha}\, d \Haus{n - 1}(z)\, dy\\
&=\int_E \phi(y)\cdot\int_{\R^n} |z - y|^{1 - n - \alpha}\, dD\chi_{B_r(x)}(z)\, dy
\end{split}
\end{equation*}
for $\Leb{1}$-a.e.\ $r > 0$. Therefore, by~\eqref{eq:weak_frac_grad_repr}, we get that 
\begin{equation} \label{eq:IBP_conv_3} 
\begin{split}
\lim_{\eps\to0}\int_{E} & \phi \cdot \nabla^{\alpha}  h_{\eps, r, x} \, dy \\
& = \frac{\mu_{n, \alpha}}{n + \alpha - 1} \int_E\phi(y) \cdot\int_{\R^{n}} |z - y|^{1 - n - \alpha} \, d D \chi_{B_{r}(x)}(z)  \, dy \\
& = \int_E \,\phi \cdot \nabla^{\alpha} \chi_{B_{r}(x)} \, dy  
\end{split}
\end{equation}
for $\Leb{1}$-a.e.\ $r > 0$. 
Finally, for the third term, note that
\begin{equation*}
\abs*{\frac{(z - y) \cdot (\phi(z) - \phi(y)) (h_{\eps, r, x}(z) - h_{\eps, r, x}(y))}{|z - y|^{n + \alpha + 1}}} 
\le 2 \frac{|\phi(z) - \phi(y)|}{|z - y|^{n + \alpha}} \in L^{1}_{z}(\R^{n})
\end{equation*} 
for all $y\in\R^n$, so that
\begin{equation*}
\lim_{\eps\to0}\div_{\rm NL}^{\alpha}(h_{\eps, r, x}, \phi)(y) 
= \div^{\alpha}_{\rm NL}(\chi_{B_r(x)}, \phi)(y) 
\end{equation*}
for $\Leb{n}$-a.e.\ $y\in\R^n$ by Lebesgue's Dominated Convergence Theorem. Since
\begin{equation*}
\abs*{\div_{\rm NL}^{\alpha}(h_{\eps, r, x}, \phi)(y)} 
\le 2 \int_{\R^{n}}\frac{|\phi(z) - \phi(y)|}{|z - y|^{n + \alpha}}\,dz \in L^{1}_y(\R^{n}),
\end{equation*}
again by Lebesgue's Dominated Convergence Theorem we can compute
\begin{equation} \label{eq:IBP_conv_2} 
\lim_{\eps\to0}\int_{E} \div_{\rm NL}^{\alpha}(h_{\eps, r, x}, \phi) \, dy 
=\int_{E} \div^{\alpha}_{\rm NL}(\chi_{B_{r}(x)}, \phi) \, dy. 
\end{equation}
Combining~\eqref{eq:IBP_E_1}, \eqref{eq:Leibniz_rule_h_phi}, \eqref{eq:IBP_conv_1}, \eqref{eq:IBP_conv_3} and~\eqref{eq:IBP_conv_2}, we obtain~\eqref{eq:int_by_parts_E_ball}.
\end{proof}

We can now deduce the following decay estimates for the fractional De Giorgi's perimeter measure, see~\cite{EG15}*{Lemma~5.3} for the analogous result in the classical setting.

\begin{theorem}[Decay estimates]\label{th:decay_D_alpha_E_B} 
Let $\alpha \in (0, 1)$. There exist $A_{n, \alpha}, B_{n, \alpha} > 0$ with the following property. Let $E\subset\R^n$ be a set with locally finite fractional Caccioppoli $\alpha$-perimeter in~$\R^n$. For any $x \in \redb^{\alpha} E$, there exists $r_x > 0$ such that
\begin{equation} \label{eq:decay_D_alpha_E_B_1} 
|D^{\alpha} \chi_{E}|(B_r(x)) \le A_{n, \alpha} r^{n - \alpha} 
\end{equation}
and
\begin{equation} \label{eq:decay_D_alpha_E_B_2} 
|D^{\alpha} \chi_{E \cap B_r(x)}|(\R^{n}) \le B_{n, \alpha} r^{n - \alpha} 
\end{equation}
for all $r \in (0, r_x)$.
\end{theorem}

\begin{proof} 
We divide the proof in two steps, dealing with the two estimates separately.

\smallskip

\textit{Step~1: proof of~\eqref{eq:decay_D_alpha_E_B_1}}. 
Fix $x \in \redb^{\alpha} E$ and choose $\phi \in \Lip_{c}(\R^{n}; \R^{n})$ such that $\phi \equiv \nu^{\alpha}_{E}(x)$ in $B_1(x)$ and $\|\phi\|_{L^{\infty}(\R^{n}; \R^{n})} \le 1$. On the one hand, by \cref{def:redb_alpha}, there exists $r_x\in(0,1)$ such that
\begin{equation} \label{eq:decay_estimate_1} 
\int_{B_r(x)} \phi\cdot\, d D^{\alpha} \chi_{E} \ge \frac{1}{2} |D^{\alpha} \chi_{E}|(B_r(x)) 
\end{equation}
for all $r\in(0,r_x)$. On the other hand, by~\eqref{eq:int_by_parts_E_ball} we have
\begin{equation}\label{eq:decay_estimate_1.5}
\begin{split} 
\int_{B_r(x)} \phi\cdot\, d D^{\alpha} \chi_{E} 
&\le \abs*{\int_{E \cap B_{r}(x)} \div^{\alpha} \phi \, dy} 
+\abs*{ \int_{E} \phi\cdot \, d D^{\alpha} \chi_{B_r(x)} }\\
&\quad+\abs*{ \int_{E} \div^{\alpha}_{\rm NL} (\chi_{B_r(x)}, \phi) \, dy } 
\end{split}
\end{equation}
for $\Leb{1}$-a.e.\ $r\in(0,r_x)$. We now estimate the three terms in the right-hand side separately. For the first one, since $\phi(y) \equiv \nu_{E}^{\alpha}(x)$ in $B_r(x)$, we can estimate
\begin{align*} 
\left |  \int_{E \cap B_r(x)} \div^{\alpha} \phi(y) \, dy \right | 
& \le \mu_{n, \alpha} \int_{E \cap B_r(x)} \int_{\R^{n}} \frac{|\phi(z) - \phi(y)|}{|z - y|^{n + \alpha}} \, dz \, d y \\
& = \mu_{n, \alpha} \int_{E \cap B_r(x)} \int_{\R^{n} \setminus B_r(x)} \frac{|\phi(z) - \nu_{E}^{\alpha}(x)|}{|z - y|^{n + \alpha}} \, dz \, d y \\
& \le 2 \mu_{n, \alpha} \int_{B_r(x)} \int_{\R^{n} \setminus B_r(x)} \frac{1}{|z - y|^{n + \alpha}} \, dz \, d y \\
& = 2 \mu_{n, \alpha} P_{\alpha}(B_r(x))
\end{align*}
so that
\begin{equation} \label{eq:decay_estimate_2} 
\left |  \int_{E \cap B_r(x)} \div^{\alpha} \phi(y) \, dy \right | 
\le 2 \mu_{n, \alpha} P_{\alpha}(B_1)\, r^{n - \alpha}. 
\end{equation}
For the second term, by \cref{result:Sobolev_is_Caccioppoli} we can estimate
\begin{equation} \label{eq:decay_estimate_3} 
\left | \int_{E} \phi \cdot \, d D^{\alpha} \chi_{B_{r}(x)}  \right | 
\le |D^{\alpha} \chi_{B_r(x)}|(\R^{n})  
\le \mu_{n, \alpha} P_{\alpha}(B_r(x))
= \mu_{n, \alpha} P_{\alpha}(B_1)\, r^{n - \alpha}. 
\end{equation}
Finally, by \cref{lem:Leibniz_frac_div}, we can estimate
\begin{equation*}
\begin{split}
\abs*{\int_{E} \div^{\alpha}_{\rm NL}(\chi_{B_r(x)}, \phi) \, dy}
&\le\| \div^{\alpha}_{\rm NL}(\chi_{B_r(x)}, \phi)\|_{L^{1}(\R^{n})} \\
&\le 2 \mu_{n, \alpha} [\chi_{B_r(x)}]_{W^{\alpha, 1}(\R^{n})}\\
&=2 \mu_{n, \alpha}P_{\alpha}(B_r(x)) 
\end{split}
\end{equation*}
so that
\begin{equation} \label{eq:decay_estimate_4} 
\abs*{\int_{E} \div^{\alpha}_{\rm NL}(\chi_{B_r(x)}, \phi) \, dy}
\le 2 \mu_{n, \alpha} P_{\alpha}(B_1)\, r^{n - \alpha}.  
\end{equation}
Combining~\eqref{eq:decay_estimate_1}, \eqref{eq:decay_estimate_1.5}, \eqref{eq:decay_estimate_2}, \eqref{eq:decay_estimate_3} and \eqref{eq:decay_estimate_4}, we conclude that
\begin{equation} \label{eq:decay_estimate_5} 
|D^{\alpha} \chi_{E}|(B_r(x)) 
\le 10 \mu_{n, \alpha} P_{\alpha}(B_1) \, r^{n - \alpha} 
\end{equation}
for $\Leb{1}$-a.e.\ $r\in(0,r_x)$. Hence~\eqref{eq:decay_D_alpha_E_B_1} follows with $A_{n, \alpha} = 10 \mu_{n, \alpha} P_{\alpha}(B_1)$ for all $r\in(0,r_x)$ by a simple continuity argument.

\smallskip

\textit{Step~2: proof of~\eqref{eq:decay_D_alpha_E_B_2}}.
Fix $x \in \redb^{\alpha} E$ and $\phi \in \Lip_{c}(\R^{n}; \R^{n})$ with $\|\phi\|_{L^{\infty}(\R^{n}; \R^{n})} \le 1$. Again by~\eqref{eq:int_by_parts_E_ball} we can estimate
\begin{equation*} 
\abs*{\int_{E \cap B_r(x)} \div^{\alpha} \phi \, dy} 
\le |D^{\alpha} \chi_{E}|(B_r(x)) + |D^{\alpha} \chi_{B_r(x)}|(\R^n) 
+ \int_{E} |\div^{\alpha}_{\rm NL}(\chi_{B_r(x)}, \phi)| \, dy 
\end{equation*}
for $\Leb{1}$-a.e.\ $r\in(0,r_x)$. Using~\eqref{eq:decay_estimate_3}, \eqref{eq:decay_estimate_4} and~\eqref{eq:decay_estimate_5}, we conclude that
\begin{equation*} 
|D^{\alpha} \chi_{E \cap B_r(x)}|(\R^{n}) \le 13 \mu_{n, \alpha} P_{\alpha}(B_1) \, r^{n - \alpha}
\end{equation*}
for $\Leb{1}$-a.e.\ $r\in(0,r_x)$. Hence~\eqref{eq:decay_D_alpha_E_B_2} follows with $B_{n, \alpha} = 13 \mu_{n, \alpha} P_{\alpha}(B_1)$ for all $r\in(0,r_x)$ by a simple continuity argument. This concludes the proof.
\end{proof}

As an easy consequence of \cref{th:decay_D_alpha_E_B}, we can prove that 
\begin{equation*}
|D^{\alpha} \chi_{E}|\ll\Haus{n-\alpha}\res\redb^\alpha E
\end{equation*}
for any set~$E$ with locally finite fractional Caccioppoli $\alpha$-perimeter in $\R^{n}$.

\begin{corollary}[$|D^{\alpha} \chi_{E}|\ll\Haus{n-\alpha}\res\redb^\alpha E$] \label{cor:abs_continuity} 
Let $\alpha \in (0, 1)$. If $E$ is a set with locally finite fractional Caccioppoli $\alpha$-perimeter in~$\R^{n}$, then
\begin{equation} \label{eq:abs_cont_estimate}
|D^{\alpha} \chi_{E}| \le 2^{n - \alpha} \frac{A_{n, \alpha}}{\omega_{n - \alpha}} \Haus{n - \alpha} \res \redb^{\alpha} E,
\end{equation} 
where $A_{n, \alpha}$ is as in~\eqref{eq:decay_D_alpha_E_B_1}.
\end{corollary}

\begin{proof}
By \eqref{eq:decay_D_alpha_E_B_1}, we have that
\begin{equation*} 
\Theta^{*}_{n - \alpha}(|D^{\alpha} \chi_{E}|, x) := \limsup_{r \to 0} \frac{|D^{\alpha} \chi_{E}|(B_{r}(x))}{\omega_{n - \alpha} r^{n - \alpha}} \le \frac{A_{n, \alpha}}{\omega_{n - \alpha}}
\end{equation*}
for any $x \in \redb^{\alpha} E$.
Therefore, \eqref{eq:abs_cont_estimate} is a simple application of \cite{AFP00}*{Theorem 2.56}.
\end{proof}

For any set~$E$ of locally finite fractional Caccioppoli $\alpha$-perimeter, \cref{cor:abs_continuity} enables us to obtain a lower bound on the Hausdorff dimension of~$\redb^{\alpha}E$.

\begin{proposition} 
Let $\alpha \in (0, 1)$. If~$E$ is a set with locally finite fractional Caccioppoli $\alpha$-perimeter in~$\R^{n}$, then 
\begin{equation} \label{eq:redb_alpha_dim_estimate} 
\dim_{\Haus{}}(\redb^{\alpha} E) \ge n - \alpha. 
\end{equation}
\end{proposition}

\begin{proof}
Since $|D^{\alpha} \chi_{E}|(\redb^{\alpha} E) > 0$ by \cref{def:redb_alpha}, by \cref{cor:abs_continuity} we conclude that $\Haus{n - \alpha}(\redb^{\alpha} E) > 0$, proving~\eqref{eq:redb_alpha_dim_estimate}.
\end{proof}

As another interesting consequence of \cref{cor:abs_continuity}, we are able to prove that assumption~\eqref{eq:coarea_int_finite} in~\cref{th:coarea_inequality} cannot be dropped.

\begin{corollary}[No coarea formula in $BV^{\alpha}(\R)$]
Let $\alpha\in(0,1)$. There exist $f\in BV^\alpha(\R^n)$ such that
\begin{equation}\label{eq:coarea_explosion}
\int_{\R} |D^{\alpha} \chi_{\{ f > t \}}|(\R^n) \, dt=+\infty.
\end{equation}
\end{corollary}

\begin{proof}
Let $E\subset \R^n$ be such that $\chi_E\in BV(\R^n)$ and consider $f:=(-\Delta)^{\frac{1-\alpha}{2}}\chi_E$. By \cref{result:correspondence_BV_alpha_bv}, we know that $f\in BV^\alpha(\R^n)$ with $|D^\alpha f|=|D\chi_E|=\Haus{n-1}\res\redb E$. If
\begin{equation*}
\int_{\R} |D^{\alpha} \chi_{\{ f > t \}}|(\R^n) \, dt<+\infty
\end{equation*}
then
\begin{equation*}
|D^\alpha f|\le \int_{\R} |D^{\alpha} \chi_{\{ f > t \}}| \, dt
\end{equation*}
by \cref{th:coarea_inequality}. Thus $|D^{\alpha} f| \ll \Haus{n - \alpha}$ by  \cref{cor:abs_continuity}, so that $\Haus{n-1}(\redb E)=0$, which is clearly absurd.   
\end{proof}

\begin{remark}
If $f\in W^{\alpha,1}(\R^n)$, then 
\begin{equation*}
\int_{\R}|D^\alpha\chi_{\set{f>t}}|(\R^n)\,dt
\le\mu_{n,\alpha}\int_{\R}P_\alpha(\set{f>t})\,dt
=\mu_{n,\alpha}[f]_{W^{\alpha,1}(\R^n)}<+\infty
\end{equation*}
by \cref{result:Sobolev_is_Caccioppoli} and Tonelli's Theorem, so that~\eqref{eq:coarea_explosion} does not hold for all $f\in BV^\alpha(\R^n)$. We do not know if~\eqref{eq:tot_var_coarea_inequality} is an equality for some functions $f\in BV^\alpha(\R^n)$.
\end{remark}

We can now prove the existence of blow-ups for sets with locally finite fractional Caccioppoli $\alpha$-perimeter in~$\R^n$, see~\cite{EG15}*{Theorem~5.13} for the analogous result in the classical setting. Here and in the following, given a set~$E$ with locally finite fractional Caccioppoli $\alpha$-perimeter and $x\in\redb^\alpha E$, we let $\tang(E,x)$ be the set of all \emph{tangent sets of~$E$ at~$x$}, i.e.\ the set of all limit points in $L^1_{\loc}(\R^n)$-topology of the family $\set*{\frac{E - x}{r} : r>0}$ as $r\to0$.

\begin{theorem}[Existence of blow-up] \label{th:blow_up} 
Let $\alpha \in (0, 1)$. Let $E$ be a set with locally finite fractional Caccioppoli $\alpha$-perimeter in~$\R^n$. For any $x \in \redb^{\alpha} E$ we have $\tang(E,x)\ne\varnothing$.
\end{theorem}

\begin{proof} 
Fix $x \in \redb^{\alpha} E$. Up to a translation, we can assume $x = 0$. We set $E_{r} :=E/r= \set*{ y \in \R^{n} : r y \in E }$ for all $r>0$. We divide the proof in two steps.

\smallskip

\textit{Step~1}. For each $p\in\N$, we define $D_{r}^p := E_{r} \cap B_p$. By the $\alpha$-homogeneity of $\div^\alpha$, we have 
\begin{align*} 
\int_{D_{r}^p} \div^{\alpha} \phi \, dy 
=  r^{-n}\int_{E \cap B_{rp}} (\div^{\alpha} \phi)( r^{-1}z) \,dz
= r^{\alpha - n}\int_{E \cap B_{rp}} \div^{\alpha} (\phi(r^{-1}\cdot))  \, dz
\end{align*}
for all $\phi \in C^{\infty}_{c}(\R^{n}; \R^{n})$. By~\eqref{eq:decay_D_alpha_E_B_2}, we thus get
\begin{equation*}
|D^\alpha\chi_{D_r^p}|(\R^n)
=r^{\alpha - n}|D^\alpha\chi_{E\cap B_{rp}}|(\R^n)
\le B_{n,\alpha}p^{n-\alpha}
\end{equation*}
for all $r>0$ such that $rp<r_0$. Hence, for each fixed $p\in\N$, we have
\begin{equation*}
\sup_{r<r_0/p}|D^{\alpha} \chi_{D_{r}^p}|(\R^{n}) \le B_{n,\alpha}p^{n-\alpha}. 
\end{equation*}

\smallskip

\textit{Step~2}. Let $(r_k)_{k\in\N}$ be such that $r_k\to0$ as $k\to+\infty$ and let $E_k:=E_{r_k}$ and $D_k^p:=D_{r_k}^p$ for simplicity. By \textit{Step~1}, for each $p\in\N$ we know that 
\begin{equation*}
\sup_{r_k<r_0/p}|D^{\alpha} \chi_{D_k^p}|(\R^{n}) \le B_{n,\alpha}p^{n-\alpha}
\quad\text{and}\quad
D_k^p\subset B_p
\quad \forall k\in\N.
\end{equation*}
Thanks to \cref{result:compactness_Caccioppoli}, by a standard diagonal argument we find a subsequence $(D_{k_j}^p)_{j\in\N}$ and a sequence $(F_p)_{p\in\N}$ of sets with finite fractional Caccioppoli $\alpha$-perimeter such that $\chi_{D_{k_j}^p}\to\chi_{F_p}$ in $L^1(\R^n)$ as $j\to+\infty$ for each $p\in\N$. Up to null sets, we have $F_p\subset F_{p+1}$, so that $\chi_{E_{k_j}}\to\chi_F$ in $L^1_{\loc}(\R^n)$, where $F:=\bigcup_{p\in\N} F_p$. We thus conclude that $F\in\tang(E,x)$. 
\end{proof}

We now give a characterisation of the blow-ups of sets with locally finite fractional Caccioppoli $\alpha$-perimeter in~$\R^n$, see Claim~\#1 in the proof of~\cite{EG15}*{Theorem~5.13} for the result in the classical setting. 

\begin{proposition}[Characterisation of blow-ups] \label{prop:char_blow_up} 
Let $\alpha \in (0, 1)$. Let $E$ be a set with locally finite fractional Caccioppoli $\alpha$-perimeter in~$\R^n$ and let $x \in \redb^{\alpha} E$. If $F \in \tang(E,x)$, then $F$~is a set of locally finite fractional Caccioppoli $\alpha$-perimeter such that $\nu^{\alpha}_{F}(y)=\nu^\alpha_E(x)$ for $|D^\alpha\chi_F|$-a.e.\ $y\in\redb^\alpha F$.
\end{proposition}

\begin{proof} 
As in the proof of \cref{th:blow_up}, we assume $x = 0$ and we set $E_{r} = E/r$. By \cref{th:blow_up}, there exists $(r_k)_{k\in\N}$ such that $r_{k} \to 0$ as $k\to+\infty$ and $\chi_{E_{r_{k}}} \to \chi_{F}$ in $L^{1}_{\loc}(\R^{n})$. By \cref{result:frac_var_meas_is_lsc}, it is clear that $F$ has locally finite fractional Caccioppoli $\alpha$-perimeter in~$\R^n$. By \eqref{eq:weak_conv}, we get
\begin{equation*}
D^{\alpha} \chi_{E_{r_{k}}} \weakto D^{\alpha} \chi_{F}
\quad
\text{in $\mathcal{M}_{\rm loc}(\R^{n}; \R^{n})$}
\end{equation*}
as $k\to+\infty$. Thus, for $\Leb{1}$-a.e.\ $L > 0$, we have
\begin{equation} \label{eq:weak_conv_blow_up_ball}
D^{\alpha} \chi_{E_{r_{k}}}(B_{L}) \to D^{\alpha} \chi_{F}(B_{L})
\quad
\text{as $k\to+\infty$}.
\end{equation}
Since 
\begin{equation*}
D^{\alpha} \chi_{E_{r}} 
= r^{\alpha - n} (\delta_{\frac{1}{r}})_{\#} D^{\alpha} \chi_{E}
\qquad
\forall r>0,
\end{equation*}
we have that 
\begin{equation*}
|D^{\alpha} \chi_{E_{r_{k}}}| (B_{L}) 
= r_{k}^{\alpha - n} |D^{\alpha} \chi_{E}|(B_{r_{k} L})
\end{equation*}
and
\begin{equation*}
D^{\alpha} \chi_{E_{r_{k}}}(B_{L}) 
= r_{k}^{\alpha - n} D^{\alpha} \chi_{E}(B_{r_{k} L}).
\end{equation*}
Since $0\in\redb^\alpha E$, we thus get
\begin{equation}\label{eq:asymptotic_blow_up} 
\lim_{k\to+\infty}\frac{D^{\alpha} \chi_{E_{r_{k}}} (B_{L})}{|D^{\alpha} \chi_{E_{r_{k}}}| (B_{L})} 
=\lim_{k\to+\infty}\frac{D^{\alpha} \chi_{E}(B_{r_{k} L})}{|D^{\alpha} \chi_{E}|(B_{r_{k} L})} 
=\nu_{E}^{\alpha}(0).
\end{equation}
Therefore, by \cref{result:frac_var_meas_is_lsc},  \eqref{eq:weak_conv_blow_up_ball} and~\eqref{eq:asymptotic_blow_up}, we obtain that
\begin{align*} 
|D^{\alpha} \chi_{F}|(B_{L}) 
& \le \liminf_{k \to + \infty} |D^{\alpha} \chi_{E_{r_{k}}}|(B_{L}) \\
& = \lim_{k \to + \infty} \int_{B_{L}} \nu_{E}^{\alpha}(0) \cdot d D^{\alpha} \chi_{E_{r_{k}}} \\
& = \int_{B_{L}} \nu_{E}^{\alpha}(0) \cdot d D^{\alpha} \chi_{F} \\
& = \int_{B_{L}} \nu_{E}^{\alpha}(0) \cdot \nu_{F}^{\alpha} \ d |D^{\alpha} \chi_{F}| \\
&\le |D^{\alpha} \chi_{F}|(B_{L})
\end{align*}
for $\Leb{1}$-a.e.\ $L > 0$. We thus get that $\nu_{F}^{\alpha}(y) = \nu_{E}^{\alpha}(0)$ for $|D^\alpha\chi_F|$-a.e.\ $y \in B_{L} \cap \redb^{\alpha} F$ and $\Leb{1}$-a.e.\ $L > 0$, so that the conclusion follows.
\end{proof}

\appendix

\section{\texorpdfstring{$C^\infty_c(\R^n)$}{Cˆinfty-c(Rˆn)} is dense in \texorpdfstring{$W^{\alpha,1}(\R^n)$}{Wˆ{alpha,1}(Rˆn)}}

The density of $C^\infty_c(\R^n)$ in $W^{\alpha,p}(\R^n)$ for all $\alpha\in(0,1)$ and $1\le p<+\infty$ is stated without proof in~\cite{DiNPV12}*{Theorem~2.4}. For the proof of this result, the authors in~\cite{DiNPV12} refer to~\cite{A75}*{Theorem~7.38}, where unfortunately the case $p=1$ is not explicitly proved. This result is also stated in~\cite{DD12}*{Proposition~4.27}, but the proof is given for the case $n=1$. For the sake of clarity, we spend some words on the proof of the density of $C^\infty_c(\R^n)$ in $W^{\alpha,1}(\R^n)$ for all $\alpha\in(0,1)$.

\begin{theorem}[$C^\infty_c(\R^n)$ is dense in $W^{\alpha,1}(\R^n)$]\label{remark:density_test_in_frac_Sobolev}
Let $\alpha\in(0,1)$. If $f\in W^{\alpha,1}(\R^n)$, then there exists $(f_k)_{k\in\N}\subset C^\infty_c(\R^n)$ such that $f_k\to f$ in $W^{\alpha,1}(\R^n)$ as $k\to+\infty$.
\end{theorem}

\begin{proof}
The proof of the density of $C^\infty(\R^n)\cap W^{\alpha,1}(\R^n)$ in $W^{\alpha,1}(\R^n)$ via a standard convolution argument is given in full details in~\cite{L09}*{Proposition~14.5} (actually, in the more general setting of \emph{Besov spaces}, see~\cite{L09}*{Section~14.8} for the relation with fractional Sobolev spaces). Thus, to conclude, we just need to show the density of $C^\infty_c(\R^n)$ in $C^\infty(\R^n)\cap W^{\alpha,1}(\R^n)$. To this aim, let $f\in C^\infty(\R^n)\cap W^{\alpha,1}(\R^n)$ be fixed. For all $R>0$, consider a cut-off function $\eta_R\in C^\infty_c(\R^n)$ defined as in \eqref{eq:def_cut-off}.
Then $f\eta_R\in C^\infty_c(\R^n)$ and the conclusion clearly follows if we show that 
\begin{equation}\label{eq:density_proof_cut-off}
\lim_{R\to+\infty}[f(1-\eta_R)]_{W^{\alpha,1}(\R^n)}=0.
\end{equation}
Indeed, we have
\begin{equation*}
\begin{split}
[f(1-\eta_R)]_{W^{\alpha,1}(\R^n)}
&\le\int_{\R^n}\int_{\R^n}\frac{|f(y)-f(x)|}{|y-x|^{n+\alpha}}\,(1-\eta_R(y))\,dy\,dx\\
&\quad+\int_{\R^n}|f(x)|\int_{\R^n}\frac{|\eta_R(y)-\eta_R(x)|}{|y-x|^{n+\alpha}}\,dy\,dx.
\end{split}
\end{equation*}
For the first term in the right-hand side, we easily get that
\begin{equation*}
\lim_{R\to+\infty}\int_{\R^n}\int_{\R^n}\frac{|f(y)-f(x)|}{|y-x|^{n+\alpha}}\,(1-\eta_R(y))\,dy\,dx=0
\end{equation*}
by Lebesgue's Dominated Convergence Theorem, since $[f]_{W^{\alpha,1}(\R^n)}<+\infty$. For the second term in the right-hand side, as in~\eqref{eq:Lip_nabla_estim_1} and~\eqref{eq:Lip_nabla_estim_2} we can estimate
\begin{equation*}
\int_{\R^n}\frac{|\eta_R(y)-\eta_R(x)|}{|y-x|^{n+\alpha}}\,dy
\le\Lip(\eta_R)\int_0^1 r^{-\alpha}\,dr
+2\int_1^{+\infty} r^{-(1+\alpha)}\,dr
\end{equation*} 
for all $x\in\R^n$, so that 
\begin{equation*}
\lim_{R\to+\infty}\int_{\R^n}|f(x)|\int_{\R^n}\frac{|\eta_R(y)-\eta_R(x)|}{|y-x|^{n+\alpha}}\,dy\,dx=0
\end{equation*}
again by Lebesgue's Dominated Convergence Theorem, since $f\in L^1(\R^n)$. Thus~\eqref{eq:density_proof_cut-off} follows and the proof is complete.
\end{proof}


\end{document}